\documentclass{article}

\usepackage{fullpage}
\usepackage{layout}
\usepackage{tkz-euclide}

\usepackage[font=small,labelfont=bf]{caption}

\usepackage{tikz}
 \usetikzlibrary{cd}
\usetikzlibrary{decorations.markings}

\usepackage{graphicx,caption}

\usepackage{amssymb,amsmath,amsthm,amscd,epsf,latexsym,verbatim}
\input epsf.tex
\usepackage{bm}
\usepackage[utf8]{inputenc}
\usepackage[T1]{fontenc}
\usepackage{graphicx, enumerate}
\usepackage{palatino, url, multicol}
\usepackage{subfig}
\usepackage{geometry}                
\usepackage[linktoc=all]{hyperref}
\usepackage[all]{xy}
\usepackage{tikz}
\usetikzlibrary{shapes,snakes}
\usetikzlibrary{arrows,shapes,positioning}
\usetikzlibrary{decorations.markings}
\tikzstyle arrowstyle=[scale=1]
\tikzstyle directed=[postaction={decorate,decoration={markings,
    mark=at position .22 with {\arrow[arrowstyle]{stealth}}}}]
\tikzstyle reverse directed=[postaction={decorate,decoration={markings,
    mark=at position .22 with {\arrowreversed[arrowstyle]{stealth};}}}]
\tikzstyle directedd=[postaction={decorate,decoration={markings,
    mark=at position .5 with {\arrow[arrowstyle]{stealth}}}}]

\theoremstyle{theorem}
\newtheorem{theorem}{Theorem}[section]

\newtheorem{lemma}[theorem]{Lemma}
\newtheorem{proposition}[theorem]{Proposition}
\newtheorem{corollary}[theorem]{Corollary}
\newtheorem*{claim}{Claim}

\theoremstyle{definition}

\newtheorem*{remark}{Remark}
\newtheorem{example}[theorem]{Example}

\theoremstyle{question}

\newtheorem*{billiardtheorem}{Billiard Rigidity Theorem}
\newcommand{\billiardrigidity}{Billiard Rigidity Theorem }
\newtheorem*{supporttheorem}{Current Support Theorem}
\newcommand{\currentsupport}{Current Support Theorem }

\newcommand{\Hyp}{{\rm Hyp}}
\newcommand{\cone}{{\rm cone}}
\newcommand{\PSL}{{\rm PSL}}
\newcommand{\Homeo}{{\rm Homeo}}
\newcommand{\tphi}{{\tilde \varphi}}
\newcommand{\hphi}{{\hat \varphi}}
\newcommand{\CG}{{\mathcal G}}
\newcommand{\tS}{{\widetilde S}}
\newcommand{\tC}{{\widetilde C}}
\newcommand{\CB}{{\mathcal B}}
\newcommand{\tHyp}{\widetilde{\rm Hyp}}
\newcommand{\hS}{{\hat S}}
\newcommand{\hC}{{\hat C}}
\newcommand{\CT}{{\mathcal T}}
\newcommand{\CE}{{\mathcal E}}
\newcommand{\tCE}{{\widetilde{\mathcal E}}}
\newcommand{\Isom}{{\rm Isom}}
\newcommand{\CA}{{\mathcal A}}
\newcommand{\Area}{{\rm Area}}

\newcommand{\ve}[1]{\marginpar{\color{purple}\tiny #1 --VE}}
\newcommand{\cl}[1]{\marginpar{\color{red}\tiny #1 --CL}}

\newcommand{\PSO}{{\rm PSO}}
\newcommand{\Stab}{{\rm Stab}}

\title{Hyperbolic cone metrics and billiards}
\author{Viveka Erlandsson, Christopher J. Leininger, and Chandrika Sadanand}
\date{}                                         

\begin{document}
\maketitle

\newcommand{\Addresses}{{
  \bigskip
  \footnotesize

 Viveka Erlandsson,  \textsc{School of Mathematics, University of Bristol, Bristol, UK, and}\par\nopagebreak
  \textsc{Department of Mathematics and Statistics, UiT The Arctic University of
Norway}\par\nopagebreak
  \textit{E-mail address:} \texttt{v.erlandsson@bristol.ac.uk}

  \medskip

  Christopher J. Leininger, \textsc{Department of Mathematics, Rice University, Houston, TX, USA}\par\nopagebreak
  \textit{E-mail address:} \texttt{cjl12@rice.edu}

  \medskip

  Chandrika Sadanand, \textsc{Department of Mathematics, Bowdoin College} \par\nopagebreak \textsc{Brunswick, Maine, USA}\par\nopagebreak
  \textit{E-mail address:} \texttt{c.sadanand@bowdoin.edu}

}}

\abstract{A negatively curved hyperbolic cone metric is called {\em rigid} if it is determined (up to isotopy) by the support of its Liouville current, and {\em flexible} otherwise.  We provide a complete characterization of rigidity and flexibility, prove that rigidity is a generic property, and parameterize the associated deformation space for any flexible metric.  As an application, we parameterize the space of hyperbolic polygons with the same symbolic coding for their billiard dynamics, and prove that generically this parameter space is a point.
}

\section{Introduction}

Let $S$ be a closed, oriented surface of genus at least two and let $\varphi \in \Hyp_c(S)$ be the isotopy class of a negatively curved hyperbolic cone metric on $S$; see \S\ref{S:Hyp cone}.  Hersonsky and Paulin \cite{HerPaul} proved that $\varphi \in \Hyp_c(S)$ is determined by its marked length spectrum.  The proof follows Otal's approach \cite{Otal}, associating to $\varphi$ a {\em Liouville} (or {\em M\"obius}) {\em geodesic current}, $L_\varphi$ which determines, and is determined by, the marked length spectrum, then proving that $L_\varphi$ determines $\varphi$; see also \cite{croke,CFF,Frazier:LS}.  The analogous result for unit area nonpositively curved Euclidean cone metrics---also called, {\em flat metrics}---was proved by the second author with Bankovic in \cite{BL} extending a special case with Duchin and Rafi \cite{DLR}, and subsequently extended to all nonpositively curved Riemannian cone metrics by Constantine \cite{Constantine}.

In our earlier paper with Duchin \cite{oldpaper} we proved a stronger theorem for flat metrics: they are typically determined by the {\em support} of their Liouville current; see \S\ref{S:geodesics and currents}.  This result starkly contrasts the case of nonpositively curved {\em Riemannian} metrics for which the Liouville current always has full support, and thus no two such metrics can be distinguished by supports.  In general, we proved that two flat metrics have the same support for their Liouville currents  if and only if the metrics differ by an affine deformation (up to isotopy), and this is possible precisely when the holonomy of the metrics have order at most $2$; see \cite{oldpaper} for more details.

For $\varphi \in \Hyp_c(S)$, there is no notion of an ``affine deformation" which preserves $\CG_{\tphi}$, the support of the Liouville current $L_\varphi$.  On the other hand, in Theorem~\ref{T:deforming} below, we show that if we are in the (highly non-generic) situation that $(S,\varphi)$ admits a locally isometric branched covering of a hyperbolic orbifold $(S,\varphi) \to \mathcal O$ {\em and} every cone point maps to an orbifold point of even order, then any deformation of $\mathcal O$ to another hyperbolic orbifold $\mathcal O'$ lifts to a deformation $\varphi'$ of $\varphi$ without disturbing the support of the Liouville current; i.e.~so that $\CG_{\tphi} = \CG_{\tphi'}$.  Our main theorem states that this is the only way two metrics can have the same support for their currents.

\begin{supporttheorem} Suppose $\varphi_1,\varphi_2 \in \Hyp_c(S)$ and $\CG_{\tphi_1} = \CG_{\tphi_2}$.  Then $\varphi_1 = \varphi_2$ or else $\varphi_1,\varphi_2$ arise precisely as described above (that is, as in Theorem~\ref{T:deforming}).
\end{supporttheorem} 

We reiterate that for either a flat or hyperbolic cone metric, $\varphi$, the generic situation is that $\CG_{\tphi}$ determines $\varphi$.  In the non-generic case, the dimension of the deformation space for flat cone metrics preserving the support is precisely $2$, while for hyperbolic cone metrics, the dimension is still finite but can be arbitrarily large: indeed, it is parameterized by the Teichm\"uller space of the quotient orbifold produced in the proof of the \currentsupport\!.  See the remark at the end of \S\ref{S:flexibility proof}.  

An interesting special case is described by the following, which states that when a cone metric has ``too many cone points", then it is necessarily rigid.

\begin{corollary} \label{C:too many cone points} If $\varphi \in \Hyp_c(S)$ has at least $32(g-1)$ cone points (where $g$ is the genus of $S$) then $\varphi$ is rigid.
\end{corollary}

\noindent See  \S \ref{S:too many cone points} for the proof.  We do not claim that the bound $32(g-1)$ is sharp, but one cannot do better than a bound which is linear in $g$. Indeed, by taking unbranched covers of a genus two surface by a surface of genus $g \geq 3$ with the pull-back metric from the one constructed in Example~\ref{Ex:branched example}, we find flexible metrics on the surface of genus $g$ with $g-1$ cone points.


\subsection{Billiard rigidity}

Suppose $P \subset \mathbb H$ is a compact $n$--gon in the hyperbolic plane (assumed simply connected, but not necessarily convex). We assume that $P$ comes equipped with a labelling of the side by elements of $\CA = \{1,2,\ldots,n\}$, starting with $1$ on some edge and proceeding in counterclockwise order.  A biinfinite billiard trajectory determines a {\em bounce sequence} in $\CA^{\mathbb Z}$ by recording the labels of the sides encountered in order, and the the {\em bounce spectrum of $P$} is the set $\CB(P) \subset \CA^{\mathbb Z}$ consisting of all bounce sequences.  We say that $P$ is {\em billiard rigid} if $\CB(P)$ determines $P$ up to label-preserving isometry, and {\em billiard flexible} otherwise.

As an application of the \currentsupport \!\!, we prove a sharp symbolic-dynamical rigidity theorem characterizing rigidity and flexibility for hyperbolic polygons.  Our theorem is analogous to, but with quite a different conclusion from, the ``Bounce Theorem'' of our paper with Duchin, \cite{oldpaper} where we proved that a Euclidean polygon is billiard flexible if and only if all its interior angles are either $\tfrac{\pi}2$ or $\tfrac{3\pi}2$.  For hyperbolic polygons, the statement is slightly more technical, and requires an additional definition.  

A polygon $P$ is {\em reflectively tiled} if it can be tiled by isometric copies of a single polygon, called the {\em tiles}, so that any two adjacent tiles differ by reflection in their shared edge.  We further require that each interior angle of a tile is of the form $\frac{\pi}k$, for some $k \in \mathbb Z$, where $k$ is even if the vertex of the tile is also a vertex of $P$.  See \S\ref{S:billiards} for a detailed discussion.


\begin{billiardtheorem}   Given hyperbolic polygons $P_1,P_2$, we have $\CB (P_1) = \CB (P_2)$ if and only if
\begin{enumerate}
\item $P_1$ is isometric to $P_2$ by a label preserving isometry, or
\item $P_1,P_2$ are reflectively tiled and there exists a label-preserving homeomorphism $H \colon P_1 \to P_2$ that maps tiles to tiles, preserving their interior angles.
\end{enumerate}
\end{billiardtheorem}

The characterization of billiard flexibility is easily derived from the theorem; see \S\ref{S:billiards}.

\begin{corollary}  \label{C:billiard flexible} A hyperbolic polygon $P$ is billiard flexible if and only it is reflectively tiled with a non-triangular tile.
\end{corollary}

Because the statement of \billiardrigidity is slightly technical, we note a couple of special cases that help to illustrate the conclusion.
First, rigidity is generic; for example, Corollary~\ref{C:irrational angle rigid} says that any polygon with at least one interior angle that is not in $\pi \mathbb Q$ is rigid.  Moreover, Proposition~\ref{P:lots of examples} shows that an $n$--gon for which no interior angle is an even submultiple of $\pi$ (i.e.~of the form $\frac{\pi}{2k}$ for some $k \in \mathbb Z$), is rigid, while Corollary~\ref{C:deforming polygons} shows that if all angles are even submultiples of $\pi$, then there is an $(n-3)$--dimensional space of $n$--gons with the same bounce spectrum; namely those having the same corresponding interior angles. See \S\ref{S:billiard rigidity proof} 

For any polygon $P$, the bounce spectrum $\CB(P)$ consists of uncountably many biinfinite sequences and it is natural to wonder if one can draw the same conclusions with less information. The answer is yes (as in the Euclidean case) and one needs only consider the bounce sequences of {\em generalized diagonals}, which are {\em compact} billiard trajectories in $P$ that start and end at a vertex of $P$.  The bounce sequence of a generalized diagonal is a finite sequence and the set of all such, $\CB_\Delta(P)$, is countable.  This is enough information to determine $\CB(P)$ and thus we have the following.

\begin{theorem} \label{T:gen diagonal hyp} Given polygons $P_1,P_2 \subset \mathbb H$, we have $\CB_\Delta(P_1) = \CB_\Delta(P_2)$ if and only if one of the two conclusions of the \billiardrigidity holds.
\end{theorem}

The bounce spectrum of a hyperbolic polygon was previously studied by Ullmo and Giannoni \cite{UlGi} in the special case that all interior angles are acute (see also \cite{UlGi2} and Nagar--Singh \cite{NaSi}).  The primary objective in \cite{UlGi} was to find a set of ``grammar rules" that completely describe $\CB(P)$ for a given polygon.  There are some interesting connections with the \billiardrigidity that we explain in \S\ref{S:UG-connections}.

\subsection{Outline}  Here we explain the key ideas in the proofs of the two theorems above, focusing primarily on the \currentsupport\!.   First, given a metric $\varphi \in \Hyp_c(S)$ consider the pull back metric $\tphi$ on the universal cover $\tilde{S}$ of $S$. The set $\CG_{\tphi}$ can be alternatively described as the set of pairs of endpoints at infinity of the {\em basic} $\tphi$--geodesics, that is those which are limits of {\em nonsingular} (disjoint from cone points) $\tphi$--geodesics; see \S\ref{S:geodesics and currents}.   

Now, fix $\varphi_1,\varphi_2 \in \Hyp_c(S)$ and suppose $\CG_{\tphi_1} = \CG_{\tphi_2}$. Having the same set of endpoints  translates into a correspondence between basic $\tphi_1$--geodesics and basic $\tphi_2$-geodesics since they must be bounded Hausdorff distance apart.  The proof of the \currentsupport begins by successively adjusting (a representative of) $\varphi_2$ by an isotopy until it and $\varphi_1$ share some useful common features.  This starts by first showing that the cone points are naturally ``aligned" by the condition $\CG_{\tphi_1} = \CG_{\tphi_2}$ using {\em chains} at infinity: after adjusting $\varphi_2$ by an isotopy, the two metrics have the same set of cone points; see \S\ref{S:concurrency and partitions}.  Analyzing the way the basic geodesics partition the cone points in the universal covering, we explain how to adjust $\varphi_2$ and construct a triangulation with vertices at the cone points so that in {\em both metrics} the triangles are isometric to triangles in $\mathbb H$; see \S\ref{S:triangulations}.

Up to this point the outline of the proof is similar to that of the analogous theorem for Euclidean cone metrics in \cite{oldpaper}, though there are various technical differences we must account for throughout.  In contrast, the remainder of the proof substantially diverges from the Euclidean case.  To  prove the theorem we must {\bf either} show that the interior angles of the triangles of the triangulation are the same in the two metrics, ensuring the triangles are isometric and thus $\varphi_1=\varphi_2$, {\bf or} produce a locally isometric branched cover of an orbifold $(S,\varphi_1) \to \mathcal O_1$ (with cone points mapping to even order orbifold points) and prove that $\varphi_2$ is obtained by lifting a deformation of $\mathcal O_1$.  

This analysis starts by removing the common set of cone points to produce a punctured surface $\dot S$, and taking (the metric completion of) the universal cover $\hS$ of $\dot S$.  Using the triangulation, we show that the basic geodesics of the pulled back metrics $\hphi_1$ and $\hphi_2$ also have the same endpoints at infinity.  Associated to $\hphi_i$, for each $i=1,2$, there is a locally isometric developing map $D_i \colon (\hS,\hphi_i) \to \mathbb H$ which is equivariant with respect to its associated holonomy homomorphism $\rho_i \colon \pi_1(\dot S) \to \PSL_2(\mathbb R)$.    A delicate argument analyzing intersection patterns of basic geodesics is used to construct an orientation preserving homomorphism $h \colon \partial \mathbb H \to \partial \mathbb H$ that conjugates $\rho_1$ to $\rho_2$ within the homeomorphism group of the circle $\partial \mathbb H$.  Furthermore, $h$ sends the endpoints of any $D_1$--image of a basic $\hphi_1$--geodesic to the endpoints of the  $D_2$--image of the associated basic $\hphi_2$--geodesic; see \S\ref{S:conjugating}. It is by investigating the holonomy homomorphism image $\rho_1(\pi_1(\dot S))$ (and a slightly enlarged group) and using the existence of the conjugating homeomorphism that we can determine whether the two metrics are the same or differ by an orbifold deformation. 

Indeed, if $\rho_1(\pi_1(\dot S))$ is indiscrete we show that $h$ must in fact lie in $\PSL_2(\mathbb R)$ by an analysis of the closed, cocompact subgroups of $\PSL_2(\mathbb R)$; see \S\ref{S:top conjugacy}.  This implies that the angle between a pair of $D_1$--developed basic $\hphi_1$--geodesics and the corresponding $D_2$--developed basic $\hphi_2$--geodesics must agree.  Applying this to extensions of the sides of the triangles implies that the interior angles for each triangle in the triangulation are the same in the two metrics, and hence we have that $\varphi_1=\varphi_2$. 

One can push this argument further using an observation from earlier in the proof: during the first adjustment stage, it is shown that all basic $\tphi_1$--geodesics through a cone point in $(\tS, \tphi_1)$ correspond to $\tphi_2$--geodesics that pass through a cone point in $(\tS, \tphi_2)$ (after an adjustment, these become the same point).  Passing to $\hS$, we deduce that for every cone point $\zeta$ of $\hS$, $h$ ``witnesses" the concurrence of the $D_1$-- and $D_2$--images of the respective basic geodesics through $\zeta$.  This manifests itself in the fact that $h$ conjugates the order two elliptic isometry fixing $D_1(\zeta)$ to the order two elliptic isometry fixing $D_2(\zeta)$.  Taking the groups $\Gamma_i^0 < \PSL_2(\mathbb R)$ generated by $\rho_i(\pi_1(\dot S))$ together with these order two elliptic elements about $D_i(\zeta)$, over all cone points $\zeta$ of $\hS$, we see that $h$ actually conjugates $\Gamma_1^0$ to $\Gamma_2^0$; see \S\ref{S:rigidity}.  Consequently, if $\Gamma_1^0$ is indiscrete we again deduce, by the same argument as above, that $h$ must lie in $\PSL_2(\mathbb R)$ and so $\varphi_1=\varphi_2$ as before.  

The only option left is that $\Gamma_1^0$ and $\Gamma_2^0$ are {\em discrete}.  Then the quotient $\mathcal O_i^0 = \mathbb H/ \Gamma_i^0$ are hyperbolic orbifolds, for $i=1,2$.  By construction each $D_i$ is equivariant and hence descend to branched covers from $(S,\varphi_i) \to \mathcal O_i^0$ under which the cone points map to even order orbifold points.  The homeomorphism $h$ conjugates $\Gamma_1^0$ to $\Gamma_2^0$ and hence determines a homeomorphisms $\mathcal O_1^0 \to \mathcal O_2^0$.  Chasing through the diagram shows that (after another adjustment) $\varphi_2$ differs from $\varphi_1$ by lifting this homeomorphism, completing the proof; see \S\ref{S:flexibility proof}.

The \billiardrigidity follows from the \currentsupport by an unfolding procedure.  Roughly speaking, given a hyperbolic $n$--gon $P$, we consider a negatively curved hyperbolic cone surfaces tiled by copies of $P$.  Then ``folding the tiles up" defines a map from the surface to $P$ that sends each tile isometrically to $P$; consequently,  we call such a surface and metric an unfolding.  A key property of an unfolding is that nonsingular geodesics in the surface project to billiard trajectories, and vice versa billiard trajectories lift to nonsingular geodesics.

Given two hyperbolic $n$--gons $P_1$ and $P_2$, we find a {\em common unfolding}, which is a surface with two different hyperbolic cone metrics $\varphi_1$ and $\varphi_2$ so that the tilings define the same cell structure (the 2-cells are the tiles) and so the labelings from the polygons agree.  When $\CB(P_1) = \CB(P_2)$, we lift a common unfolding to the universal cover of the surface, and observe that the bounce sequence of a billiard trajectory also specifies (from a given starting tile) a symbolic coding of a nonsingular geodesic.  In this way, $\CB(P_1) = \CB(P_2)$ naturally implies $\CG_{\tphi_1} = \CG_{\tphi_2}$ since a $\tphi_1$--geodesic and $\tphi_2$--geodesic passing through the same set of tiles forces them to remain a bounded distance apart; see \S\ref{S:unfoldings}.

This first part of the proof is similar to the Euclidean billiard case, but again, at this point the proofs diverge.  Since $\CG_{\tphi_1} = \CG_{\tphi_2}$, we find a homeomorphism $h: \partial\mathbb{H}\to \partial\mathbb{H}$ and show that it conjugates the group generated by reflections in the faces of $P_1$ to the group generated by reflections in the faces of $P_2$.  If $h \in \PSL_2(\mathbb R)$, by looking at the extensions of the geodesic sides to biinfinite geodesics and the effect of $h$ on the endpoints, we see that the extension of $h$ to an isometry of $\mathbb H$ maps $P_1$ to $P_2$.  

Now, to determine whether $h$ belongs to $\PSL_2(\mathbb R)$ we proceed similarly to the argument above. If the group $R^0_{P_1}$ generated by reflections in the sides of $P_1$ together with the order two elliptic elements fixing the vertices is indiscrete, then $h$ is necessarily in $\PSL_2(\mathbb R)$, and we are in case (1) of the theorem; see \S\ref{S:reflection groups}.  If on the other hand $R^0_{P_1}$ is discrete, one can check that this group is also generated by reflections in a (typically) smaller polygon inside of $P_1$.  The union of the lines of reflection for this group defines a tiling of $\mathbb H$ for which $P_1$ is a union of tiles.  The map $h$ sends the endpoints of all these lines of reflection to the endpoints of the lines of reflection for the associated group determined by $P_2$.  This is enough to construct a homeomorphism $\mathbb H \to \mathbb H$ sending tiles of one tiling to tiles of the other and sending $P_1$ to $P_2$.  The interior angles of the tiles is preserved by this homeomorphism since $h$ conjugates the dihedral vertex stabilizer in one group to an isomorphic group in the other.  This essentially puts us in case (2) of the theorem, and completes the proof; see \S\ref{S:billiard rigidity proof}.

\medskip

The paper is organized as follows. In \S\ref{S:prelim} we introduce the notation and terminology used throughout the paper, describe the geodesics in $\tS$ and $\hS$, and discuss topological conjugates of subgroups of $\PSL_2(\mathbb{R})$. \S\ref{S:deformations} is devoted to proving Theorem~\ref{T:deforming} mentioned above, showing that two hyperbolic cone metrics $\varphi_1, \varphi_2$ which branch cover hyperbolic orbifolds (with cone points projecting to even order orbifold points) and differ by a lift of an orbifold deformation satisfy $\CG_{\tphi_1} = \CG_{\tphi_2}$. The normalization procedures of the metric $\varphi_2$ alluded to in the outline above are explained in \S\ref{S:geodesics and triangulations}. 
In \S\ref{S:holonomy} we construct the homeomorphism $h: \partial\mathbb{H}\to  \partial\mathbb{H}$ which conjugates the holonomies and is a key tool in the proof of the \currentsupport\!\!, which we prove in \S\ref{S:proof of current}. Finally, in \S\ref{S:billiards} we discuss hyperbolic billiards and prove the \billiardrigidity\!.

\bigskip

\noindent
{\bf Acknowledgements.} We would like to thank Ben Barrett, Moon Duchin, Hugo Parlier, Alan Reid, and Juan Souto for useful conversations throughout the course of this work.  In particular, the argument involving triangle groups in Lemma~\ref{L:billiards indiscrete} is due to Alan Reid.  The authors would like to thank the referee for helpful comments on the initial version of the paper.  Leininger was partially supported by NSF grant DMS-1811518 and DMS-2106419. Erlandsson was partially supported by EPSRC grant EP/T015926/1.

\section{Preliminaries}\label{S:prelim}

\subsection{Hyperbolic cone metrics} \label{S:Hyp cone}

We write $\tHyp_c(S)$ for the space of negatively curved hyperbolic cone metrics on $S$.  These are metrics, locally isometric to $\mathbb H$ away from a finite, positive number of cone singularities with cone angles greater than $2\pi$.  We refer to the non-cone points as {\em regular points}.   The group $\Homeo_0(S)$ of homeomorphisms isotopic to the identity acts (on the right) on $\tHyp_c(S)$ by pullback, and we let $\Hyp_c(S)$ be the quotient by this action.  That is, given $\varphi,\varphi' \in \tHyp_c(S)$, declaring $\varphi \sim \varphi'$ if there exists a homeomorphism $f \colon S \to S$ isotopic to the identity, such that $f^*(\varphi') = \varphi$, then $\Hyp_c(S) = \tHyp_c(S)/\!\!\sim$ is the quotient of $\tHyp_c(S)$ by the equivalence relation $\sim$.

\begin{remark} While we have assumed that the cone point set is nonempty, we note that the  \currentsupport\! is trivially valid without this assumption.  Indeed, for any nonsingular hyperbolic metric, its Liouville current has full support. Thus, any two nonsingular hyperbolic metrics have the same support for their Liouville currents, and do in fact differ by the lift of a deformation from the trivial orbifold covering of itself.
\end{remark}

We fix a universal covering $p \colon \tS \to S$.  For any reference negatively curved metric, we let $S^1_\infty$ be the boundary at infinity of $\tS$ with respect to the pull back metric. This defines a compactification of $\tS$ which is independent of the metric, up to homeomorphism that is the identity on $\tS$.  In particular, $S^1_\infty$ is independent of the choice of negatively curved reference metric (indeed, $S^1_\infty$ is the Gromov boundary, which can be defined by any geodesic metric on $S$ pulled back to $\tS$). For any $\varphi \in \tHyp_c(S)$, the metric $\tphi = p^*(\varphi)$ is CAT(-1), by Gromov's link condition and the Cartan-Hadamard Theorem; see \cite{BridHaef}.


\subsection{Puncturing surfaces} \label{S:puncturing surfaces}

Suppose $\varphi \in \tHyp_c(S)$, let $C_0=\cone(\varphi)$ be the set of cone points for $\varphi$, and let $\dot S = S\smallsetminus C_0$ denote the 
punctured surface obtained by deleting the cone points.  Let $\hat p \colon \widetilde{\dot S} \to \dot S$ be the universal cover and $\hphi = \hat p^*(\varphi)$ the pull back metric.  Since $\hat p$ is a local isometry, it extends over the metric completions, and we denote the map by the same name:
\[ \hat p \colon \hat S \to S\]
where $\hat S$ denotes the completion of $\widetilde{\dot S}$. 
To describe the local behavior of this map and the space $\hat S$ near the completion points, suppose that $\epsilon > 0$ is sufficiently small so that the $\epsilon$--neighborhood of $C_0$ is a disjoint union of topological disks, $\Delta_1 \cup \ldots \cup \Delta_k$, with each $\Delta_j$ containing exactly one point $\zeta_j \in C_0$.  For any component $U \subset \hat p^{-1}(\Delta_j - \{\zeta_j\})$,
\[ \hat p|_U \colon U \to \Delta_j - \{\zeta_j\} \]
is a universal cover of $\Delta_j - \{\zeta_j\}$.  For such $U$, there is a single completion point $\hat \zeta_U$ that projects to $\zeta_j$.

In fact, any non-convergent Cauchy sequence in $\widetilde{\dot S}$ projects to such a sequence in $\dot S$, which is necessarily eventually contained in some $\Delta_j$.  By taking $\epsilon >0$ sufficiently small, we can assume that the sequence is contained in a single component $U$, and hence converges to the point $\hat \zeta_U$.  That is, the points $\hat \zeta_U$ account for all completion points.  We write $\hat C_0 = \hat p^{-1}(C_0)  \subset \hS$ for the completion points, which (by an abuse of terminology) we also refer to as cone points.

Note that $\hat p$ factors through a map $\tilde p \colon \hat S \to \tS$.  Furthermore, there are covering actions of $\pi_1\dot S$ and $\pi_1S$ on $\hS$ and $\tS$ respectively, and a homomorphism $\pi_1\dot S \to \pi_1S$ (induced by inclusion $\dot S\to S$, up to conjugation) so that $\tilde p$ is equivariant:
\[ \begin{tikzcd}[row sep=8]
\pi_1 \dot S \ar[d, phantom, "\circlearrowright"] \ar[r] & \pi_1S \ar[d, phantom, "\circlearrowright"] \\
\hat S \ar[r, "\tilde p"] \ar[dddr,"\hat p" below] & \tS \ar[ddd, "p"] \\\\\\
& \, S. 
\end{tikzcd}
\]

In what follows, when considering two metrics $\varphi_1,\varphi_2 \in \tHyp_c(S)$, we will only consider the space $\hat S$ when $\varphi_1$ and $\varphi_2$ have common cone point set $C_0 = \cone(\varphi_1) = \cone(\varphi_2)$. Under this assumption, we note that the discussion above implies that the metric completion of $\widetilde {\dot S}$ is the same for either metric $\hphi_1$ or $\hphi_2$.
Also, with respect to either metric $\hphi_1$ or $\hphi_2$, $\hS$ is CAT(-1), and in particular is Gromov hyperbolic.  The identity $(\hS,\hphi_1) \to (\hS,\hphi_2)$ is a quasi-isometry (in fact, after adjusting $\varphi_2$ by an arbitrarily small homeomorphism (fixing $C_0$), we may assume $\varphi_1$ and $\varphi_2$ are biLipschitz equivalent, and hence the same is true of $\hphi_1$ and $\hphi_2$).  The boundary at infinity of $\hS$, denoted $\hS_\infty^1$, is thus well-defined and independent of the metric (in the same sense as for $\tS$).  Moreover, any two points of the boundary are connected by a biinfinite geodesic.  Although $\hS^1_{\infty}$ is not a circle (it is not even compact), it is homeomorphic to a subset of the circle in a $\pi_1\dot S$--invariant way, since $\widetilde{\dot S}$ is a surface.

\subsection{Geodesics in $\tS$ and $\hS$} \label{S:geodesics and currents} 

We suppose $\varphi \in \tHyp_c(S)$ throughout and let $C_0=\cone(\varphi)$. We refer to the sets $\tC_0 = p^{-1}(C_0)$ and $\hC_0 = \hat{p}^{-1}(C_0)$ as the cone point sets of $(\tS, \tphi)$ and $(\hS, \hphi)$. By a $\tphi$--geodesic or $\hphi$--geodesic we will always mean a biinfinite geodesic in $(\tS, \tphi)$ or $(\hS, \hphi)$, respectively. Moreover, we consider the geodesics as unparameterized, and we equip the space of geodesics with the quotient topology coming from the compact-open topology on the set of parametrized geodesics, where we forget the parametrization.

A $\tphi$--geodesic segment is a subsegment of a $\tphi$--geodesic and can be finite, half-infinite (i.e. a ray), or bi-infinite, and we make the analogous convention for $\hphi$--geodesic segments. A $\tphi$--geodesic (resp. $\hphi$--geodesic) segment is {\em singular} if its interior intersects $\tC_0$ (resp. $\hC_0 $) and {\em nonsingular} otherwise. Note that nonsingular geodesic segments are hyperbolic geodesic segments.  If a $\tphi$--geodesic segment $\eta$ meets a cone point $\zeta \in \tC_0$ in the interior of $\eta$, then any sufficiently small disk neighborhood of $\zeta$ is divided into two half-disks by $\eta$ which we call the {\em sides of $\eta$ near $\zeta$} (nested disks determine nested sides), and we note that being geodesic, the angle on each side is at least $\pi$.  If a cone point $\hat\zeta \in \hC_0$ is encountered in the interior of a singular $\hphi$--geodesic segment, then any sufficiently small ``disk" neighborhood (i.e.~an open set $U$ as above) is divided into three sides; on only one side does it make sense to measure the angle, but again that angle must be at least $\pi$. Finally, a nonsingular $\tphi$- or $\hphi$--geodesic segment between two cone points is called a {\em saddle connection}.

Observe that a $\tphi$--geodesic $\eta$ has two {\em globally} defined sides since it divides $\tS$ into two components.  Orienting $\eta$, and appealing to the orientation on $\tS$, we can refer to the two sides as the {\em positive} (left-hand) and {\em negative} (right-hand) sides of $\eta$.  A $\hphi$--geodesic can have infinitely many globally defined sides, but choosing an orientation they can still be prescribed a sign, positive or negative.  Of course, orienting a geodesic segment also gives rise to a well-defined sign for the sides at cone points encountered (locally), as well.

The set of distinct pairs of points in $S^1_\infty$ is denoted by
\[ \CG(\tS) = S^1_\infty \times S^1_\infty \setminus \Delta/(x,y) \sim(y,x)\]
where $\Delta$ is the diagonal in $S^1_\infty \times S^1_\infty$. 
Given a $\tphi$--geodesic $\eta$, we let $\partial_{\tphi}(\eta) \in \CG(\tS)$ be the set of endpoints on $S^1_\infty$ of $\gamma$. Since $\tphi$ is CAT(-1), this determines a homeomorphism
$$\partial_{\tphi}: \{\text{\small $\tphi$--geodesics in }\tS\} \to \mathcal{G}(\tS).$$
Note that two $\tphi$--geodesics are either disjoint, meet at exactly one point where they transversely intersect, or coincide either in a single cone point or in a (unique maximal) geodesic segment. If they coincide in a segment, this is either a concatenation of saddle connections or an infinite ray emanating from a cone point (and the two geodesics share an endpoint at infinity). In the latter case we say the two geodesics are {\em cone point asymptotic}. We say that $\{x_1, y_1\}, \{x_2, y_2\} \in \mathcal{G}(\tS)$ {\em link} if $x_2, y_2\notin\{x_1, y_1\}$ and $x_2$ and $y_2$ lie in different components of $S^1_\infty\setminus\{x_1, y_1\}$. We say that two $\tphi$--geodesics $\eta_1$ and $\eta_2$ {\em cross} if their endpoints  $\partial_{\tphi}(\eta_1)$ and $\partial_{\tphi}(\eta_2)$ link (meaning, $S^1_\infty \setminus \partial_\tphi(\eta_1)$ contains one point of $\partial_\tphi(\eta_2)$ in each component).  When $\eta_1$ and $\eta_2$ cross, they either intersect transversely at a point or they coincide either in a single cone point or in a concatenation of saddle connections, see Figure \ref{F:crossing}. Moreover, if a pair of geodesics intersect transversely, then they must cross.  On the other hand, a pair of geodesics may coincide in a nontrivial segment or cone point without crossing. 


\begin{figure}[htb]
\begin{center}
  \captionsetup{width=.85\linewidth}
\begin{tikzpicture}[scale = .5]
\draw (0,0) circle (3);
\draw (-1.3077,-2.7) -- (-1,-1) -- (0,0.8) -- (2.5,1.6583);
\draw (-2.5,-1.6583) -- (-1,-1) -- (0,0.8) -- (-1.3077,2.7); 
\draw[fill=black] (-1.3077,-2.7) circle (.04cm);
\draw[fill=black] (-1.3077,2.7) circle (.04cm);
\draw[fill=black] (2.5,1.6583) circle (.04cm);
\draw[fill=black] (-2.5,-1.6583) circle (.04cm);
\draw[fill=black] (-1,-1) circle (.06cm);
\draw[fill=black] (0,0.8) circle (.06cm);
\draw[fill=black] (-0.6,-0.3) circle (.06cm);
\node[right] at (1,1.7) {\small{$\eta_1$}};
\node[right] at (-2.5,-1.0) {\small{$\eta_1$}};
\node[right] at (-1.3,-2.1) {\small{$\eta_2$}};
\node[right] at (-1.7,1.7) {\small{$\eta_2$}};
\node[below] at (-1.4077,-2.7) {\small{$y_2$}};
\node[above] at (-1.3077,2.7) {\small{$x_2$}};
\node[above] at (3,1.4583) {\small{$x_1$}};
\node[below] at (-2.7,-1.6583) {\small{$y_1$}};
\draw (9,0) circle (3);
\draw [domain=-43:43] plot ({4.75736+4*cos(\x)}, {4*sin(\x)});
\draw (11.5,1.6583) -- (6.5,-1.6583);
\draw[fill=black] (7.72,-2.7) circle (.04cm);
\draw[fill=black] (7.72,2.7) circle (.04cm);
\draw[fill=black] (11.5,1.6583) circle (.04cm);
\draw[fill=black] (6.5,-1.6583) circle (.04cm);
\node[below] at (7.72,-2.7) {\small{$y_2$}};
\node[above] at (7.72,2.7) {\small{$x_2$}};
\node[above] at (12,1.4583) {\small{$x_1$}};
\node[below] at (6.3,-1.6583) {\small{$y_1$}};
\node[left] at (8.7,1.1) {\small{$\eta_2$}};
\node[right] at (10,1.5) {\small{$\eta_1$}};
\end{tikzpicture}
\caption{Two examples of crossing geodesics $\eta_1$ and $\eta_2$ with their linking endpoints $\partial_{\tphi}(\eta_1)=\{x_1, y_1\}$ and $\partial_{\tphi}(\eta_2)=\{x_2, y_2\}$.}
\label{F:crossing}
\end{center}
\end{figure}
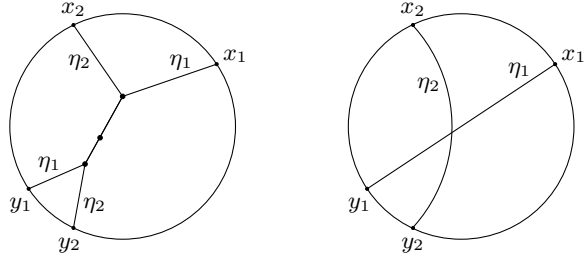

We let $\CG(\tphi)$ denote the closure of the set of nonsingular $\tphi$--geodesics and call the geodesics in $\CG(\tphi)$ the {\em basic $\tphi$--geodesics}. We define
\[ \CG_{\tphi}  =  \partial_{\tphi}(\CG(\tphi)) \subset \CG(\tS)\]
to be the set of all endpoints of basic geodesics. 
As mentioned in the introduction, the set $\CG_\tphi$ is exactly the support of the Liouville current $L_\varphi$ (see \cite{BL,Constantine}).  As a notational reminder, when some metric is involved, we will typically include the metric as a subscript to denote endpoints of geodesics while it will be included in parentheses to denote the actual geodesics.

Similarly, we let $\mathcal{G}(\hphi)$ be the closure of the set of of nonsingular $\hphi$--geodesics which we call {\em basic $\hphi$--geodesics}, denoting the set of their endpoints by $\mathcal{G}_{\hphi}\subset \CG(\hS) = \hS_\infty^1 \times \hS_\infty^1 \setminus \Delta/(x,y) \!\sim\! (y,x)$. 
 
The following lemma describes a few elementary properties of basic geodesics in $\tS$ and $\hS$. For $\tS$, the analogous statements in the case of Euclidean cone metrics were proved in \cite[Section~3]{BL} and for non-positively curved Riemannian cone metrics in \cite[Section~5]{Constantine}. 

\begin{lemma} [Basic Geodesics] \label{L:basic geodesics} Suppose $\varphi \in \tHyp_c(S)$.  For any basic geodesic in $(\tS, \tphi)$ that meets a cone point, it does so making angle exactly $\pi$ on one side at that point.  If a basic $\tphi$--geodesic encounters more than one cone point, then after choosing an orientation and ordering the cone points it meets accordingly, the sides on which the angle is $\pi$ can change sign at most once.  Consequently, the set $\CG^2_{\tphi}$ of basic geodesics meeting at least $2$ cone points is countable. The same statements are true for basic geodesics in $(\hS, \hphi)$: at each point $\zeta\in\hC_0$ a basic $\hphi$--geodesic encountering it makes angle exactly $\pi$, and if a basic $\hphi$--geodesic encounters more than one point of $\hC_0$, the signs of the sides on which it makes angle $\pi$ switch at most once.
\end{lemma}

\begin{proof}
We prove the statements for basic $\tphi$--geodesics, then describe any modifications necessary for the corresponding statement for basic $\hphi$--geodesics.

Consider any $\eta \in \CG(\tphi)$ and suppose that $\zeta$ is a cone point on $\eta$.  We first prove that the angle is $\pi$ on one side of $\eta$ at $\zeta$. Let $\{\eta_i\}$ be a sequence of nonsingular geodesics converging to $\eta$, and observe that for each sufficiently large $i$, there is a small disk neighborhood of $\zeta$ which $\eta_i$ intersects, and lies on one side or the other of $\zeta$ (since $\eta_i$ does not contain $\zeta$).   Passing to a subsequence if necessary, we can assume all $\{\eta_i\}$ are on the same side of $\eta$ in each of these neighborhoods.  Since each $\eta_i$ make angle $\pi$ on each side of every point it passes through, it follows that $\eta$ must make angle $\pi$ at $\zeta$ on the side containing all $\eta_i$.  For $\eta \in \CG(\hphi)$, the argument is similar, noting that in this case (in sufficiently small neighborhoods) the approximating geodesics $\eta_i$ must lie on the side of $\eta$ making finite angle, which is then necessarily equal to $\pi$.

Next we assume $\eta \in \CG(\tphi)$, and show that the side one which it makes angle $\pi$ can change at most once. To this end, assume $\eta$ passes through at least three points of $\tC_0$, say $\zeta_1, \zeta_2, \zeta_3$, and that these points appear in order along $\eta$, with respect to some orientation on $\eta$. We suppose $\eta$ makes angle $\pi$ on the positive side at $\zeta_1$ and $\zeta_3$ but at the negative side at $\zeta_2$, and will arrive at a contradiction. As above, let $\{\eta_i\}$ be an approximating sequence of nonsingular geodesics. Note that for large $i$ we must have that $\eta_i$ is on the positive side of $\eta$ in neighborhoods of $\zeta_1$ and $\zeta_3$ while on the negative side of $\eta$ in a neighborhood of $\zeta_2$. It follows that for large $i$ the geodesics $\eta_i$ and $\eta$ contain segments that bound a bigon, which is impossible since the metric is CAT(-1).  Since the metric $\hphi$ on $\hS$ is also CAT(-1), the exact same argument works in $\hS$ as well.

Finally, we explain why these facts imply that the set of basic geodesics meeting at least $2$ cone points is countable. Every $\tphi$--geodesic meeting at least two cone points contains a saddle connection.  Since a saddle connection is determined by its endpoints in $\tC_0$, and since $\tC_0$ is countable, there are only countably many saddle connections.  For each saddle connection, there are only countably many basic $\tphi$--geodesics containing it: one that always makes angle $\pi$ on the positive side, one that always makes angle $\pi$ on the negative side, and a countable number that switch.  To see that this last set is countable, divide into two cases, depending on whether it switches from $\pi$ on the positive side to $\pi$ on the negative side, or vice-versa, then index the cone points encountered by any such geodesic with an interval in $\mathbb Z$ (with one of the endpoints of the saddle connections being $0$) and record the index of the last cone point before the switch of sides occurs.  Therefore, the set of basic $\tphi$--geodesics containing at least two cone points is a countable union of countable sets, hence countable.
\end{proof}

We say that a geodesic segment (in either $\tS$ or $\hS$) is a {\em basic segment} if it is either nonsingular or it meets each cone point in its interior at an angle exactly $\pi$ on one side and this side switches at most once. This name is justified because we will show below that a segment is basic if and only if it can be extended to a basic geodesic. In fact, we show that the description of basic geodesics in Lemma \ref{L:basic geodesics} actually characterizes basic geodesics: 

\begin{lemma} \label{L:characterizing basics}
If $\eta$ is a $\tphi-$geodesic (or $\hphi$--geodesic) which makes angle $\pi$ on one side at every cone point it meets, and this side switches at most once, then $\eta$ belongs to $\mathcal G(\tphi)$ (respectively, $\CG(\hphi)$). Moreover, a geodesic segment is a basic segment if and only if it can be extended to a basic geodesic.
\end{lemma} 

\begin{proof}
Let $\eta$ be a $\tphi$--geodesic as in the statement (the case of $\hphi$--geodesics follow by almost identical arguments). We will construct a sequence of nonsingular geodesics $\{\eta_i\}$ converging to $\eta$, showing that $\eta$ is a basic geodesic. 

First suppose that $\eta$ makes angle $\pi$ on the same side (say the positive side) at every cone point it encounters. Pick any regular point $x$ on $\eta$ and let $\delta$ be a short geodesic segment emanating from $x$ on the positive side of $\eta$ and meeting $\eta$ orthogonally.  By taking $\delta$ short enough we can assume it contains no cone points. For every point $y$ on $\delta$ let $\eta_y$ be a basic geodesic through $y$ and orthogonal to $\delta$. Note that $\{\eta_y\}_{y \in \delta}$ is a set of pairwise disjoint geodesics, since if any two intersected this would result in a geodesic triangle with angle sum at least $\pi$, which is impossible in a CAT(-1) metric. Hence for each $y \in \delta - \{x\}$, the geodesic $\eta_y$ lies on the positive side of $\eta = \eta_x$.  Moreover, only countably many $y$ gives rise to a singular $\eta_y$ (since there are only countably many points in $\tC_0$) and hence picking any sequence of points $\{y_i\}$ on $\delta$ that avoids this countable set and converging to $x$ we get a sequence of nonsingular geodesics $\eta_i = \eta_{y_i}$ limiting to $\eta$. 

Now suppose $\eta$ passes through at least two cone points and that the side at which it makes angle $\pi$ at the cone points switches once.  Assume $\eta$ makes angle $\pi$ on, say, the negative side at cone point $\zeta$ and on the positive side at cone point $\xi$ and that hat $\eta$ encounters $\zeta$ immediately before $\xi$. Let $[\zeta, \xi]$ denote the segment of $\eta$ between the two points and let $x$ be any regular point on $[\xi, \zeta]$. For each $\theta\in(0, \pi)$ let $\eta_{\theta}$ be a geodesic that transversely intersects $\eta$ at $x$ making an angle as shown in Figure~\ref{F:approximating from the middle}. Then $\eta_{\theta}$ is nonsingular for all but countably many directions $\theta$, and hence picking any sequence $\{\theta_i\}$, avoiding this countable set, such that $\theta_i\to 0$ and setting $\eta_i = \eta_{\theta_i}$ results in a sequence of nonsingular geodesics $\{\eta_{i}\}$ converging to $\eta$.  

\begin{figure} 
\begin{center}
  \captionsetup{width=.80\linewidth}
\begin{tikzpicture}
\draw (-1,0) -- (11,0);
\filldraw (2,0) circle (1.3pt);
\filldraw (8,0) circle (1.3pt);
\filldraw (10,0) circle (1.3pt);
\filldraw (0,0) circle (1.3pt);
\filldraw (5,0) circle (.7pt);
\draw (-.5,-1.5) -- (10.5,1.5);
  \draw [thin,domain=180:360] plot ({2+.2*cos(\x)}, {.2*sin(\x)});
  \draw [thin,domain=180:360] plot ({.2*cos(\x)}, {.2*sin(\x)});
  \draw [thin,domain=0:180] plot ({10+.2*cos(\x)}, {.2*sin(\x)});
  \draw [thin,domain=0:180] plot ({8+.2*cos(\x)}, {.2*sin(\x)});
  \draw [domain=0:15] plot ({5+.7*cos(\x)}, {.7*sin(\x)});
\node at (0,-.35) {\small $\pi$};
\node at (2,-.35) {\small $\pi$};
\node at (2,.25) {\small $\zeta$};
\node at (8,-.25) {\small $\xi$};
\node at (8,.35) {\small $\pi$};
\node at (10,.35) {\small $\pi$};
\node at (5,.3) {\small $x$};
\node at (6.5,.2) {\small $\theta$};
\node at (10.7,-.3) {\small $\eta$};
\node at (9.3,1.4) {\small $\eta_\theta$};
\end{tikzpicture} 
\caption{Constructing an approximation of $\eta$ by $\eta_\theta$, as $\theta \to 0$.} \label{F:approximating from the middle}
 \end{center}
 \end{figure}
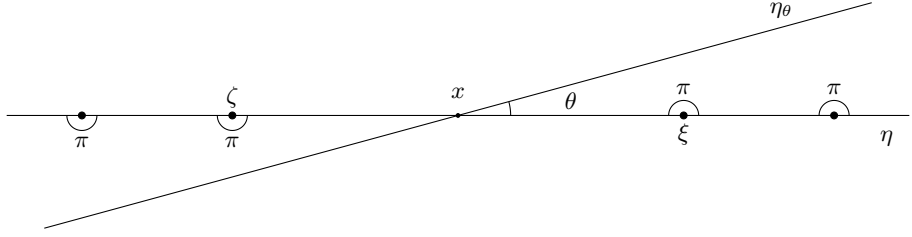

For the second claim, let $\sigma$ be a basic geodesic segment. First suppose its interior makes angle $\pi$ on the positive side at every cone point it meets. Extend $\sigma$ to a geodesic by making angle exactly $\pi$ on the positive side at every cone point it meets. By the above, the resulting geodesic is a basic geodesic. Now, if $\sigma$ switches once from making angle $\pi$ on the positive side to making angle $\pi$ on the negative side, extend $\sigma$ forward by a geodesic ray making angle $\pi$ on the negative side, and backward by a geodesic ray making angle $\pi$ on the positive side. Again by the above, the result is a basic geodesic. The converse follows from Lemma \ref{L:basic geodesics} and the definition of a basic geodesic segment.
\end{proof}

We note that, since in a neighborhood of any regular point the map $\tilde p: \hS\to \tS$ is an isometry, and since cone points are sent to cone points, any nonsingular geodesic in $\hS$ projects to a nonsingular geodesic in $\tS$. For the same reason, every nonsingular geodesic in $\tS$ can be lifted to a nonsingular geodesic in $\hS$. In fact, the same is true for basic geodesics.

\begin{lemma}\label{L:basic to basic}
The map $\tilde p: \hS\to \tS$ sends basic geodesics to basic geodesics, and every basic geodesic in $\tS$ is the $\tilde p$--image of a basic geodesic in $\hS$.
\end{lemma}

\begin{proof}
Given a basic geodesic $\eta$ in $\hS$, there is a sequence of nonsingular geodesics $\{ \eta_i \}$ that converges to $\eta$. Then the sequence $\{ \tilde p (\eta_i) \}$ of nonsingular geodesics must converge to $\tilde p (\eta)$, and so $\tilde p (\eta)$ must be a basic geodesic in $\tS$.

Now let $\lambda$ be a basic geodesic in $\tS$ that is parametrized at unit speed, and $\{\lambda_i\}$ a sequence of nonsingular unit speed geodesics converging to $\lambda$. We can assume $\{\lambda_i\}$ is one of the sequences constructed in the proof of Lemma \ref{L:characterizing basics} above. In particular, it is a sequence of nonsingular geodesics that are either pairwise disjoint or all coincide in a unique (regular) point. In the latter case, let $x$ be this point (which by construction is a point on $\lambda)$ and in the former case, there is a common perpendicular geodesic segment to $\{\lambda_i\}$ and $\lambda$ and $x$ is its intersection with $\lambda$.

Let $U$ be an evenly covered open disk containing $x$ and no cone points. Consider the tail of the sequence $\{\lambda_i\}$ that intersects $U$. Dropping everything but the tail and re-indexing, every geodesic in the sequence $\{ \lambda_i \}$ intersects $U$. Choosing base points in $U$, lift the geodesics $\{ \lambda_i\}$ to nonsingular geodesics $\{ \eta_i \}$ in $\hS$. We claim that $\{ \eta_i \}$ converges to a geodesic $\eta$ in $\hS$. If this were true, $\eta$ would be a basic geodesic since it is the limit of a sequence of nonsingular geodesics, and $\tilde p (\eta) = \lambda$. To show that $\{ \eta_i \}$ converges, we first recall that convergence in the compact open topology is equivalent to uniform convergence on compact sets.

Since $\{ \lambda_i \}$ converges in the compact open topology, $\{ \lambda_i \}$ must converge uniformly on compact sets. Given a compact set $K \subseteq \mathbb{R}$, there must be a ball $B_K$ containing $U$ that contains $\cup_i \lambda_i(K)$. If $B_K$ contains cone points, there are geodesic rays emanating from the cone points to $\partial B_K$  that do not intersect $\cup_i \lambda_i$ and do not intersect each other.

In the case where $\lambda_i$ intersect in a point $x$, these rays can be chosen as follows. Connect $x$ to each cone point by a geodesic segment, and then take the rays to be a geodesic continuation of these segments emanating from the cone points. No two can intersect because if they did, a geodesic bigon would be created with vertices $x$ and the intersection point.

In the case where $\lambda_i$ are disjoint, let $\gamma$ be the common perpendicular geodesic segment to $\{\lambda_i\}$ and recall $x = \lambda \cap \gamma$. The rays can be found as follows. Take the perpendicular geodesic segments from $\gamma$ to each cone point, and then take the rays to be a geodesic continuation of these segments emanating from the cone points. No two can intersect because if they did, a geodesic triangle would be formed by the intersecting geodesics and $\gamma$ with angle sum larger than $\pi$.

Let the union of the rays be called $R$. Then let $B := B_K \setminus R$. We have chosen $R$ so that $B$ contains no cone points, and so that $B$ is simply connected. Therefore, $B$ can be lifted homeomorphically and intrinsically isometrically to a bounded compact set in $\hS$ containing $\cup_i \eta_i(K)$. Therefore $\{ \eta_i \}$ converges uniformly on $K$ as desired.
\end{proof}

We end this subsection with an observation about the sets $\CG_{\tphi}$ and $\CG_{\hphi}$ of endpoints of basic geodesics in $\tS$ and $\hS$. To prove the \currentsupport we will be interested in metrics $\varphi_1, \varphi_2$ for which $\CG_{\tphi_1}=\CG_{\tphi_2}$, and we will see that this is equivalent to the assumption that $\CG_{\hphi_1}=\CG_{\hphi_2}$ (up to replacing $\varphi_2$ with an equivalent metric).  We prove one implication here;  see Lemma~\ref{L:hat combinatorics} for the converse. 

\begin{lemma} \label{L:same hat, same tilde}
Let $\varphi_1, \varphi_2 \in \tHyp_c(S)$ with $C_0 = \cone(\varphi_1) = \cone(\varphi_2)$. If $\CG_{\hphi_1} = \CG_{\hphi_2}$ then $\CG_{\tphi_1} = \CG_{\tphi_2}$.
\end{lemma}

\begin{proof}
Let $\eta_1\in \CG(\tphi_1)$ be a basic $\tphi_1$--geodesic. Then there exists a basic $\hphi_1$--geodesic $\hat\eta_1$ such that $\tilde p(\hat\eta_1)=\eta_1$. Since $\CG_{\hphi_1} = \CG_{\hphi_2}$ there exists a basic $\hphi_2$--geodesic $\hat\eta_1$ such that the Hausdorff distance (with respect to, say, $\hphi_1$) between $\hat\eta_1$ and $\hat\eta_2$ is bounded. Since $\tilde p$ maps basic geodesics in $\hS$ to basic geodesics in $\tS$ we have that $\eta_2 = \tilde p(\hat\eta_2)$ is a basic $\tphi_2$--geodesic. Moreover, since $\tilde{p}$ is 1-Lipschitz, we must have that $\eta_1 = \tilde p(\hat\eta_1)$ and $\eta_2 = \tilde p(\hat\eta_2)$ stay bounded Hausdorff distance apart, and hence must have the same endpoints. We have shown that $\CG_{\tphi_1}\subset \CG_{\tphi_2}$. The opposite inclusion follows by a symmetric argument. 
\end{proof}

\subsection{Developing and holonomy}\label{S:prelim develop}

As all along, suppose $\varphi \in \tHyp_c(S)$ and let $C_0 =\cone(\varphi) \subset S$ be the set of cone points, $\hphi$ the pull back metric $ \hat p^*(\varphi)$ and $\hat C_0 = \hat p^{-1}(C_0)\subset \hS$. Since $\widetilde{\dot S}$ is simply connected and $\hphi$ is locally isometric to the hyperbolic plane, there is a {\em developing map} $\widetilde{\dot S} \to \mathbb H$: an orientation preserving local isometry, unique up to post-composing by an element of $\PSL_2(\mathbb R)$.  This map also extends to the metric completion, which we continue to call a developing map and denote it by
\[ D \colon \hat S \to \mathbb H. \]
Observe that for every basic geodesic $\eta \in \CG(\hphi)$, $D(\eta)$ is a bi-infinite geodesic in $\mathbb{H}$.
\begin{lemma} \label{L:Dev surjective}
Fix any $\varphi \in \tHyp_c(S)$.  Then for any point $\zeta \in \hS$, any geodesic in $\mathbb H$ through $D(\zeta)$ is the $D$--image of a basic $\hphi$--geodesic in $\hS$ through $\zeta$.  In particular, $D$ is surjective and every geodesic in $\mathbb H$ is the $D$--image of a basic $\hphi$--geodesic.
\end{lemma}

\begin{proof}
First suppose $\zeta$ is any point in $\hS \smallsetminus\hC_0$.  For every unit tangent vector in ${\bf v} \in T_\zeta^1 \hS$ we can find a basic $\hat\varphi$--geodesic $\eta_{\bf v}$ through $\zeta$ tangent to ${\bf v}$ (see Lemma~\ref{L:characterizing basics}).  Observe that the derivative $dD_\zeta$ is an isometry, and that $D(\eta_{\bf v})$ is a geodesic through $D(\zeta)$ tangent to $dD_\zeta({\bf v})$.  Since $dD_\zeta$ is an isometry, every geodesic in $\mathbb H$ through $D(\zeta)$ is the image of some basic $\hphi$--geodesic in $\hS$. 
Essentially the same argument works if $\zeta \in \hC_0$: in this case, however, the derivative $dD_\zeta$ restricted to the unit vectors $T_\zeta^1 \hS$ in the tangent cone $T_\zeta \hS$ at $\zeta$ maps onto $T_{D(\zeta)}^1(\mathbb H)$ (in fact, the restriction of the derivative $dD_\zeta \colon T_\zeta^1\hS \to T_{D(\zeta)}^1\mathbb H$ is the universal cover).  This proves the first claim of the lemma.  The second claim is immediate from the first.
\end{proof}

Let $\mathcal K \subseteq \hS^1_\infty$ be the subset of the Gromov boundary of $\hS$ consisting of end points of a basic $\hphi$--geodesics. We extend the map $D$ to the union 
\[ \bar D \colon \hS \cup \mathcal K \to \overline{\mathbb H},\]
so that $\bar D(\mathcal K) \subset \partial \mathbb H$.  To see how this is done, let $k \in \mathcal K$, and define $\bar D(k)$ to be the forward end point of $D(\eta)$ where $\eta$ is any oriented basic $\hphi$--geodesic whose forward end point is $k$. This is well-defined because any two forward asymptotic basic $\hphi$--geodesics will have forward asymptotic $D$--images.  

\begin{corollary} \label{C:extendD}
The restriction $\bar D|_{\mathcal K} : \mathcal K \to \partial \mathbb H$ is surjective. 
\end{corollary}
\begin{proof} For every $y \in \partial \mathbb{H}$ one can choose an oriented geodesic $\eta'$ that has $y$ as its forward end point. By Lemma~\ref{L:Dev surjective}, there is a basic $\hphi$--geodesic $\eta$ with $D(\eta) =\eta'$. Then $\bar D(k)=y$ where $k$ is the forward end point of $\eta$.
\end{proof}

Precomposing $D$ with any covering transformation $\gamma \in \pi_1\dot S$ acting on $\hS$, gives another developing map to $\mathbb H$ that must therefore differ from $D$ by post-composing with an element in $\PSL_2(\mathbb{R})$ denoted $\rho(\gamma)$.  The assignment $\gamma \mapsto \rho(\gamma)$ thus defines the {\em holonomy homomorphism} $\rho \colon \pi_1\dot S \to \PSL_2(\mathbb R)$.  That is, $\rho$ is defined by
\[ D(\gamma \cdot x) = \rho(\gamma) \cdot D(x),\]
for all $x \in \hat S$ and $\gamma \in \pi_1 \dot S$.
We let $\Gamma = \rho(\pi_1\dot S)$, and note that since $D$ is $\rho$--equivariant, it descends to a continuous map $q \colon S \to \mathcal O = \mathbb H/\Gamma$ making the following diagram commute:
\[ \xymatrix{ \hat S \ar[r]^{D} \ar[d]_{\hat p} & \mathbb H \ar[d] \\
S \ar[r]^{q} & \mathcal O.}\]
Here $\mathcal O = \mathbb H/\Gamma$ is just the quotient topological space, {\em unless} $\Gamma$ is discrete, in which case it admits the structure of a hyperbolic $2$--orbifold.  In any case, we note that since $S$ is compact and $q$ is continuous, $\mathcal O$ is compact.  That is, $\Gamma$ is a cocompact subgroup.



\subsection{Topological conjugacy}\label{S:top conjugacy}

Recall that the action of $\Isom(\mathbb H)$ on $\mathbb H$ extends uniquely to an action on $\overline {\mathbb H} = \mathbb H \cup \partial \mathbb H$.  We will implicitly use this extension without further comment, and likewise for the orientation preserving index two subgroup $\PSL_2(\mathbb R) < \Isom(\mathbb H)$.

Given two subgroups $G_1, G_2<\Isom(\mathbb{H})$, we say that the homeomorphism $h:\partial\mathbb{H}\to\partial\mathbb{H}$ {\em topologically conjugates $G_1$ to $G_2$} if for all $\gamma_1\in G_1$ there exists $\gamma_2\in G_2$ such that
$$h\gamma_1 h^{-1}(x) = \gamma_2(x)$$
for all $x\in \partial\mathbb{H}$.  Given two hyperbolic cone metrics  $\varphi_1$ and $\varphi_2$ with $\CG_{\tphi_1} =\CG_{\tphi_2}$ and the corresponding holonomy homomorphisms $\rho_1$ and $\rho_2$, we will see in the course of the proof that $\Gamma_1=\rho_1(\pi_1 \dot S)$ and $\Gamma_2=\rho_2(\pi_1 \dot S)$ are topologically conjugate.  In fact, we will prove that the homomorphisms are topologically conjugate (see Proposition~\ref{P:holonomy conjugation}), meaning that there is a homeomorphism $h \colon \partial \mathbb H \to \partial \mathbb H$ so that for all $x\in\partial \mathbb H$ and $\gamma \in \pi_1\dot S$, 
\[ h (\rho_1(\gamma)\cdot x) = \rho_2(\gamma) \cdot h(x).\]
In this case we will say that $h$ {\em topologically conjugates $\rho_1$ to $\rho_2$}.

Part of the proof of the \currentsupport\! involves proving that quite often the homeomorphism $h$ which topologically conjugates $\rho_1$ to $\rho_2$ is in fact in $\PSL_2(\mathbb R)$.  The following will be an important ingredient.


\begin{proposition}\label{p:subgroups}
Suppose $G_1, G_2$ are two cocompact subgroups of $\Isom(\mathbb{H})$. If there exists an orientation preserving homeomorphism $h:\partial\mathbb{H}\to\partial\mathbb{H}$ topologically conjugating $G_1$ to $G_2$, then either the groups are discrete or $h\in\PSL_2(\mathbb{R})$. 
\end{proposition}

\begin{proof}
First note that $h$ must topologically conjugate the orientation preserving subgroup of $G_1$ to the orientation preserving subgroup of $G_2$.  Since discreteness of a group is equivalent to discreteness of its orientation preserving subgroup, we may assume that $G_1$ and $G_2$ are in $\PSL_2(\mathbb{R})$. 

For $i=1, 2$ let $\bar{G}_i$ denote the closure of $G_i$ and let $\bar{G}_i^{id}$ be the connected component of $\bar{G}_i$ containing the identity. Since $h$ conjugates $G_1$ to $G_2$ it also topologically conjugates $\bar{G}_1$ to $\bar{G}_2$ and thus $\bar{G}_1^{id}$ to $\bar{G}_2^{id}$. Also, since $\bar{G}^{id}_i$ is a closed subgroup of $\PSL_2(\mathbb{R})$ it is a Lie subgroup by the closed subgroup theorem. It follows from general Lie theory (by the bijection between closed connected subgroups of $\PSL_2(\mathbb{R})$ and subalgebras of $\mathfrak{s}\mathfrak{l}_2(\mathbb{R})$) that $\bar G_1^{id}$, as well as $\bar G_2^{id}$, must be one of the following groups up to conjugation in $\PSL_2(\mathbb{R})$:

\vspace{0.5cm}
\begin{tabular}{ll}
\vspace{0.5cm}
$\bullet$ $\displaystyle{T = \left\{ \left. \left( \begin{array}{cc} 1 & t \\ 0 & 1  \end{array} \right) \right| t\in\mathbb{R}   \right\}}$ & $\bullet$  $\PSO(2)$\\
\vspace{0.5cm}
 $\bullet$ $\displaystyle{U = \left\{ \left. \left( \begin{array}{cc} s & t \\ 0 & \tfrac1s  \end{array} \right) \right| s, t\in\mathbb{R}, s > 0   \right\} = \Stab\{\infty\}}$ & $\bullet$ $\PSL_2(\mathbb R)$\\

 $\bullet$ $\displaystyle{D = \left\{ \left. \left( \begin{array}{cc} s & 0 \\ 0 & \tfrac1s  \end{array} \right) \right| s > 0   \right\}} = \Stab\{0\} \cap \Stab\{\infty\}$ & $\bullet$ $\{id\}$\\

\end{tabular} 
\vspace{0.5cm}\\

Note that if $\bar{G}_1^{id} = \{id\}$ then $G_1$, and therefore also $G_2$, is discrete. So suppose $\bar{G}_1^{id}$, $\bar{G}^{id}_2$ are non-trivial. Using the facts that $\gamma\in \bar{G}_1^{id}$ fixes $p\in\partial\mathbb{H}$ if and only if $h\gamma h^{-1}\in\bar{G}_2^{id}$ fixes $h(p)$ and that the five non-trivial groups in the above list can be distinguished by considering fixed points sets and point stabilizers, we see that $\bar{G}^{id}_1$ and $\bar{G}^{id}_2$ are $\PSL_2(\mathbb{R})$--conjugate to the same group. Hence, by composing $h$ with an element in $\PSL_2(\mathbb{R})$ we can assume that $\bar{G}_1^{id}=\bar{G}^{id}_2$ and is one of the above non-trivial groups. Now, conjugation by $h$ defines an orientation preserving continuous automorphism of $\bar{G}_1^{id}$:
$$\rho \colon \bar{G}_1^{id}\to \bar{G}_1^{id}, \quad \gamma\mapsto h\gamma h^{-1}.$$
Below we will see that in all cases except $\bar{G}_1^{id}=D$ this forces $h$ to be an element of $\PSL_2(\mathbb{R})$.\\

\noindent{\em Case 1:} Suppose that $\bar{G}_1^{id}\in\{T, U, \PSL_2(\mathbb{R})\}$ and identify $\partial\mathbb{H}$ with $\mathbb{R}\cup\{\infty\}$. We can assume $h$ fixes infinity: this is necessarily the case if $\bar{G}_1^{id}\in\{T, U\}$ and if $\bar{G}_1^{id}=\PSL_2(\mathbb{R})$ we can further compose $h$ with an element of $\PSL_2(\mathbb{R})$ so that it is true. 
Note that in all cases $T< \bar{G}_1^{id}$ and $h$ must conjugate $T$ to itself (again by considering the fixed point sets). That is, $\rho\vert_T$ defines an orientation preserving continuous automorphism of $T$. Each element in $T$ corresponds to a translation $\gamma_t: x\mapsto x+t$ for some $t\in\mathbb{R}$ and, by identifying $T$ with $\mathbb{R}$ through the identification  $\gamma_t \mapsto t$, $\rho$ gives rise to an orientation preserving continuous automorphism of $\mathbb{R}$. Any such automorphism is of the form $t\mapsto \lambda t$ for some fixed $\lambda > 0$. Hence there exists $\lambda\neq 0$ such that $\rho(\gamma_t)= \gamma_{\lambda t}$ for all $\gamma_t\in T$ and letting $p=h(0)\neq\infty$ we have for each $t\in\mathbb{R}$,
$$h(t) = h(\gamma_t(0)) = h\gamma_t h^{-1}(p) =\rho(\gamma_t)(p) = \gamma_{\lambda t}(p) = \lambda t + p.$$
Hence $h$ is in $\PSL_2(\mathbb{R})$. \\

\noindent{\em Case 2:} Suppose $\bar{G}_1^{id}=\PSO(2)$ and identify $\partial\mathbb{H}$ with the circle $S^1$. By also identifying $\PSO(2)$ with $S^1$ (by sending each element to its angle of rotation) we have that $\rho: \PSO(2)\to\PSO(2)$ induces an orientation preserving continuous automorphism of $S^1$. The only such automorphism is the identity. It follows that $h\gamma h^{-1} = \gamma$ for all $\gamma\in \PSO(2)$, i.e. $h$ commutes with every rotation. Therefore $h$ is a rotation itself, and in particular, $h$ is in $\PSL_2(\mathbb{R})$ . \\


\noindent{\em Case 3:} Suppose $\bar{G}_1^{id}=D$. In fact, we will show that this cannot happen due to cocompactness of $G_1$. Suppose $\bar{G}_1\setminus D\neq \emptyset$ and let $\gamma\in \bar{G}_1\setminus D$. Then $\gamma\notin\Stab\{0\} \cap \Stab\{\infty\}$, but $\gamma \bar{G}_1^{id}\gamma^{-1} = \bar{G}_1^{id} = D = \Stab\{0\} \cap \Stab\{\infty\}$. Hence we must have that $\gamma(0)=\infty$ and $\gamma(\infty)=0$, i.e. $\gamma\in\Stab\{0, \infty\}$. It follows that $\bar{G}_1= \Stab\{0, \infty\} =\displaystyle{ \left\langle D, \left( \begin{array}{cc} 0 & -1 \\ 1 & 0  \end{array} \right) \right\rangle} = D'$. Hence either $\bar{G}_1 = D$ or $\bar{G}_1 = D'$. However, neither $D$ nor $D'$ is cocompact. To see this, set
$$ W_n=\left\{(x, y)\mid y>\frac{\vert x\vert}{n}\right\}$$
and note that $\{ W_n \mid n \in \mathbb N \}$ is an open cover of $\mathbb{H}$ (in the upper half space model) which descends to an open cover of both $\mathbb H/D$ and $\mathbb H/D'$ with no finite subcover. However $G_1$, and hence also $\bar{G}_1$, is cocompact giving a contradiction. 
\end{proof}

We note that Proposition \ref{p:subgroups} above gives no information about the conjugating homeomorphism $h$ when the groups are discrete. However, in a very special case of discrete groups we can also conclude that $h$ is in $\PSL_2(\mathbb{R})$, as the lemma below shows. Recall that a {\em triangle group} is a discrete subgroup of $\Isom(\mathbb{H})$ having a triangle $\Delta$ as its fundamental domain and such that it is generated by reflections in the sides of $\Delta$.  By a {\em Fuchsian triangle group} we mean the orientation preserving subgroup of a triangle group. 

\begin{lemma}\label{l:trianglegroup}
Suppose $G_1, G_2<\PSL_2(\mathbb{R})$ and that there exists an orientation preserving homeomorphism $h:\partial\mathbb{H}\to \partial\mathbb{H}$ conjugating $G_1$ to $G_2$. If $G_1$ is a Fuchsian triangle group, then $h\in\PSL_2(\mathbb{R})$. 
\end{lemma}

\begin{proof}
The Douady-Earle extension \cite{DouadyEarle} extends any self-homeomorphism of $\partial\mathbb{H}$ to a self-homeomorphism of the closure $\overline{\mathbb{H}}$ of the hyperbolic plane in an equivariant way: pre- or post-composing the extension with an element of $\PSL_2(\mathbb{R})$ is the same as first pre- or post-composing the homeomorphism with this element and then extending the resulting homeomorphism. Let $\bar{h}: \overline{\mathbb{H}}\to\overline{\mathbb{H}}$ be the Douady-Earle extension of $h$. The equivariance property implies that $\bar{h}\gamma = \rho(\gamma)\bar{h}$ for all $\gamma\in G_1$, where  $\rho: G_1\to G_2$ is the isomorphism $\gamma\mapsto h\gamma h^{-1}$ induced by conjugation by $h$. It follows that $\bar{h}$ descends to a homeomorphism $\mathcal{O}_1\to \mathcal{O}_2$ between the orbifolds $\mathcal{O}_i = \mathbb{H}/G_i$. However, $\mathcal{O}_1$ is a triangle orbifold---which has a trivial Teichm\"uller space---and hence the orbifold deformation must be isotopic to an isometry. It follows that the boundary map $h$ is the extension of an isometry and hence an element of $\PSL_2(\mathbb{R})$ (in fact, so is $\bar{h}$ by construction). 
\end{proof}

\section{Deformations} \label{S:deformations}

Here we prove that the construction briefly described in the introduction does indeed yield examples with the desired properties. Specifically, we prove the following.

\begin{theorem} \label{T:deforming} Suppose $\varphi_1 \in \tHyp_c(S)$, $q \colon (S,\varphi_1) \to \mathcal O_1$ is a locally isometric branched covering of an orbifold with $q(\cone(\varphi_1)) \subset \mathcal E_1$, the even order orbifold points.   Then for any (orbifold) homeomorphism $F \colon \mathcal O_1 \to \mathcal O_2$, the metric $\varphi_2$ obtained by pulling back the hyperbolic metric by $F \circ q$ has $\CG_{\tphi_1} = \CG_{\tphi_2}$ (in fact, $\CG_{\hphi_1} = \CG_{\hphi_2}$).  If $F$ is not orbifold isotopic to an isometry, then $\varphi_1 \not \sim \varphi_2$.
\end{theorem}

Before we begin the proof, we provide an example exhibiting the behavior in this theorem.

\begin{example} \label{Ex:branched example} Let $S$ be a genus $2$ surface, which we view as an octagon with opposite sides identified.  There is an order eight rotational symmetry, and we let $\rho \colon S \to S$ denote its square, which thus has order four.  The quotient $S/\langle \rho \rangle$ is a sphere, and the quotient map $q \colon S \to S/\langle \rho \rangle$ is a branched cover.  The center of the octagon, $\xi \in S$, projects to a single point $P_1 \in S/\langle \rho \rangle$, and $q$ is locally four-to-one near that point. The midpoints and vertices of the octagon account for the rest of the branch points.  The midpoints define four points of $S$ near which $q$ is locally two-to-one, mapping to two points we call $P_2$ and $P_3$, and the vertices define a single point of $S$ near which $q$ is locally four-to-one, mapping to a single point we call $P_4$; see figure~\ref{F:example branched}.

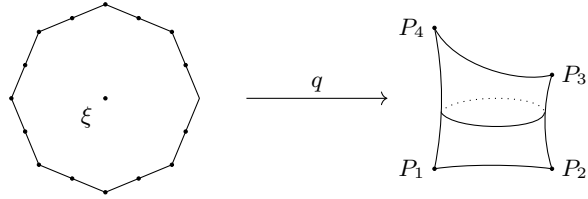
\begin{figure}[htb]
\begin{center}
  \captionsetup{width=.85\linewidth}
\begin{tikzpicture} [scale = 1.25]
\draw (1,0) -- ({cos(45)},{sin(45)}) -- ({cos(90)},{sin(90)}) -- ({cos(135)},{sin(135)}) -- ({cos(180)},{sin(180)}) -- ({cos(225)},{sin(225)}) -- ({cos(270)},{sin(270)}) -- ({cos(315)},{sin(315)}) -- (1,0);
\filldraw ({cos(45)},{sin(45)}) circle [radius = .5pt];
\filldraw ({cos(90)},{sin(90)}) circle [radius = .5pt];
\filldraw ({cos(135)},{sin(135)}) circle [radius = .5pt];
\filldraw ({cos(180)},{sin(180)}) circle [radius = .5pt];
\filldraw ({cos(225)},{sin(225)}) circle [radius = .5pt];
\filldraw ({cos(270)},{sin(270)}) circle [radius = .5pt];
\filldraw ({cos(315)},{sin(315)}) circle [radius = .5pt];
\filldraw ({sqrt(2)/4},{sqrt(2)/4+.5}) circle [radius = .5pt];
\filldraw ({sqrt(2)/4+.5},{sqrt(2)/4}) circle [radius = .5pt];
\filldraw ({-sqrt(2)/4},{sqrt(2)/4+.5}) circle [radius = .5pt];
\filldraw ({-sqrt(2)/4-.5},{sqrt(2)/4}) circle [radius = .5pt];
\filldraw ({sqrt(2)/4},{-sqrt(2)/4-.5}) circle [radius = .5pt];
\filldraw ({sqrt(2)/4+.5},{-sqrt(2)/4}) circle [radius = .5pt];
\filldraw ({-sqrt(2)/4},{-sqrt(2)/4-.5}) circle [radius = .5pt];
\filldraw ({-sqrt(2)/4-.5},{-sqrt(2)/4}) circle [radius = .5pt];
\filldraw (0,0) circle [radius = .5pt];
\node at (-.2,-.2) {\small $\xi$};
\draw[->] (1.5,0) -- (3,0);
\node at (2.25,.15) {\small $q$};
\draw (3.5,1-.25) .. controls (3.6,.5-.25) and (3.6,0-.25) .. (3.5,-.5-.25);
\draw (3.5,1-.25) .. controls (3.7,.6-.25) and (4.4,.4-.25) .. (4.75,.5-.25);
\draw (3.5,-.5-.25) .. controls (3.8,-.45-.25) and (4.4,-.45-.25) .. (4.75,-.5-.25);
\draw (4.75,.5-.25) .. controls (4.65,.2-.25) and (4.65,-.2-.25) .. (4.75,-.5-.25);
\draw (3.57,-.15) .. controls (3.7,-.35) and (4.65,-.35) .. (4.67,-.15);
\draw[dotted] (3.57,-.15) .. controls (3.7,.05) and (4.65,.05) .. (4.67,-.15);
\filldraw (3.5,.75) circle [radius = .5pt];
\filldraw (3.5,-.75) circle [radius = .5pt];
\filldraw (4.75,-.75) circle [radius = .5pt];
\filldraw (4.75,.25) circle [radius = .5pt];
\node[left] at (3.5,-.75) {\small $P_1$};
\node[right] at (4.75,-.75) {\small $P_2$};
\node[right] at (4.75,.25) {\small $P_3$};
\node[left] at (3.5,.75) {\small $P_4$};
\end{tikzpicture}
\caption{A genus $2$ surface obtained by identifying opposite sides admits an obvious order-eight symmetry, and $\rho$ is its order-four square. The branched cover $q \colon S \to S/\langle \rho \rangle$ is branched over the four points $P_1,P_2,P_3,P_4$, and branches at all the points indicated on the octagon, which are identified in $S$ to $6$ points.  We consider $S/\langle \rho \rangle$ as a hyperbolic orbifold with $P_1,P_2,P_3,P_4$ as orbifold points of orders $2,2,2,4$, respectively.} \label{F:example branched}
\end{center}
\end{figure}

Now, if we were to impose a hyperbolic orbifold structure on $S / \langle \rho \rangle$ so that $P_1,P_2,P_3,P_4$ were orbifold points with orders $4,2,2,4$, respectively, then the pull back metric via $q$ would be an honest hyperbolic structure and $q$ would be an orbifold covering map.  Instead, we make $S/\langle \rho \rangle$ into a hyperbolic orbifold $\mathcal O_1$ so that the orders of $P_1,P_2,P_3,P_4$ are $2,2,2,4$, instead.  Pulling back the metric via $q$ now gives $S$ a hyperbolic cone metric $\varphi_1$ on $S$ so that the cone angle of $\xi$ is $4\pi$, and $q$ is a branched cover.  We may deform this hyperbolic orbifold by a homeomorphism to another hyperbolic orbifold $F \colon \mathcal O_1 \to \mathcal O_2$, and we obtain a new hyperbolic cone metric $\varphi_2$ by pulling this metric back via $F \circ q$.  According to Theorem~\ref{T:deforming}, $\CG_{\tphi_1} = \CG_{\tphi_2}$.
\end{example}

Delaunay cell decompositions provide a convenient tool for analyzing hyperbolic orbifolds (and elsewhere below), and so we recall their construction in a special case of interest for us.
Let $G < \PSL_2(\mathbb R)$ be a discrete, cocompact subgroup and assume that it contains an even order elliptic element.  Let $\tCE \subset \mathbb H$ denote the $G$--invariant set of fixed points of elliptic elements of even order.  The associated {\em Delaunay cell decomposition} $\tilde \Delta(\tCE)$ is the unique $G$--invariant cell decomposition of $\mathbb H$ whose $2$--cells are compact, convex polygons with vertices in $\tCE$ characterized by the property that their circumscribing circles contain no points of $\tCE$ in their interiors and so that each $2$--cell is the convex hull of its vertices which are all contained in its circumscribing circle; see e.g. \cite{BTV-Delaunay}.  We collect some observations about this in the next lemma.

\begin{lemma} \label{L:nice Delaunay} Suppose $G < \PSL_2(\mathbb R)$ is a cocompact discrete subgroup containing an even order elliptic element and $\tCE$ is the set of fixed points of even order elliptic elements.  Then
\begin{enumerate}
\item The $1$--skeleton $\tilde \Delta(\tCE)^{(1)}$ is a  $G$--invariant union of biinfinite geodesics through $\tCE$.
\item If $H < G$ is a nontrivial, maximal, cyclic subgroup, and $H$ has odd order, then $H$ fixes a point in the interior of a $2$--cell of $\tilde \Delta(\tCE)$.  Moreover, each $2$--cell contains at most one elliptic fixed point in its interior.
\end{enumerate}
\end{lemma}
Note in condition $2$, since the even order elliptic elements fix vertices, any elliptic fixed point in the interior of a $2$--cell is necessarily the fixed point of an odd order elliptic subgroup.
\begin{proof} $\tilde \Delta(\tCE)^{(1)}$ is a $G$--invariant union of geodesic segments.  Given any $\zeta \in \tCE$ there is a unique elliptic element $\tau_\zeta \in G$  of order two fixing $\zeta$.  For any edge $e$ of $\tilde \Delta(\tCE)^{(1)}$ having $\zeta$ as an endpoint, $\tau_\zeta(e)$ is also an edge having $\zeta$ as an endpoint, and together $e \cup \tau_\zeta(e)$ is a geodesic segment (because it makes angle $\pi$ on both sides at $\zeta$).  Applying the same reasoning at the other endpoints of $e$ and $\tau_\zeta(e)$, and repeating recursively, we see that the unique biinfinite geodesic containing $e$ is a union of edges of $\tilde \Delta (\tCE)$, proving the first statement.

For the second statement, observe that the unique fixed point $\xi$ cannot be in the $1$--skeleton.  To see this, note first that if $\xi$ were a vertex, then $H$ would contain an even order element, and hence $|H|$ would be even, a contradiction.  On the other hand, if $\xi$ were on the interior of an edge $e$, then by $G$--invariance of $\tilde \Delta(\tCE)$, $H \cdot e$ would consist of more than one edge containing $\xi$, meaning that $\xi$ is a vertex, another contradiction.  Therefore, $\xi$ is necessarily in the interior of some $2$--cell $\sigma$.  Since $H$ fixes $\xi$, it fixes $\sigma$, and hence the (unique) circumcircle of $\sigma$ is invariant by $H$ and thus has $\xi$ as its center.  Any other elliptic subgroup with a unique fixed point in the interior of $H$ must also have its fixed point as the center of this circumcenter, and hence is contained in $H$.  That is, there is a unique elliptic fixed point in the interior of $\sigma$, as required.
\end{proof}

We will use the Delaunay cell structure to construct maps from $\mathbb H$ to itself.  To describe the construction of the maps, we first recall that given any two $k$--simplices $\sigma_1,\sigma_2 \in \mathbb H$ and a bijection between the vertices $\sigma_1^{(0)} \to \sigma_2^{(0)}$, there is a natural way of extending this to a map $\sigma_1 \to \sigma_2$ as follows.  In the hyperboloid model, the vertices on the hyperboloid represent linearly independent vectors and the bijection between the vertices of $\sigma_1$ and $\sigma_2$ determines a linear isomorphism between their respective spans. Restricting this map to the simplex $\sigma_1$ and radially projecting back to the hyperboloid defines the {\em canonical map} $\sigma_1 \to \sigma_2$ extending the bijection on the vertices. We observe that (1) pre- and post-composing a canonical map with isometries of $\mathbb H$ is again a canonical map, and (2) the restriction of a canonical map to a face is again a canonical map (c.f.~\cite[Section~11.4]{Ratcliffe}).

Next suppose $\mathcal T_1,\mathcal T_2$ are {\em geometric triangulations} of the hyperbolic plane, a pair of hyperbolic orbifolds, or hyperbolic (cone) surfaces: that is, these are $\Delta$--complex structures (in the sense of \cite{Hatcher}) so that each triangle of either structure admits a locally isometric parameterization by a geodesic triangle in $\mathbb H$.  A cellular homeomorphism is a {\em canonical map} if its restriction to each simplex of $\mathcal T_1$ is a canonical map to a simplex of $\mathcal T_2$; or more precisely, the lift via the locally isometric parameterizations of the restriction to each simplex is a canonical map.  Every cellular homeomorphism is isotopic through cellular maps to a canonical map---indeed, it is straightforward to construct the isotopy inductively over the skeleta.

\begin{lemma} \label{L:Delaunay maps}
Let $G_1,G_2 < \PSL_2(\mathbb R)$ be two discrete, cocompact groups, containing even order elliptic elements, and  $\rho \colon G_1 \to G_2$ an isomorphism.  If $\tCE_1,\tCE_2 \subset \mathbb H$ are the sets of even order elliptic fixed points, then there is a $G_2$--invariant cell structure $\tilde \Delta_2$ with $\tilde \Delta_2^{(0)} = \tCE_2$ and a $\rho$--equivariant cellular homeomorphism $\tilde F \colon (\mathbb H,\tilde \Delta_1) \to (\mathbb H,\tilde \Delta_2)$ where $\tilde \Delta_1 = \tilde \Delta(\tCE_1)$.  Moreover, $\tilde \Delta_2$ satisfies the same conclusions as in Lemma~\ref{L:nice Delaunay}.

In fact, there are subdivisions of the cell structures to triangulations whose vertex set is the set of all elliptic fixed points so that the map is an equivariant canonical map.  The descent $F \colon \mathbb H/G_1 \to \mathbb H/G_2$ of $\tilde F$ is likewise a canonical map with respect to geometric triangulations that have exactly one edge incident to each odd-order orbifold point.
\end{lemma}
\begin{proof}
Set $\tilde \Delta_1 = \Delta(\tCE_1)$.  By Lemma~\ref{L:nice Delaunay} part 1, there is a set $\Theta_1$ of biinfinite geodesics so that
\[ \tilde \Delta_1 = \bigcup_{\eta \in \Theta_1} \eta.\]
For any $\zeta \in \tilde \Delta_1^{(0)}$ there is a maximal elliptic subgroup $E_1(\zeta) < G_1$ of even order, fixing $\zeta$.  Let
\[ \Theta_1(\zeta) = \{\eta \in \Theta_1 \mid \zeta \in \eta \}\]
and observe that $\Theta_1(\zeta)$ is $E_1(\zeta)$--invariant.

The isomorphism $\rho$ is given as a topological conjugacy by a homeomorphism
\[ h \colon \partial \mathbb H \to \partial \mathbb H,\]
since the isomorphism is a quasi-isometry and thus extends to an equivariant homeomorphism of the boundaries at infinity.  The homeomorphism $h$ defines a straightening map $\eta \mapsto h_*(\eta)$ that sends any biinfinite geodesic $\eta$ to the unique geodesic with endpoints $h(\partial \eta)$.  Since $h$ is $\rho$--equivariant, so is $h_*$. We note that since $h$ is a homeomorphism, two geodesics $\eta$, $\eta'$ intersect in $\mathbb{H}$ if and only if the geodesics $h_*(\eta)$ and $h_*(\eta')$ intersect as well. This is because $h$ must preserve the linking (or lack thereof) of $\partial \eta$ and $\partial \eta'$.

We will define $\tilde{F}$ inductively over the skeleta, and in the process, we define the cell structure $\tilde \Delta_2$.  For any vertex $\zeta \in \tilde \Delta_1^{(0)}$, we define $\tilde F(\zeta)$, to be the unique fixed point of $\rho(E_1(\zeta)) < G_2$.  By construction, the restriction of $\tilde F$ to $\tilde \Delta_1^{(0)}$ is $\rho$--equivariant.  We claim that for any $\zeta \in \tilde \Delta_1^{(0)}$, we have
\[ \tilde F(\zeta) \in \bigcap_{\eta \in \Theta_1(\zeta)} h_*(\eta). \]
To see this, let $\tau_\zeta \in E_1(\zeta)$ be the unique element of order $2$, and note that $\tau_\zeta(\eta) = \eta$ for all $\eta \in \Theta_1(\zeta)$, or for any geodesic passing through $\zeta$.  Since $\rho$ is induced by conjugation by $h \colon \partial \mathbb H \to \partial \mathbb H$, restricting to $\partial \mathbb H$ we have
\[ \tau_{\tilde F(\zeta)} = \rho(\tau_\zeta) =  h \circ \tau_\zeta \circ h^{-1}.
\]
Therefore, for any $\eta$ passing through $\zeta$, we have
\begin{equation} \label{Eq:baby concurrence} \tau_{\tilde F(\zeta)}(\partial h_*(\eta)) = \tau_{\tilde F(\zeta)}(h(\partial \eta)) = h(\tau_\zeta(h^{-1}(h(\partial \eta)))) = h(\tau_\zeta(\partial \eta)) = h(\partial \eta) = \partial h_*(\eta).
\end{equation}
That is, $h_*(\eta)$ is invariant by $\tau_{\tilde F(\zeta)}$ and consequently $\tilde F(\zeta)$ is on $h_*(\eta)$, as required.

For every edge $e$ of $\tilde \Delta_1$ connecting a pair of vertices $\zeta,\zeta' \in \tilde \Delta_1^{(0)}$, define $\tilde F$ to send $e$ by the canonical map to the geodesic segment between $\tilde F(\zeta)$ and $\tilde F(\zeta')$.  We observe that if $\eta \in \Theta_1$ is the unique geodesic containing $e$, then $h_*(\eta)$ contains $\tilde F(e)$.  Moreover, $h_*(\eta)$ preserves the order of the vertices along $\eta$. To see this, orient $\eta$ and suppose that $\zeta, \zeta'\in\eta$ such that $\zeta$ is encountered before $\zeta'$ along $\eta$ with respect to this orientation. Let $\tau_{\zeta}, \tau_{\zeta'} \in G_1$ be order two elliptic elements fixing $\zeta, \zeta'$ respectively. Note that $\tau_{\zeta}$ and $\tau_{\zeta'}$ both fix $\eta$ (but interchange its endpoints) and the hyperbolic element $\tau_{\zeta'}\circ\tau_{\zeta}$ has $\eta$ as its axis and translates along it in the positive direction; see Figure~\ref{F:eta hyperbolic}. It follows that $\rho(\tau_{\zeta'}\circ\tau_{\zeta})$ has $h_*(\eta)$ as its axis and translates along it in the positive direction. Since $\rho(\tau_{\zeta'}\circ\tau_{\zeta}) = \rho(\tau_{\zeta'})\circ\rho(\tau_{\zeta})$ and $\rho(\tau_{\zeta})$ and $\rho(\tau_{\zeta'})$ are elliptic elements fixing $\tilde F(\zeta)$ and $\tilde F(\zeta')$ it follows that $F(\zeta)$ is encountered before $F(\zeta')$ along $h_*(\eta)$ and so the order is preserved. In particular, $\tilde F$ maps each $\eta \in \Theta_1$ to $h_*(\eta)$ by a map which is the canonical map on each segment between consecutive vertices of $\tilde \Delta_1$.  For any two geodesics $\eta,\eta'$, we have $\eta \cap \eta' = \emptyset$ if and only if $h_*(\eta) \cap h_*(\eta') = \emptyset$ and it follows that $\tilde F$ is injective on the $1$--skeleton $\tilde \Delta_1^{(1)}$.   
Since the map on the zero--skeleton is $\rho$--equivariant and the maps on the edges are canonical maps, it follows that $\tilde F$ is a $\rho$--equivariant homeomorphism onto its image.
Consequently, $\tilde F(\tilde \Delta_1^{(1)})$ is the $1$--skeleton of a cell structure $\tilde \Delta_2$ on the image.

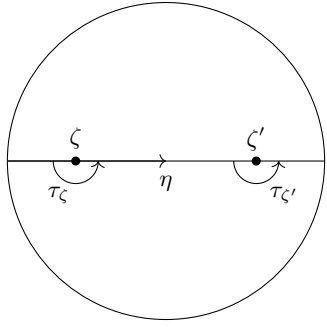
\begin{figure}[htb]
\begin{center}
  \captionsetup{width=.85\linewidth}
\begin{tikzpicture} [scale = 1.5]
\draw (0,0) circle [radius = 40pt];
\draw (-1.4,0) -- (1.4,0);
\draw[->] (-1.4,0) -- (0,0);
\filldraw (-.8,0) circle [radius=1pt];
\filldraw (.8,0) circle [radius=1pt];
\draw [->,domain=180:360] plot ({.8+.2*cos(\x)}, {.2*sin(\x)});
\draw [->,domain=180:360] plot ({-.8+.2*cos(\x)}, {.2*sin(\x)});
\node at (0,-.2) {\small $\eta$};
\node at (-.8,.2) {\small $\zeta$};
\node at (.8,.2) {\small $\zeta'$};
\node at (-.95,-.3) {\small $\tau_\zeta$};
\node at (1.05,-.3) {\small $\tau_{\zeta'}$};
\end{tikzpicture}
\caption{The composition of $\tau_\zeta$ and $\tau_{\zeta'}$ is a hyperbolic isometry with axis $\eta$.} \label{F:eta hyperbolic}
\end{center}
\end{figure}

Evidently, $\tilde \Delta_2$ has property (1) of Lemma~\ref{L:nice Delaunay}. Because of this $\rho$--equivariance, $\tilde F$ preserves the cyclic ordering of edges incident to each vertex. This implies that cycles that bound a $2$--cell in $\tilde \Delta_1$ are sent to cycles that bound a $2$--cell in $\tilde \Delta_2$. This suggests how we may extend $\tilde F$ to a map on all of $\mathbb{H}$. The details of this follow shortly.

To see that property (2) also holds, we note that the boundary of any $2$--cell containing an elliptic fixed point in its interior is sent to the boundary of a $2$--cell of $\tilde \Delta_2$ by $\rho$--equivariance, and that this defines a bijection between the $2$--cells containing elliptic fixed points in their interiors.  We also note that each $2$--cell of $\tilde \Delta_2$ is a convex polygon: the $1$--skeleton is a union of biinfinite geodesics and so the interior angles are all less than $\pi$.

Finally, we subdivide each $\tilde \Delta_i$ to a triangulations $\tilde \Delta_i'$ as follows.  For each non-triangular $2$--cell of $\tilde \Delta_1$ not containing an elliptic fixed point in its interior, we note the interior is mapped disjoint from itself by any nontrivial element of $G_1$, and so the same is true for the corresponding $2$--cell of $\tilde \Delta_2$.  We subdivide the union of all such $2$--cells of $\tilde \Delta_1$ in any $G_1$--invariant way, without introducing new vertices, so that each new $1$--cell is a geodesic segment.  Convexity of the $2$--cells of $\tilde \Delta_2$ implies that there is a $G_2$--invariant subdivision of $2$--cells not containing elliptic fixed points in their interiors so that $\tilde F$ extends to a cellular homeomorphism on the subdivision.  Finally, for each $2$--cell $\sigma$ of $\tilde \Delta_1$ that contains in its interior the fixed point $\xi$ of some maximal, odd order elliptic subgroup $H_1 < G_1$, first pick a vertex of $v$ of $\sigma$ and begin the subdivision by connecting consecutive vertices of the orbit $H_1 \cdot v$ (if they are not already connected by an edge), doing so in a $G_1$--invariant way.  The cell $\sigma$ now contains an $H_1$--invariant $k$--gon with $\xi$ in its interior and $k$ polygons permuted by $H_1$, where $k = |H_1|$, unless $\sigma$ was already a $k$--gon.  Now we cone off the $k$--gon to $\xi$ and subdivide the remaining polygons in a $G_1$--invariant way.  See Figure~\ref{F:subdividing}.  

\begin{figure}[htb]
\begin{center}
  \captionsetup{width=.85\linewidth}
\begin{tikzpicture} [scale = 1.25]
\draw (1,0) -- ({cos(40)}, {sin(40)}) -- ({cos(80)}, {sin(80)}) -- ({cos(120)}, {sin(120)}) -- ({cos(160)}, {sin(160)}) -- ({cos(200)}, {sin(200)}) -- ({cos(240)}, {sin(240)}) -- ({cos(280)}, {sin(280)}) -- ({cos(320)}, {sin(320)}) -- (1,0);
\filldraw (0,0) circle [radius = .5pt];
\node at (-.1,-.1) {\small $\xi$};
\node at (-.75,.75) {\small $\sigma$};
\draw[->] (1.5,0) -- (3,0);
\node at (2.25,.2) {\small subdivide};
\draw (5.5,0) -- ({4.5+cos(40)}, {sin(40)}) -- ({4.5+cos(80)}, {sin(80)}) -- ({4.5+cos(120)}, {sin(120)}) -- ({4.5+cos(160)}, {sin(160)}) -- ({4.5+cos(200)}, {sin(200)}) -- ({4.5+cos(240)}, {sin(240)}) -- ({4.5+cos(280)}, {sin(280)}) -- ({4.5+cos(320)}, {sin(320)}) -- (5.5,0);
\draw (5.5,0)  -- ({4.5+cos(120)}, {sin(120)}) -- ({4.5+cos(240)}, {sin(240)}) -- (5.5,0);
\draw (5.5,0) -- ({4.5+cos(80)}, {sin(80)});
\draw ({4.5+cos(120)}, {sin(120)}) -- ({4.5+cos(200)}, {sin(200)});
\draw ({4.5+cos(240)}, {sin(240)}) -- ({4.5+cos(320)}, {sin(320)});
\draw (4.5,0) -- (5.5,0);
\draw (4.5,0) -- ({4.5+cos(120)}, {sin(120)}) ;
\draw (4.5,0) -- ({4.5+cos(240)}, {sin(240)}) ;
\filldraw (4.5,0) circle [radius = .5pt];
\end{tikzpicture}
\caption{Subdividing a $2$--cell $\sigma$ containing the fixed point $\xi$ of a maximal cyclic subgroup $H_1 < G_1$ of order $k = 3$ in its interior.  In this example, $\sigma$ is the $9$--gon on the left.  The subdivision is as shown on the right.} \label{F:subdividing}
\end{center}
\end{figure}
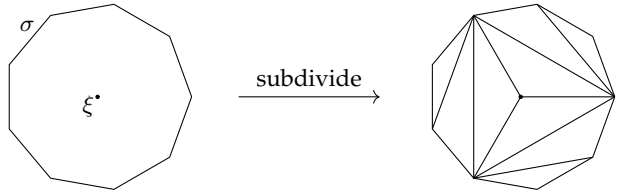

We have thus subdivided $\tilde \Delta_1$ to a $G_1$--invariant triangulation $\tilde \Delta_1'$, and the $\rho$--equivariant homeomorphism $\tilde F|_{\tilde \Delta_1}$, together with the convexity of the $2$--cells of $\tilde \Delta_2$ determines an associated subdivision to a triangulation $\tilde \Delta_2'$ so that $\tilde F|_{\tilde \Delta_1}$ extends to a canonical map with respect to $\tilde \Delta_1'$ and $\tilde \Delta_2'$.  The properties of the descent to the orbifolds has been built into the definition, completing the proof.
\end{proof}

We will also need the following property of isomorphic, cocompact, discrete groups $G_1, G_2$ as in Lemma~\ref{L:Delaunay maps}.
\begin{lemma} \label{L:Delaunay combinatorics preserved} Suppose $G_1,G_2$ are any two cocompact, discrete groups containing even order elliptic elements, $\rho \colon G_1 \to G_2$ is an isomorphism given by conjugation by $h \colon \partial \mathbb H \to \partial \mathbb H$, and suppose
\[ \tilde F \colon (\mathbb H,\tilde \Delta_1) \to (\mathbb H,\tilde \Delta_2)\]
is the cellular homeomorphism from Lemma~\ref{L:Delaunay maps}.  Then for any biinfinite geodesic $\eta$ in $\mathbb H$, $\eta$ nontrivially intersects a $k$--cell $\sigma$ of $\tilde \Delta (\tCE_1)$ if and only if the straightening $h_*(\eta)$ of $\tilde F(\eta)$ nontrivially intersects $\tilde F(\sigma)$.
\end{lemma}
\begin{proof} The fact that the lemma holds for vertices $\zeta \in \tilde \Delta_1^{(0)} = \tCE_1$ and geodesics $\eta$ that intersect $\zeta$ follows immediately from \eqref{Eq:baby concurrence} (the only thing being used in that argument was that $\zeta$ was a point on $\eta$).  

Now we show that the lemma holds for an edge $e$.  By the previous paragraph, it suffices to show that $\eta$ intersects the interior of $e$ if and only if $h_*(\eta)$ intersects the interior of $\tilde F(e)$. For this, observe that $\eta$ intersects the interior of $e$ if and only if there are geodesics $\eta_+$ and $\eta_-$ through the endpoints $\zeta_+$ and $\zeta_-$ of $e$, respectively, so that $\eta$ is in the subset of $\mathbb H$ bounded by $\eta_+$ and $\eta_-$.  Since this configuration of $\eta,\eta_-,\eta_+$ is determined by the endpoints on $\partial \mathbb H$, it follows that $\eta$ is in the subset of $\mathbb H$ bounded by $\eta_+$ and $\eta_-$ if and only if $h_*(\eta)$ is contained in the subset of $\mathbb H$ bounded by $h_*(\eta_+)$ and $h_*(\eta_-)$, and this happens (for some $\eta_+,\eta_-$) if and only if $h_*(\eta)$ intersects the interior of $\tilde F(e)$ (since $h_*(\eta_+)$ and $h_*(\eta_-)$ contain the endpoints $\tilde F(\zeta_+)$ and $\tilde F(\zeta_-)$).  

Since a geodesic that intersects a $2$--cell must also intersect its boundary, the case of $0$--cells and $1$--cells implies the case of $2$--cells.  This observation completes the proof.
\end{proof}

We are now ready to prove Theorem \ref{T:deforming}. 

\begin{proof}[Proof of Theorem~\ref{T:deforming}] Suppose $G_1,G_2 < \PSL_2(\mathbb R)$ are two cocompact discrete groups with quotient orbifolds $p_i \colon \mathbb H \to \mathcal O_i = \mathbb H/G_i$, and suppose that $F \colon \mathcal O_1 \to \mathcal O_2$ is a homeomorphism.  We also let $\CE_i \subset \mathcal O_i$ be the even order orbifold points, so that $p_i(\tCE_i) = \CE_i$.  

Now suppose $q_1 \colon (S,\varphi_1) \to \mathcal O_1$ is a locally isometric branched cover of the orbifold, branched only over even order orbifold points.  Let $q_2 = F \circ q_1 \colon S \to \mathcal O_2$ be the composition, and $\varphi_2$ be the pull back of the hyperbolic metric on $\mathcal O_2$ (so that with respect to $\varphi_2$, $q_2$ is a locally isometric branched cover). 

The homeomorphism $F \colon \mathcal O_1 \to \mathcal O_2$ induces an isomorphism $F_* \colon G_1 \to G_2$.   Lemma~\ref{L:Delaunay maps} produces a homeomorphism $\mathcal O_1 \to \mathcal O_2$ inducing $F_* \colon G_1\to G_2$ that is isotopic to $F$ (by the Dehn-Nielsen-Baer Theorem). Therefore, by applying an isotopy if necessary, we may assume that $F$ is the homeomorphism from Lemma~\ref{L:Delaunay maps}
(the isotopy lifts to $S$, and thus changing $F$ in this way results in a metric equivalent to $\varphi_2$).  
For $i=1,2$, we let $\tilde \Delta_i$ be the cell structures of $\mathbb H$ from that lemma, and $\Delta_i$ the corresponding cell structures on $\mathcal O_i$, for $i =1,2$.  We pull these cell structures back to cell structures $\Lambda = \Lambda_1=\Lambda_2$ on $S$.  Each $2$--cell of $\Lambda$ is a(n immersed) geodesic polyon with respect to either metric $\varphi_i$, whose interior is embedded. The restriction of $q_i$ to (the interior of) each such polygon is either a locally isometric homeomorphism, or else it is a $k$-to-$1$ orbifold quotient, where $k$ is an odd integer, the order of the orbifold point in the image of the interior.  This is perhaps easiest to see by lifting these cell structures back to cell structure $\hat \Lambda$ on $\hS$, which we can view as lifting the cell structures $\tilde \Delta_i$ by the developing maps $D_i \colon (\hS, \hphi_i) \to \mathbb H$ (note that $D_2 = \tilde F \circ D_1$).

Next, consider the commutative diagram, in which all maps are cellular (by construction).
\begin{center}
\begin{tikzpicture}
\node at (0,0){$\begin{tikzcd}[row sep=15]
(\hat S,\hphi_1) \ar[r, "id"] \ar[dd, "\hat p"] 		& (\hat S,\hphi_2) \ar[drr, "D_2" ]  \ar[dd, "\hat p"] 	& \quad \quad \quad & \quad \quad \quad\\
										&										& \mathbb H	\ar[from=ull, crossing over, "D_1" near end] \ar[r,"\tilde F" below]					& \mathbb H \ar[dd, "p_2"] \\
(S,\varphi_1) \ar[r, "id"] 						& (S,\varphi_2)	\ar[drr, "q_2" near start]  \\
										&						&\mathcal O_1  \ar[r,"F" below] \ar[from=uu, crossing over, "p_1"] \ar[from=ull, crossing over, "q_1" below]	&  \, \mathcal O_2.
\end{tikzcd}$};	
\end{tikzpicture}
\end{center}

By Lemma~\ref{L:same hat, same tilde}, if we can show that $\CG_{\hphi_1} = \CG_{\hphi_2}$, then we will have $\CG_{\tphi_1} = \CG_{\tphi_2}$, as required.  For this, assume we have any nonsingular geodesic $\delta$ in $(\hS,\hphi_1)$.  Then $D_1(\delta_1)$ is a biinfinite geodesic.   For any cell $\sigma$ of $\tilde \Delta_1$,  Lemma~\ref{L:Delaunay combinatorics preserved} tells us that the straightened image geodesic $\tilde F_*(D_1(\delta))$ intersects $\tilde F(\sigma)$ if and only if $D_1(\delta)$ intersects $\sigma$.  Since $\tilde F$ is cellular, this means that $\tilde F_*(D_1(\delta))$ meets the cell $\tilde F(\sigma)$ if and only if $\tilde F(D_1(\delta))$ does.  Therefore, there is a homotopy from $\tilde F(D_1(\delta))$ to $\tilde F_*(D_1(\delta))$ with bounded tracks that respects the cell structure, and thus lifts to a homotopy from $\delta$ to the $\hat \varphi_2$--geodesic $\delta'$, also preserving the cell structures (and having bounded tracks).  Note that $\delta$ meets a cell $\sigma$ of $\Lambda$ if and only if $\delta'$ does, and hence $\delta'$ is also nonsingular.  Therefore, $\CG_{\hphi_1} \subset \CG_{\hphi_2}$.  By a symmetric argument, they are equal.  Lemma~\ref{L:same hat, same tilde} completes the proof.
\end{proof}


\section{Geodesics and triangulations} \label{S:geodesics and triangulations}

Recall that $\mathcal{G}(\tphi_i)$ denotes the set of all basic $\tphi_i$--geodesics and that the maps $\partial_{\tphi_i}: \mathcal{G}(\tphi_i) \to \mathcal{G}_{\tphi_i}$, for $i=1,2$, which maps each geodesic to its endpoints are homeomorphisms. Hence, under the assumption that $\mathcal{G}_{\tphi_1}=\mathcal{G}_{\tphi_2}$, this gives rise to a homeomorphism 
$$g: \mathcal{G}(\tphi_1)\to \mathcal{G}(\tphi_2)$$
which maps each basic $\tphi_1$--geodesic $\eta$ to the basic $\tphi_2$--geodesic having the same endpoints as $\eta$. In other words, $g(\eta)$ is the $\tphi_2${\em -straightening} of the $\tphi_1$--geodesic $\eta$. Note that orienting $\eta$ naturally induces an orientation on $g(\eta)$. 

In this section, we will show that the straightening map preserves many properties of basic geodesics.  Along the way we will begin adjusting the metric $\varphi_2$ to equivalent metrics (under the assumption $\CG_{\tphi_1} = \CG_{\tphi_2}$).  A useful tool in the analysis that we also develop are triangulations of $S$ adapted to the two metrics.

\subsection{Concurrency points and partitions}\label{S:concurrency and partitions}

Recall by Lemma \ref{L:basic geodesics} that only countably many basic $\tphi_i$--geodesics pass through more than one cone point and that we denote the set of the corresponding pairs of endpoints by $\CG^2_{\tphi_1}$.  Also, given $\zeta \in \tS$, we will write
\[  \CG(\tphi_i,\zeta) = \{ \gamma \in \CG(\tphi_i) \mid \zeta \in \gamma \} \subset \CG(\tphi_i), \quad \mbox{ and } \quad \CG_{\tphi_i}(\zeta) = \partial_{\tphi_i}(\CG(\tphi_i,\zeta)) \subset \CG_{\tphi_i} \]
for each $i=1,2$. 
That is, $\CG(\tphi_i,\zeta)$ is the set of basic $\tphi_i$--geodesics that contain $\zeta$ and $\CG_{\tphi_i}(\zeta)$ the set of all pairs of endpoints at infinity of such geodesics.  For $\zeta \in \hS$, will use the corresponding notation $\CG(\hphi_i,\zeta)$ and $\CG_{\hphi_i}(\zeta)$ for the analogously defined subsets of $\CG(\hphi_i)$ and $\CG_{\hphi_i}$, respectively, for $i=1,2$.

Using the notion of {\em chains}, the following is proved in \cite{BL} for Euclidean cone metrics, but the proof is essentially identical in the case of hyperbolic cone metrics.
\begin{proposition} \label{P:chain consequence} Suppose $\varphi_1,\varphi_2 \in \tHyp_c(S)$ with $\CG_{\tphi_1} = \CG_{\tphi_2}$.  There is a countable, $\pi_1S$--invariant set $\Omega \subset \CG_{\tphi_1} = \CG_{\tphi_2}$ containing $\CG^2_{\tphi_1} \cup \CG^2_{\tphi_2}$ with the following property.  For any $\zeta_1 \in \cone(\tphi_1)$ there exists $\zeta_2 \in \cone(\tphi_2)$ so that
\[ \CG_{\tphi_1}(\zeta_1) - \Omega = \CG_{\tphi_2}(\zeta_2) - \Omega.\]
Moreover, sending $\zeta_1$ to $\zeta_2$ determines a $\pi_1S$--equivariant bijection $\cone(\tphi_1) \to \cone(\tphi_2)$.
\end{proposition}
\begin{proof}[Idea of the proof.] A {\em chain} associated to any cone point $\zeta_i \in \cone(\tphi_i)$, for $i=1,2$, is a bi-infinite sequence ${\bf x} = \{x_k\}_{k\in \mathbb Z} \subset S^1_\infty$ so that for any $k \in \mathbb Z$, $x_{k-1},x_k,x_{k+1}$ is a counter-clockwise ordered triple in $S^1_\infty$, and so that
\[ \{x_k,x_{k+1}\} \in \CG_{\tphi_1}(\zeta) \setminus (\CG_{\tphi_1}^2 \cup \CG_{\tphi_2}^2).\]
From the fact that the cone angles are greater than $2\pi$ at $\zeta_i$, it is straightforward to see that for any $x \in S^1_\infty$ with $x_{k+1},x,x_{k-1}$ counterclockwise ordered, $\{x,x_k\}$ is {\em not} in $\CG_{\tphi_i}$.    Any biinfinite sequence $\{x_k\} \subset S^1_\infty$ with the property that $\{x_k,x_{k+1}\} \subset \CG_{\tphi_i} \setminus (\CG_{\tphi_1}^2 \cup \CG_{\tphi_2}^2)$ and the property for triples $x_{k+1},x,x_{k-1}$ just described is in fact a chain associated to a unique cone point of $\tphi_i$.   See \cite[Proposition~4.1]{BL}.

Because chains depend only on $\CG_{\tphi_1} = \CG_{\tphi_2}$ and $\CG_{\tphi_1}^2 \cup \CG_{\tphi_2}^2$, we get functions from the set of chains to both $\cone(\tphi_1)$ and $\cone(\tphi_2)$.  In fact, there is an equivalence relation on chains, describable in terms of $\CG_{\tphi_1} = \CG_{\tphi_2}$ and $\CG_{\tphi_1}^2 \cup \CG_{\tphi_2}^2$, so that two chains are equivalent if and only if they are sent to the same cone point of $\cone(\tphi_i)$, for each $i=1,2$ (see \cite[Lemma~4.4]{BL}).  Therefore, the set of equivalence classes of chains is in a bijective correspondence with $\cone(\tphi_1)$ and $\cone(\tphi_2)$.  This in turn determines a bijection $\cone(\tphi_1) \to \cone(\tphi_2)$.  The chain condition is $\pi_1S$--invariant, and the map to cone points is equivariant, from which we deduce that the bijection $\cone(\tphi_1) \to \cone(\tphi_2)$ to $\pi_1S$--equivariant.  For any $\zeta_i$, $i=1,2$, the set of pairs of endpoints of geodesics in $\CG_{\tphi_i}(\zeta_i)$ which are not consecutive terms in some chain is countable, and thus taking $\Omega$ to be the union of all such pairs produces the countable set $\Omega$ containing $\CG_{\tphi_1}^2 \cup \CG_{\tphi_2}^2$. See \cite[Section 4]{BL} for more details.
\end{proof}

Going forward we let $C_0 = \cone(\varphi_1)$ and $\tC_0 = p^{-1}(C_0) = \cone(\tphi_1)$.  For the remainder of this subsection, we assume $\CG_{\tphi_1} = \CG_{\tphi_2}$, and let $\Omega$ be the set from Proposition~\ref{P:chain consequence}.
The $\pi_1S$--equivariant bijection from that proposition can be extended to a $\pi_1S$--equivariant homeomorphism, which is thus the lift of a homeomorphism $S \to S$, isotopic to the identity.  We replace $\varphi_2$ with its pull back by this homeomorphism, which does not change the equivalence class, but allows us to say that $\cone(\tphi_1) =\tC_0 =  \cone(\tphi_2)$. More precisely, we can promote the conclusion of the proposition to say that for all $\zeta \in \tC_0$, we have
\[ \CG_{\tphi_1}(\zeta) - \Omega = \CG_{\tphi_2}(\zeta) - \Omega. \]

The conclusion of Proposition~\ref{P:chain consequence} may hold for {\em other} points $\zeta$, and adjusting $\varphi_2$ to take this into account will be useful later.  Say that $\zeta_1 \in \tS$ is a {\em $(\tphi_1,\tphi_2,\Omega)$--concurrence point} (or briefly {\em concurrence point}) if there exists $\zeta_2 \in \tS$ so that
\[ \CG_{\tphi_1}(\zeta_1) - \Omega = \CG_{\tphi_2}(\zeta_2) - \Omega\]
where $\Omega$ is as in Proposition \ref{P:chain consequence}.   In particular, the proposition says that every point in $\tC_0$ is a concurrence point.  We note that since $\Omega$ and $\CG_{\tphi_1} = \CG_{\tphi_2}$ are $\pi_1S$--invariant, if $\zeta$ is a concurrence point, then so is $\gamma \cdot \zeta$, for any $\gamma \in \pi_1S$.  We say $\tC \subset \tS$ is a {\em concurrency set} (for $(\tphi_1,\tphi_2,\Omega)$), if $\tC$:
\begin{enumerate}
\item consists of $(\tphi_1,\tphi_2,\Omega)$--concurrence points 
\item contains $\tC_0$,
\item is $\pi_1S$--invariant, and
\item contains only finitely many $\pi_1S$--orbits of points.
\end{enumerate}
If $\tC$ is a concurrency set, then we let $C = p(\tC)$, and also call $C \subset S$ a concurrency set (by $\pi_1S$--invariance, note that $\tC = p^{-1}(C)$).

As we did with the cone points, note that for any concurrency set $\tC$ we can adjust $\varphi_2$ by a homeomorphism isotopic to the identity such that for all $\zeta\in\tC$ we have 
\[ \CG_{\tphi_1}(\zeta) - \Omega = \CG_{\tphi_2}(\zeta) - \Omega. \]
We will say that $\varphi_2$ has been $C$--{\em normalized} with respect to $\varphi_1$.  When the reference to $\varphi_1$ is obvious, we will simply say that $\varphi_2$ has $C$--normalized. We summarize the above discussion for later reference: 

\begin{lemma}\label{l:convention}
Assume $\mathcal{G}_{\tphi_1}=\mathcal{G}_{\tphi_2}$ and let $\tC$ be any concurrency set for $(\tphi_1,\tphi_2,\Omega)$. Then $\tphi_2\sim\tphi_2'$ for a $C$--normalized $\tphi_2'\in\tHyp_c(S)$; that is, $\tphi_2'$ satisfies
\[ \CG_{\tphi_1}(\zeta) - \Omega = \CG_{\tphi_2'}(\zeta) - \Omega \]
for all $\zeta\in \tC$. 
\qed
\end{lemma}

The upshot of this lemma is that, in order to prove the \currentsupport\!\!, we can assume that we have adjusted $\varphi_2$ such that it is $C$--normalized, for some concurrency set $C$. 
In fact, for most of the proof of the \currentsupport\!\!, we can assume our concurrency set is just the set of cone points, $C_0= \cone(\varphi_1) = \cone(\varphi_2)$. However, later we will find additional concurrency points, and rather than returning to all of the setup and preliminaries, we will proceed with an arbitrary concurrency set.\\

Now, let $\tC\subset \tS$ be a concurrency set for $(\tphi_1,\tphi_2,\Omega)$, assume $\varphi_2$ is $C$--normalized, and let $\eta$ be a $\tphi_i$--geodesic. As before, we say that $\eta$ is singular if it contains a cone point and nonsingular otherwise. We say that $\eta$ is {\em $C$--singular} if it contains a point of $\tC$, and {\em non-$C$--singular} otherwise. Let
$$\CG(\tphi_i,\Omega)  = \partial_{\varphi_i}^{-1}(\CG_{\tphi_i} - \Omega) \subset \CG(\tphi_i)$$ 
for $i=1, 2$.
The homeomorphism $g:\mathcal{G}(\tphi_1)\to\mathcal{G}(\tphi_2)$ restricts to a homeomorphism 
$$g \colon \CG(\tphi_1,\Omega) \to \CG(\tphi_2,\Omega)$$ 
Note that by Proposition~\ref{P:chain consequence}, $\eta\in \CG(\tphi_1,\Omega)$ is nonsingular if and only if $g(\eta)$ is nonsingular.  Moreover, $\eta$ passes through $\zeta\in \tC$ if and only if $g(\eta)$ does (assuming, as we have, that $\varphi_2$ is $C$--normalized) by Lemma \ref{l:convention}.  It is worth pointing out that, a priori, $\CG(\tphi_1,\Omega)$ might not contain all nonsingular $\tphi_1$--geodesics because of information loss in $\Omega$:  there is the potential for (at most countably many) nonsingular $\tphi_1$--geodesics to be missing from $\CG(\tphi_1,\Omega)$ (though this ultimately turns out not to be the case).  Below we will need to approximate basic geodesics by sequences of nonsingular geodesics in $\CG(\tphi_1,\Omega)$ and we note that we can always do this: any basic geodesic is, by definition, a limit of nonsingular geodesics and note that the proof of Lemma \ref{L:characterizing basics} actually shows that there are uncountably many sequences of nonsingular geodesics approximating each basic geodesic. 

Let $\eta\in \mathcal{G}(\tphi_i)$ be a basic $\tphi_i$--geodesic and $\zeta \in \tC$. Recall that $\eta$ divides $\tS$ into two half-planes, the positive (left) half plane $\mathcal H^+(\eta)$ and the negative (right) half plane $\mathcal H^-(\eta)$.
If $\zeta\in \tC$ is disjoint from $\eta$, we say that $\zeta$ lies {\em to the left} of $\eta$ if it belongs to $\mathcal{H}^+(\eta)$ and that it lies {\em to the right} if it belongs to $\mathcal{H}^-(\eta)$.  If $\zeta \in \eta$, then we say that $\zeta$ is {\em on} $\eta$.  If $\zeta \in \tC_0$ and is on $\eta$, then $\eta$ must make angle $\pi$ on exactly one side at $\zeta$: if $\eta$ makes angle $\pi$ on the left of $\zeta$ (respectively, on the right of $\zeta$), and we say that $\zeta$ is {\em on and to the right of $\eta$} (respectively, {\em on and to the left of $\eta$}). If $\zeta \in \tC \smallsetminus \tC_0$ and its on $\eta$, then $\eta$ makes angle $\pi$ on both sides at $\zeta$, and we say that $\zeta$ is {\em on and splits $\eta$}.  Therefore, $\eta \in \CG(\tphi_i)$ defines a partition of $\tC$ into five sets depending on whether a point is to the left of $\eta$, to the right of $\eta$, on and to the left of $\eta$, on and to the right of $\eta$, or on and splits $\eta$ (where any of the last three could determine empty sets). Note that while the labels left, right, $+$, $-$ are determined by arbitrarily choosing an orientation on $\eta$, the partition itself is defined by the unoriented geodesic.
We will see that $g$ preserves this partition (c.f.~\cite[Lemma~12]{oldpaper}). We start with the following special case.

\begin{lemma}\label{l:partition} 
Suppose $\mathcal{G}_{\tphi_1} = \mathcal{G}_{\tphi_2}$, $\Omega$ is as in Proposition~\ref{P:chain consequence}, $\tC\subset\tS$ is a concurrency set, and $\varphi_2$ is $C$--normalized. Each nonsingular $\eta\in\CG(\tphi_1,\Omega)$ defines the same partition of $\tC$ as $g(\eta)$ does. 
\end{lemma}

\begin{proof}
Let $\eta\in \CG(\tphi_1,\Omega)$ be nonsingular and fix an orientation on $\eta$ and recall that this induces an orientation on $g(\eta)$. 
Let $\zeta\in\tC$. Note that $\zeta$ is on $\eta$ if and only if it is on $g(\eta)$ by the definition of concurrency points and the assumption that $\varphi_2$ is $C$--normalized.  Moreover, it cannot be a cone point since $\eta$ is nonsingular, and hence $\zeta$ splits both $\eta$ and $\zeta$.  

So suppose, without loss of generality, that $\zeta$ is to the left of $\eta$, i.e.~$\zeta\in\mathcal{H}^+(\eta)$. We claim that there exists a geodesic $\delta\in \CG(\tphi_1,\Omega)$ passing through $\zeta$ and contained in $\mathcal{H}^+(\eta)$. 

To see that such $\delta$ exists, let $\mu$ be a minimal length geodesic path from $\zeta$ to $\eta$ meeting $\eta$ at a point $x$ (thus $\mu$ meets $\eta$ orthogonally if $x$ is not a cone point, and with angle at least $\pi/2$ on both sides if it is a cone point, c.f. Figure \ref{F:bigon}). Construct $\delta$ by concatenating two basic geodesic rays emanating from $\zeta$ and orthogonal to $\mu$ such that it makes angle $\pi$ at $\zeta$ on the same side as $x$ is on (see the image at left in Figure \ref{F:bigon}). Then $\delta$ and $\eta$ do not intersect, since if they did, this would create a geodesic triangle with angle sum greater than $\pi$ which is not possible. Furthermore, $\delta$ cannot be asymptotic to $\eta$ since this would create a partially ideal triangle with angle sum $\pi$. Therefore, $\delta$ and $\eta$ have distinct unlinked end points in $S^1_\infty$. If necessary we perturb the angle at which $\delta$ meets $\mu$ while maintaining unlinked distinct endpoints from $\eta$ so that $\delta$ is in $\CG(\tphi_1,\Omega)$.

Now, note that $g(\delta)$ also passes through $\zeta$ and that $\delta$ and $g(\delta)$ have the same endpoints as $\eta$ and $g(\eta)$. If $\zeta$ is to the right of $g(\eta)$ there would be a bigon bounded by subsegments of $g(\eta)$ and $g(\delta)$ as seen in Figure \ref{F:bigon}, a contradiction. Hence $\zeta$ lies to the left of $g(\eta)$. 
\end{proof}

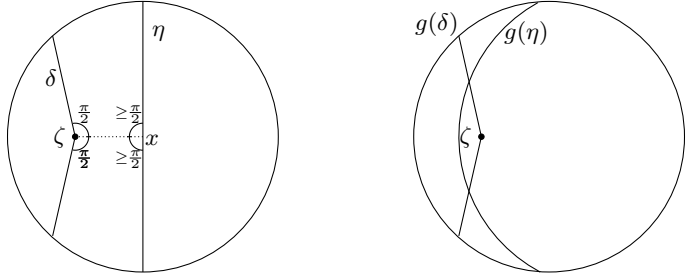
\begin{figure}[htb]
\begin{center}
  \captionsetup{width=.85\linewidth}
\begin{tikzpicture}[scale = .6]
\draw (0,0) circle (3);
\draw (0, 3) -- (0,-3);
\draw (-2,-2.2) -- (-1.5,0) -- (-2,2.2361);
\draw[densely dotted] (-1.5,0) -- (0,0); 
\draw[fill=black] (-1.5,0) circle (.06cm);
\draw [domain=-100:100] plot ({-1.5+0.3*cos(\x)}, {0.3*sin(\x)});
\draw (-1.15, 0) -- (-1.25, 0);
\draw [domain=-90:90] plot ({-0.3*cos(\x)}, {-0.3*sin(\x)});
\draw (-0.25, 0) -- (-0.35, 0);
\node[left] at (-1.5,0) {\small{$\zeta$}};
\node[right] at (0,2.3) {\small{$\eta$}};
\node[left] at (-1.7,1.3) {\small{$\delta$}};
\node[below] at (-1.3,-0.1) {\tiny{$\frac{\pi}{2}$}};
\node[above] at (-1.3,0.1) {\tiny{$\frac{\pi}{2}$}};
\node[below] at (-1.3,-0.1) {\tiny{$\frac{\pi}{2}$}};
\node[above] at (0.2,-0.4) {\small{$x$}};
\node[above] at (-0.3,0.1) {\tiny{$\geq\hspace{-0.1cm}\frac{\pi}{2}$}};
\node[below] at (-0.3,-0.1) {\tiny{$\geq\hspace{-0.1cm}\frac{\pi}{2}$}};

\draw (9,0) circle (3);
\draw [domain=-118:-241] plot ({10.4+3.4*cos(\x)}, {3.4*sin(\x)});
\draw (7,-2.2) -- (7.5,0) -- (7,2.2361); 
\draw[fill=black] (7.5,0) circle (.06cm);
\node[left] at (7.5,0) {\small{$\zeta$}};
\node[left] at (9.2,2.3) {\small{$g(\eta)$}};
\node[left] at (7.2,2.5) {\small{$g(\delta)$}};
\end{tikzpicture}
\caption{\textbf {Left:} The point $\zeta\in \tC$ and basic geodesic $\delta$ lie in $\mathcal{H}^+(\eta)$. \textbf {Right:} The bigon created if $g(\eta)$ did not preserve the partition}
\label{F:bigon}
\end{center}
\end{figure}

The next lemma proves a stronger property for concurrence points and extends Lemma~\ref{l:partition}.


\begin{lemma}\label{l:partition2}
Suppose $\mathcal{G}_{\tphi_1} = \mathcal{G}_{\tphi_2}$, $\Omega$ is as in Proposition~\ref{P:chain consequence}, $\tC\subset\tS$ is a concurrency set, and $\varphi_2$ is $C$--normalized.  For any $\zeta\in \tC$ we have $\CG_{\tphi_1}(\zeta) = \CG_{\tphi_2}(\zeta)$. Moreover, any $\eta\in \mathcal{G}(\tphi_1)$ defines the same partition of $\tC$ (into five sets) as $g(\eta)$ does.
\end{lemma}

\begin{proof}
Let $\zeta\in \tC$ and let $\eta\in\CG(\tphi_1,\zeta)$ be an oriented basic $\tphi_1$--geodesic passing through $\zeta$.  Choose $\{\eta_n\} \subset \CG(\tphi_1,\Omega)$ to be a sequence of nonsingular oriented $\tphi_1$--geodesics limiting to $\eta$. If $\zeta$ is a cone point which is, say, on and to the left of $\eta$, we have that $\zeta$ lies to the left of $\eta_n$ for all large $n$.
If $\zeta$ is not a cone point we can choose a sequence of nonsingular geodesics $\eta_n$ that converge to $\eta$ such that $\zeta$ lies to the left of $\eta_n$ for all $n$ (see proof of Lemma~\ref{L:characterizing basics} for a construction of such a sequence). 

Now, $\{g(\eta_n)\}$ is a sequence of nonsingular geodesics which limits to $g(\eta)$ and $\zeta$ lies to the left of each $g(\eta_n)$ by Lemma \ref{l:partition}, and hence $\zeta$ lies to the left of $g(\eta)$ or is on $g(\eta)$. Note that $\eta$ and $g(\eta)$ have the same endpoints. Suppose $\zeta$ is {\em not} on $g(\eta)$. Then there exists (an oriented) $\delta\in \CG(\tphi_2,\Omega)$ such that $\zeta$ is to the left of $\delta$ (it is in $\mathcal{H}^+(\delta)$) and $g(\eta)$ is to the right of $\delta$ (it is contained in $\mathcal{H}^-(\delta)$). Such a geodesic can be constructed by a similar method as was used in the proof of Lemma \ref{l:partition}: take the minimal length geodesic path $\mu$ from $\zeta$ to $g(\eta)$ meeting $g(\eta)$ at a point $x$ (orthogonally if $x$ is not a cone point).  For any point $y$ in the interior of $\mu$ that is not a cone point, let $\delta_y$ be any basic $\tphi_2$--geodesic intersecting $\mu$ orthogonally at $y$, and if necessary, perturb $y$ slightly to avoid the countably many geodesics not in $\CG(\tphi_2,\Omega)$; the resulting geodesic $\delta = \delta_y$ cannot intersect $g(\eta)$ since if it did it would create a geodesic triangle with angle sum at least $\pi$. Again by Lemma \ref{l:partition}, $g^{-1}(\delta)\in  \CG(\tphi_1,\Omega)$ also has $\zeta$ on the left (and has the same endpoints as $\delta$).  However, since $\eta$ passes through $\zeta$ we have then that $\eta$ and $g^{-1}(\delta)$ contain segments bounding a geodesic bigon, a contradiction. Hence we must have that $\zeta$ is on $g(\eta)$ as desired.   Therefore, $\CG_{\tphi_1}(\zeta) \subset \CG_{\tphi_2}(\zeta)$.  The other containment is proved by a symmetric argument, and thus $\CG_{\tphi_1}(\zeta) = \CG_{\tphi_2}(\zeta)$.

The second claim is proved similarly, appealing partially to the first part of the lemma.  Let $\eta\in \mathcal{G}(\tphi_1)$ be an oriented geodesic and $\zeta\in\tC$. If $\zeta$ lies {\em on} $\eta$ then by the first argument, $\zeta$ also lies {\em on} $g(\zeta)$.  If $\zeta \in \tC \smallsetminus \tC_0$, then $\zeta$ splits $\eta$ if and only if $g(\eta)$ splits $\eta$.  So suppose $\zeta$ is either to the left of $\eta$ or on and to the left of $\eta$.  As above, we can approximate $\eta$ with nonsingular geodesics $\{\eta_n\} \subset \CG(\tphi_1,\Omega)$ such that $\zeta$ lies to the left of each $\eta_n$. Then $\zeta$ lies to the left of each $g(\eta_n)$ by Lemma~\ref{l:partition}, and since $g(\eta_n)$ converges to $g(\eta)$, it must be that either $\zeta$ is to the left of $g(\eta)$ or on and to the left of $g(\eta)$.  Since $\zeta$ is on $\eta$ if and only if its on $g(\eta)$ (by the first part of the lemma), it follows that $\zeta$ is on and to the left of $\eta$ if and only if it is on and to the left of $g(\eta)$.  Consequently, $\zeta$ is to the left of $\eta$ if and only if $\zeta$ is to the left of $g(\eta)$.  The case of $\zeta$ to the right of $\eta$ or on and to the right of $\eta$ is handled in exactly the same way.  This completes the proof.
\end{proof}

\begin{remark}
According to Lemma \ref{l:partition2} a $(\tphi_1,\tphi_2,\Omega)$--concurrency point $\zeta$ has the property that (after normalizing $\varphi_2$) all basic $\tphi_1$--geodesics which contain it straighten to $\tphi_2$--geodesics that contain it.  Thus we are justified in simply referring to $\zeta$ as a $(\tphi_1,\tphi_2)$--concurrency point (and likewise referring to concurrency sets for $(\tphi_1,\tphi_2)$) without reference to $\Omega$.  Indeed, going forward $\Omega$ will play no further role and so we discard it from the discussion.
\end{remark}

\subsection{Saddle connections, rays, and triangles}

For the discussion in the remainder of this subsection, assume $\varphi_1,\varphi_2 \in \tHyp_c(S)$ with $\CG_{\tphi_1} = \CG_{\tphi_2}$, we fix a concurrency set $C$ for $(\varphi_1,\varphi_2)$, and assume $\varphi_2$ has been $C$--normalized.

Given a $C$--singular geodesic $\eta$, choosing an orientation of $\eta$ induces an order on the points of $\tC$ it encounters. We say that the (oriented) geodesics $\eta \in \CG(\tphi_1)$ and $g(\eta) \in \CG(\tphi_2)$ have the {\em same combinatorics} (with respect to $\tC$) if they pass through exactly the same set of points of $\tC$ and in the same order.  The goal of this subsection is to show that $\eta$ and $g(\eta)$ have the same combinatorics for all $\eta \in \CG(\tphi_1)$.
Observe that by Lemma \ref{l:partition2}, $\eta$ and $g(\eta)$ encounter the same set of points of $\tC$, so we will need only show that the order in which they encounter the concurrence points is preserved by $g$.  
First we introduce some more terminology.

We say a $\tphi_i$--geodesic ray is a {\em $(C, \tphi_i)$--ray} (or just a {\em $C$-ray} if the metric is understood) if its initial point is in $\tC$ but its interior is disjoint from $\tC$. We say that $\zeta\in\tC$ and $x\in S^1_{\infty}$ determine a $(C, \tphi_i)$--ray if the $\tphi_i$--geodesic ray from $\zeta$ to $x$ is a $(C,\tphi_i)$--ray. 

Let $\zeta, \xi\in \tC$ and $\delta$ be a $\tphi_i$--geodesic segment connecting $\zeta$ and $\xi$. If $\delta$ contains no point of $\tC$ in its interior, we say that $\delta$ is a $(C, \tphi_i)$--{\em saddle connection}, or just a {\em $C$--saddle connection} when the metric is clear. We say that $\zeta, \eta$ determine a $(C, \tphi_i)$--saddle connection when the $\tphi_i$--geodesic segment between them is a $(C,\tphi_i)$--saddle connection.  A $(C, \varphi_i)$--saddle connection in $S$ is the image of a $(C, \tphi_i)$--saddle connection in $\tS$.

\begin{lemma}\label{l:saddle connection}
If $\mathcal{G}_{\tphi_1} = \mathcal{G}_{\tphi_2}$, $\tC$ is a concurrency set, $\varphi_2$ is $C$--normalized, and $\zeta,\xi \in \tC$, $x \in S^1_\infty$, then:
\begin{enumerate}
\item $\zeta, \xi$ determine a $(C, \tphi_1)$--saddle connection if and only if they determine a $(C,\tphi_2)$--saddle connection.
\item $\zeta, x $ determine a $(C, \tphi_1$)--ray if and only if they determine a $(C, \tphi_2)$--ray.
\end{enumerate}
\end{lemma}

\begin{proof}
Let $\delta_1$ be a $(C, \tphi_1)$--saddle connection between $\zeta$ and $\xi$ and $\delta_2$ the unique $\varphi_2$--geodesic segment between $\zeta$ and $\xi$. We need to show that $\delta_2$ contains no points of $\tC$ in its interior. 
Suppose to the contrary that there were some $\omega \in \tC$, distinct from $\zeta,\xi$ on $\delta_2$.   Since $\omega$ is not on $\delta_1$, we can find an oriented, non-$C$--singular $\tphi_1$--geodesic $\eta \in \CG(\tphi_1)$ so that $\zeta,\xi$ are to the left of $\delta_1$ and $\omega$ is to the right.  By Lemma~\ref{l:partition2}, $g(\eta)$ is non-$C$--singular and $\zeta,\xi$ are to the left and $\omega$ is to the right.  But then subsegments of $\delta_2$ and $g(\eta)$ bound a bigon, which is a contradiction.  Therefore, $\delta_2$ is a $(C,\tphi_2)$--saddle connection.  A symmetric argument proves the other implication.


The proof of the second claim is similar.  Suppose $r_i$ is the $\tphi_i$--ray from $\zeta$ to $x$ and assume $r_1$ is a $C$--ray while $r_2$ contains some point $\omega \in \tC$ different than $\zeta$.  There is an oriented non-$C$--singular $\tphi_1$--geodesic $\eta \in \CG(\tphi_1)$ so that $r_1$ is to the left of $\eta$ {\em and does not share an endpoint with it} while $\omega$ is to the right.  By Lemma~\ref{l:partition2} $g(\eta)$ is non-$C$--singular, $\zeta$ is to the left, and $\omega$ is to the right, but $r_2$ must therefore travel from the left of $g(\eta)$ starting at $\zeta$, to the right passing through $\omega$, and back to the left again to limit to $x$.  Therefore $r_2$ and $g(\eta)$ contain segments that bound a bigon, which is a contradiction.  So $r_2$ is a $C$--ray.  Again a symmetric argument proves the other implication.
\end{proof}

If $r$ is a $(C, \tphi_1$)--ray from $\zeta \in \tC$ to $x \in S^1_\infty$, then Lemma~\ref{l:saddle connection} implies there is a $(C, \tphi_2)$--ray from $\zeta$ to $x$, which we suggestively denote $g(r)$.  Similarly, the lemma implies that for each $(C,\tphi_1)$--saddle connection $\delta$ between a pair $\zeta,\xi \in \tC$ there is a $(C,\tphi_2)$--saddle connection we denote $g(\delta)$ between $\zeta,\xi$ as well.

\begin{corollary}\label{c:concatenation}
Suppose $\CG_{\tphi_1} = \CG_{\tphi_2}$, $\tC$ is a concurrency set, and $\varphi_2$ is $C$--normalized. Then for every $\eta\in\mathcal{G}(\tphi_1)$, we have that $\eta$ and $g(\eta)$ have the same combinatorics. Equivalently, $\eta$ is non-$C$--singular if and only if $g(\eta)$ is, and if $\eta$ is a (finite, infinite, bi-infinite) concatenation
$$\eta = \ldots\delta_i\cdot \delta_{i+1}\cdot \delta_{i+2}\ldots$$ 
of $(C,\tphi_1)$--saddle connections and/or $(C, \tphi_1)$--rays, then 
$$g(\eta) = \ldots g(\delta_i)\cdot  g(\delta_{i+1})\cdot g(\delta_{i+2})\ldots.$$
\end{corollary}
\begin{proof} Given an oriented $\eta \in \CG(\tphi_1)$, if $\eta$ is non-$C$--singular, then by Lemma~\ref{l:partition2}, so is $g(\eta)$.  If $\eta$ is a nontrivial concatenation $\eta =\ldots \delta_i \cdot \delta_{i+1} \ldots$, then all the points of $\tC$ that appear as endpoints of the saddle connections/rays of this concatenation also appear in $g(\eta)$ by Lemma~\ref{l:partition2}.  In particular, each $g(\delta_i)$ is a $(C,\tphi_2)$--saddle connection or $(C, \tphi_2)$--ray contained in $g(\eta)$.  If the points of $\tC$ did not appear in the same order along $g(\eta)$ as they did in $\eta$, then $g(\delta_i)$ would overlap with some $g(\delta_j)$ for some $i \neq j$, contradicting the fact that these are $(C,\tphi_2)$--saddle connections and/or $(C, \tphi_2)$--rays, by Lemma~\ref{l:saddle connection}.
\end{proof}

We say that $T$ is a $(C,\tphi_i)${\em -triangle} if it is a $\tphi_i$--geodesic triangle such that the vertices lie in $\tC$, its edges are $C$--saddle connections, and there are no points of $\tC$ in its interior. Note that, since $\tC_0\subset\tC$, any $(C,\tphi_i)$--triangle is isometric to a hyperbolic triangle in $\mathbb{H}$. If the covering map $p$ restricted to a $(C,\tphi_i)$--triangle is injective on the interior, then we call the image in $S$ a $(C,\varphi_i)$--triangle (in fact, in this case, the negative curvature of $\varphi_i$ implies that the restriction of $p$ can only fail to be injective on the vertices).  We say that three points of $\tC$ {\em determine a $(C,\tphi_i)$--triangle} if each pair determines a $(C,\tphi_i)$--saddle connection, and the region bounded by these $C$--saddle connections is such a triangle. 

\begin{lemma}\label{l:triangle}
Suppose $\CG_{\tphi_1} = \CG_{\tphi_2}$, $\tC$ is a concurrency set,  and $\varphi_2$ is $C$--normalized. If $\zeta, \xi, \chi \in \tC$ determine a $(C,\tphi_1)$--triangle then they also determine a $(C,\tphi_2)$--triangle, and the two triangles have the same orientation. 
\end{lemma}

\begin{proof}
Let $T$ be the $(C,\tphi_1)$--triangle determined by $\zeta, \xi, \chi \in \tC$, and orient it positively (its interior is to the left of its boundary) and suppose this induces the cyclic orientation $\zeta<\xi<\chi < \zeta$ on its vertices. Let $\alpha_1$, $\alpha_2$, and $\alpha_3$ be the 
$(C,\tphi_1)$--saddle connection from $\zeta$ and $\xi$, from $\xi$ to $\chi$, and from $\chi$ to $\zeta$, respectively.
By Lemma~\ref{l:saddle connection}, these determine $(C, \tphi_2)$--saddle connections $\alpha_1'=g(\alpha_1)$, $\alpha_2'=g(\alpha_2)$, and $\alpha_3'=g(\alpha_3)$, and these bound a triangular region $T'$. We need to show that $T'$ has no points of $\tC$ in its interior, and that it has the same orientation as $T$. 

Extend $\alpha_1, \alpha_2, \alpha_3$ to oriented basic $\tphi_1$--geodesics $\eta_1, \eta_2, \eta_3$ such that each of the three points $\zeta, \xi, \chi$ lie to the left or on and to the right of each of the geodesics; see Figure~\ref{F:triangle}.  This choice ensures that for each $i \neq j$, $\eta_i$ and $\eta_j$ cross and intersect in a single point (namely one of $\zeta,\xi,\chi$).
By Corollary \ref{c:concatenation} $g(\eta_i)$ is a basic geodesic which has $g(\alpha_i)$ as a subsegment for each $i=1,2,3$.

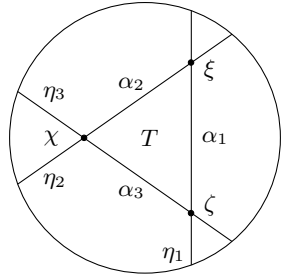
\begin{figure}[htb]
\begin{center}
  \captionsetup{width=.85\linewidth}
\begin{tikzpicture}[scale = .6]
\draw (0,0) circle (3);
\draw ({3*cos(70)},{3*sin(70)}) -- ({3*cos(-70)},{3*sin(-70)});
\draw ({3*cos(50)},{3*sin(50)}) -- ({3*cos(200)},{3*sin(200)});
\draw ({3*cos(-50)},{3*sin(-50)}) -- ({3*cos(-200)},{3*sin(-200)});
\draw[fill=black] (1.02,1.67) circle (.06cm);
\draw[fill=black] (1.02,-1.67) circle (.06cm);
\draw[fill=black] (-1.35,0) circle (.06cm);
\node[right] at (1.05,0) {\small{$\alpha_1$}};
\node[above] at (-.3,.85) {\small{$\alpha_2$}};
\node[below] at (-.3,-.85) {\small{$\alpha_3$}};
\node[right] at (1.1,-1.5) {\small{$\zeta$}};
\node[right] at (1.1,1.5) {\small{$\xi$}};
\node[left] at (-1.7,0) {\small{$\chi$}};
\node[above] at (-2,.6) {\small{$\eta_3$}};
\node[below] at (-2,-.6) {\small{$\eta_2$}};
\node[left] at (1.1,-2.6) {\small{$\eta_1$}};
\node at (.1,0) {\small{$T$}};
\end{tikzpicture}
\caption{The points $\zeta,\xi$ are on and to the right of $\eta_1$, while $\chi$ is to the left of $\eta_1$. The intersection of the positive half-planes is the triangle: $\mathcal H^+(\eta_1) \cap \mathcal H^+(\eta_2) \cap \mathcal H^+(\eta_3) = T$.}
\label{F:triangle}
\end{center}
\end{figure}

Note that the intersection of the three left half-planes $\mathcal{H}^+(\eta_1), \mathcal{H}^+(\eta_2)$, and $\mathcal{H}^+(\eta_3)$ is $T$ and hence contains no other points of $\tC$ besides the vertices $\zeta, \xi, \chi$. By Lemma~\ref{l:partition2} (preservation of partitions), the same is true for the intersection of $\mathcal{H}^+(g(\eta_1)), \mathcal{H}^+(g(\eta_2))$, and $\mathcal{H}^+(g(\eta_3))$. However, by the above, this intersection is exactly $T'$. Hence $T'$ is indeed a $(C,\tphi_2)$--triangle. By Corollary~\ref{c:concatenation}, $\zeta$ precedes $\xi$ along $g(\eta_1)$ and since $\chi$ is to the left of $g(\eta_1)$, $T'$ has the same orientation as $T$.
\end{proof}

We end this subsection with one more fact we will need for future reference.

\begin{lemma}\label{l:transverse}
Suppose $\mathcal{G}_{\tphi_1} = \mathcal{G}_{\tphi_2}$, $\tC$ is a concurrency set and assume $\varphi_2$ is $C$--normalized. Let $\eta$ be a basic $\tphi_1$--geodesic or a $(C, \tphi_1)$--saddle connection. If $\eta$ transversely crosses the interior of a $(C, \tphi)$--saddle connection $\delta$, then $g(\eta)$ transversely crosses the interior of $g(\delta)$.
\end{lemma}

\begin{proof}
Extend $\eta$ and $\delta$ to basic $\tphi_1$--geodesics $\eta'$ and $\delta'$ (where $\eta'=\eta$ in the case $\eta$ is a basic geodesic to start). Since $\eta'$ and $\delta'$ cross, their endpoints link, and since $g(\eta')$ and $g(\delta')$ have the same endpoints as $\eta'$ and $\delta'$, respectively, $g(\eta')$ and $g(\delta')$ also cross. 
 
Suppose $\zeta, \xi\in \tC$ are the endpoints of $\delta$. Note that $\zeta$ and $\xi$ are disjoint from and lie on opposite sides of $\eta'$. Hence, since $g$ preserves partitions and concurrency,  $\zeta$ and $\xi$ are disjoint from and lie on opposite sides of $g(\eta')$. It follows that $g(\eta')$ transversely crosses $g(\delta')$ at a point on the interior of the segment connecting $\zeta$ and $\xi$, which is exactly $g(\eta)$ by Corollary \ref{c:concatenation}. 


Now, if $\eta=\eta'$ is a geodesic we are done. If $\eta$ is a saddle connection, say with endpoints $\zeta', \xi'$ we get by the same argument as above that a basic extension of $\delta$ partition $\zeta'$ and $\xi'$ on different sides and, as above, we have that the intersection must happen at a point belonging to the saddle connection $\eta$. 
\end{proof}

\subsection{Triangulations}\label{S:triangulations}
In this section we continue with the assumption that $\CG_{\tphi_1} = \CG_{\tphi_2}$ for some $\varphi_1,\varphi_2 \in \tHyp_c(S)$, and that $\varphi_2$ has been $C$--normalized for a concurrency set $C$.  We will explain how to adjust $\varphi_2$ to an equivalent metric $\varphi_2'$ and find a {\em common triangulation} with vertices in $C$ and edges which are $(C,\varphi_i)$--saddle connections for $i=1$ {\em and} $2$ (see below for precise definitions).  We note that the isotopy to the identity arising in the adjustment of $\varphi_2$ may nontrivially braid $C$.  Once this adjustment is made, however, we will  show that for any other triangulation, the analogous adjustment of $\varphi_2'$ can be done {\em without} additional braiding of $C$.


Now suppose $\CT$ is a {\em $(C,\varphi_1)$--triangulation} (or just $\varphi_1$--triangulation, if $C$ is understood).  This is a $\Delta$--complex structure $\CT$ for $S$ (in the sense of \cite{Hatcher}) so that the $0$--skeleton is precisely $C$, and so that each edge is a $(C, \varphi_1$)--saddle connection.  In particular, each $2$--cell is a $(C,\varphi_1)$--triangle as defined above.

We define a map $f_\CT \colon S \to S$ to be the descent to $S$ of a $\pi_1S$--equivariant map $\tilde f_\CT \colon \tS \to \tS$ which is the identity on $\tC$ and which sends each $\tphi_1$--triangle $T_1$ to the straightened $\tphi_2$--triangle $T_2$.  By Lemma~\ref{l:triangle}, $T_2$ is also isometric to a geodesic triangle in the hyperbolic plane, and $\tilde f_\CT$ is defined to be the canonical triangle map from $T_1$ to $T_2$ (conjugated by the isometries to the hyperbolic triangles).  Equivariance of $\tilde f_\CT$ is immediate.
\begin{proposition} The map $f_\CT$ is a homeomorphism, isotopic to the identity on $S$.
\end{proposition}
\begin{proof} Since $\tilde f_\CT$ is the identity on $\tC$ and is equivariant, $(f_\CT)_*$ is the identity, and so $f_\CT$ has degree $1$ and is homotopic to the identity.  Since it maps each oriented triangle to a similarly oriented triangle, it is in fact a homeomorphism.  Furthermore, since isotopy and homotopy are the same equivalence relation for homeomorphisms for closed surfaces, $f_\CT$ is isotopic to the identity, as required.
\end{proof}

Given a $(C,\varphi_1)$--triangulation $\CT$, we may adjust the metric $\varphi_2$ by pulling back via $f_\CT$ to produce an equivalent metric. Having done so, and replacing $\varphi_2$ with the equivalent metric we say that $\varphi_2$ has been {\em $(C, \mathcal{T})$--normalized} and we note that $\CT$ is simultaneously a $(C,\varphi_1)$--triangulation and a $(C,\varphi_2)$--triangulation, and the map $f_\CT$ is the identity. In this case we say that $\CT$ is a {\em $(C,\varphi_1,\varphi_2)$--triangulation}. In other words, we have the following result:

\begin{lemma}\label{L:triangle normalized}
Suppose $\mathcal{G}_{\tphi_1} = \mathcal{G}_{\tphi_1}$.  If $\tC$ is any concurrency set and $\CT$ a $(C, \varphi_1)$--triangulation, then $\varphi_2\sim \varphi_2'$ for some $(C, \CT)$--normalized $\varphi_2'$. That is, up to equivalence, we can always assume that $\CT$ is a $(C,\varphi_1,\varphi_2)$--triangulation. \qed
\end{lemma}

Having adjusted $\varphi_2$ so that $\CT$ is a $(C,\varphi_1,\varphi_2)$--triangulation, the next lemma says that this adjustment ``anchors" the cone points in relation to the geodesics.
\begin{lemma}  \label{L:relative straightening} Suppose $\mathcal{T}$ is a $(C,\varphi_1,\varphi_2)$--triangulation. 
For any $\eta \in \CG(\tphi_1)$, $g(\eta)$ is isotopic to $\eta$ by an isotopy $h_t \colon \tS \to \tS$ for $t \in [0,1]$ that preserves the triangles of $\CT$.  The same is true for any $(C,\tphi_1)$--saddle connection $\eta$.
\end{lemma}
\begin{proof} Let $\eta$ be any basic $\tphi_1$--geodesic or $(C,\tphi_i)$--saddle connection. For any triangle $T_1$ of $\widetilde{\CT}$, we have that $\eta \cap T_1$ is either empty, a vertex of $T_1$, a side of $T_1$, or an arc cutting through the interior of $T_1$, by convexity.  In the last case, the arc may run between a vertex and the interior of a side or from the interior of one side to the interior of another.  
If $\eta \cap T_1$ is not an arc, then by Lemmas \ref{l:partition2} and \ref{l:saddle connection} $g(\eta) \cap T_1 = \eta \cap T_1$.  On the other hand, if $\eta \cap T_1$ is an arc cutting through the interior of $T_1$, then by Lemma \ref{l:transverse}, $g(\eta) \cap T_1$ is also an arc connecting the same vertex to opposite sides, or the same pair of sides.  

Now we produce the required isotopy $h_t$ in two steps. The first step takes place during times $t \in [0,\frac{1}{2}]$ and the second step takes place during times $t \in [\frac{1}{2}, 1]$. We define $h_t$ for $t \in [0, \frac{1}{2}]$ so that it is the identity outside a small neighborhood of the $1$--skeleton $\widetilde{\CT}^{(1)}$, has $h_t(\widetilde{\CT}^{(1)}) = \widetilde{\CT}^{(1)}$, and so that $h_{\frac{1}{2}}(\eta \cap {\widetilde{\CT}}^{(1)}) = g(\eta) \cap {\widetilde{\CT}}^{(1)}$.  Now for any triangle $T_1$ of $\widetilde{\CT}$, $h_{\frac{1}{2}}(\eta) \cap T_1 = g(\eta) \cap T_1$, or else $h_{\frac{1}{2}}(\eta) \cap T_1$ and $g(\eta) \cap T_1$ are arcs with the same endpoints.  It is easy to complete the isotopy so that for $t \in [\frac{1}{2},1]$, $h_t$ is the identity on $\widetilde{\CT}^{(1)}$ and so that $h_1(\eta) \cap T_1 = g(\eta) \cap T_1$ for every triangle $T_1$.
\end{proof}

If an isotopy is the identity on a set $C$ throughout the isotopy, then we say that it is an {\em isotopy relative to $C$}, and that homeomorphism differing by such an isotopy are {\em isotopic relative to $C$}.
\begin{corollary} \label{C:compatible} 
Suppose $\CT$ is a $(C,\varphi_1)$--triangulation and that $\varphi_2$ has been $(C, \CT)$--normalized. If $\CT'$ is another $(C,\varphi_1)$--triangulation, then the associated homeomorphism $f_{\CT'} \colon S \to S$ is isotopic to the identity relative to $C$.
\end{corollary}
\begin{proof}  Each $(C,\varphi_1)$--saddle connection making up the edges of $\CT'$ is homotopic {\em in the complement of $C$} to its $(C,\varphi_2)$--straightening (since their lifts are isotopic in $\tS$ fixing $\tC$). Homotopy of arcs in the complement of the cone points implies isotopy of said arcs (fixing $C$ throughout the isotopy).  Since the saddle connections making up the edges of the triangulation have pairwise disjoint interiors, there is an isotopy of the identity, relative to $C$, that sends each $(C,\varphi_1)$--saddle connection of $\CT'$ to a $(C,\varphi_2)$--saddle connection.  After a further isotopy preserving the $1$--skeleton of $\CT'$, we obtain a map which is the canonical triangle map on each triangle of $\CT'$, and is thus $f_{\CT'}$.
\end{proof}

We say that $\varphi_2$ is {\em $C$--uber-normalized} (with respect to $\varphi_1$) if for any $(C,\varphi_1)$--triangulation $\CT$ of $S$, the isotopy to the identity involved in adjusting $\varphi_2$ so  $\CT$ is a $(C,\varphi_2)$--triangulation fixes $C$.  By the corollary, if $\varphi_2$ is $(C,\CT)$--normalized, then it is $C$--uber-normalized.  
We record the following fact about the examples from \S\ref{S:deformations} which will be useful later.

\begin{lemma} \label{L:examples C--uber-normalized} Suppose $\varphi_1,\varphi_2 \in \tHyp_c(S)$ are obtained from branched covers $q \colon (S,\varphi_1) \to \mathcal O_1$ and $F \circ q \colon (S,\varphi_2) \to \mathcal O_2$, for some homeomorphism $F \colon \mathcal O_1 \to \mathcal O_2$ as in Theorem~\ref{T:deforming}, so that $\CG_{\tphi_1} = \CG_{\tphi_2}$.  Then the preimage of the even order orbifold points, $C = q^{-1}(\mathcal E_1)$, is a concurrency set and $\varphi_2$ is $C$--uber-normalized with respect to $\varphi_1$.
\end{lemma}
\begin{proof} As in the proof of Theorem~\ref{T:deforming}, we adjust $F$ by a homeomorphism isotopic to the identity so that $\mathcal O_1 \to \mathcal O_2$ is the homeomorphism of Lemma~\ref{L:Delaunay maps}.  Since an isotopy on an orbifold preserves the orbifold points, the lift of the isotopy is an isotopy of the identity $S \to S$ relative to $C$.  It follows that the adjustment to $F$ changes $\varphi_2$ within its equivalence class by a homeomorphism isotopic to the identity relative to $C$.  Therefore, it suffices to prove the theorem under this assumption.  From the proof of Theorem~\ref{T:deforming}, we have a cell structure on $S$ with vertices in $C$ for which the $2$--cells are convex hyperbolic polygons with respect to both $\varphi_1$ and $\varphi_2$.  As in the proof of Theorem~\ref{T:deforming}, pulling back these cell structures to $\tS$ with vertex set $\tC = p^{-1}(C)$ we found that for every nonsingular $\tphi_1$--geodesic $\eta$, the $\tphi_2$--straightening $g(\eta)$ is nonsingular and it intersects the exact same set of cells as $\eta$ (this is proved in $\hS$, but this property pushes down to $\tS$).  Consequently, $\eta$ and $g(\eta)$ define the same partition of the set $\tC$.  Now we may make a nearly identical argument to the one in the proof of Lemma~\ref{l:partition2} to see that $\CG_{\tphi_1}(\zeta) = \CG_{\tphi_2}(\zeta)$ for all $\zeta \in \tC$, proving that $\tC$ is a concurrency set (by assumption, $\tC$ contains all cone points) and that $\varphi_2$ is $C$--normalized with respect to $\varphi_1$.

Now subdivide each $2$--cell into $(C,\varphi_1)$--triangles.  We can adjust $\varphi_2$ by a homeomorphism isotopic to the identity relative to $C$ so that the triangles are also $(C,\varphi_2)$--triangles, and the identity is the canonical map.  Since the isotopy is relative to $C$, the new metric is $C$--uber-normalized if and only if the original one is.  However, the new metric is $C$--uber-normalized by Corollary~\ref{C:compatible}.  This completes the proof.
\end{proof}


\section{Holonomy} \label{S:holonomy}

A key ingredient in the proof of the \currentsupport is Proposition~\ref{P:holonomy conjugation} below, which states that if $\CG_{\tphi_1} = \CG_{\tphi_2}$ (and $\varphi_2$ is suitable normalized), then there is an orientation preserving homeomorphism $h \colon \partial \mathbb H \to \partial \mathbb H$ topologically conjugating the associated holonomy homomorphisms $\rho_1$ to $\rho_2$ (see \S\ref{S:top conjugacy}).  The majority of this section is devoted to proving this proposition, and ends by showing that if the conjugating homeomorphism $h$ is in $\PSL_2(\mathbb R)$, then then $\varphi_1 \sim \varphi_2$. 

For the remainder of this section we shall assume $\CG_{\tphi_1} = \CG_{\tphi_2}$ for $\varphi_1,\varphi_2 \in \tHyp_c(S)$, that $C$ is a concurrency set for $(\varphi_1, \varphi_2)$, and that $\varphi_2$ has been $C$--uber-normalized.
In particular we have $C_0=\cone(\varphi_1)=\cone(\varphi_2) \subseteq C$.  

Recall from \S\ref{S:puncturing surfaces} that $\hS$ is the completion of the universal cover of $\dot S=S\setminus C_0$ with respect to (either of) the pulled back metrics $\hphi_i$ of $\varphi_i$, for $i=1,2$, and that the universal covering map extends to the completion which we denote by $\hat p: \hS\to S$.
Because $\varphi_2$ is $C$--uber-normalized, for any $(C,\varphi_1)$--triangulation $\mathcal T$ we may adjust $\varphi_2$ by a homeomorphism isotopic to the identity by an isotopy which fixes $C$, so that $\mathcal T$ is a $(C,\varphi_1,\varphi_2)$--triangulation.  Since $C_0 \subseteq C$, the $C$--uber-normalization implies that the isotopy lifts to $\hS$, and thus $\CG_{\hphi_2}$ does not change under such an adjustment.  The developing map for $\hphi_i$ extends to the completion $D_i: \hS\to \mathbb{H}$, and is equivariant with respect to the holonomy homomorphism $\rho_i: \pi_1 \dot S\to \PSL_2(\mathbb{R})$, for $i=1,2$, where $\dot S = S\setminus C_0$; see \S\ref{S:prelim develop}.

\subsection{Combinatorics of basic geodesics in $\hS$}

Analogous to the definition for $\tphi_i$--geodesics in $\tS$, we say that a $\hphi_i$--geodesic is {\em non-$C$--singular} if it is disjoint from $\hC=\hat p^{-1}(C)$ (and simply nonsingular if it is disjoint from the completion points $\hC_0=\hat p^{-1}(C_0)$).  A $(C, \hphi_i)$--saddle connection is a $\hphi_i$--geodesic segment between two points of $\hC$ such that its interior is disjoint from $\hC$. Also, recall that $\mathcal{G}(\hphi_i)$ denotes the set of all basic $\hphi_i$--geodesics, i.e. the closure of the set of nonsingular $\hphi_i$--geodesics. Moreover, as we did in $\tS$, we say that two basic geodesics  $\eta_1 \in \CG(\hphi_1)$ and $\eta_2 \in \CG(\hphi_2)$ have the {\em same combinatorics} if they have the same endpoints in $\hat S^1_{\infty}$ and pass through exactly the same set of points of $\hat C$, in the same order.

\begin{lemma}\label{L:hat combinatorics} Let $\varphi_1,\varphi_2 \in \tHyp_c(S)$ with $\CG_{\tphi_1}=\CG_{\tphi_2}$, $C$ any concurrency set, and suppose $\tphi_2$ is $C$--uber-normalized. Then $\CG_{\hphi_1}=\CG_{\hphi_2}$ and hence we have a well-defined straightening map $\hat g \colon \CG(\hphi_1) \to \CG(\hphi_2)$. Moreover, for all $\eta \in \CG(\hphi_1)$, $\eta$ and $\hat g(\eta)$ have the same combinatorics. 

\end{lemma}

As a consequence of Lemma~\ref{L:hat combinatorics}, the $\hphi_2$--straightening map $\hat{g}$ may be extended to a map from all basic $\hphi_1$--geodesic segments and rays (as well as lines), with endpoints (possibly at infinity) in the set $\hC \cup \hS^1_\infty$, to basic $\hphi_2$--geodesic segments and rays having the same combinatorics. This is obtained by extending any basic $\hphi_1$--geodesic segment or ray to a (biinfinite) basic $\hphi_1$--geodesic (see Lemma~\ref{L:characterizing basics}), applying $\hat g$, then appealing to the lemma to find an appropriate subsegment or ray of the resulting basic $\hphi_2$--geodesic.

\begin{proof} Fix a $(C,\varphi_1,\varphi_2)$--triangulation $\CT$, adjusting $\varphi_2$ to an equivalent metric if necessary.  Let $\widetilde{\CT}$ be the lifted triangulation to $\tS$ and further lift $\widetilde{\CT}$ to a triangulation $\widehat{\CT}$ in $\hS$.  As noted above, because $\varphi_2$ was $C$--uber-normalized, the isotopy to the identity of the homeomorphism used to adjust $\varphi_2$ fixes $C$ (and hence $C_0$) and so lifts to $\hS$.  In particular, proving the conclusion of the lemma for the adjusted metric implies it for the original metric.

As in Lemma~\ref{L:characterizing basics}, every non-$C$--singular $\hphi_i$--geodesic projects to a non-$C$--singular geodesic in $\tS$ by $\tilde p$ since this map is a local isometry (away from the cone points).  Let $\eta$ be a non-$C$--singular $\hphi_1$--geodesic. The $\tphi_2$--straightening $g(\tilde p(\eta))$ is isotopic to $\tilde p(\eta)$ in the complement of $\tC$, by an isotopy that preserves each triangle of $\widetilde{\CT}$ by Lemma~\ref{L:relative straightening}. We can lift this isotopy, triangle-by-triangle, to an isotopy from $\eta$ to a non-$C$--singular $\hphi_2$--geodesic $\eta'$ passing through the same triangles of $\widehat{\CT}$, which therefore has the same endpoints in $\hat{S}^1_{\infty}$.  Any basic $\hphi_1$--geodesic $\eta$ is a limit of nonsingular $\hphi_1$--geodesics, and by analyzing the limits in each triangle (and appealing to Corollary~\ref{c:concatenation}), it follows that there is a basic $\hphi_2$--geodesic $\eta'$ running through precisely the same set of triangles as $\eta$ (and through the same set of vertices, in the same order).  It follows that $\CG_{\hphi_1} \subset \CG_{\hphi_2}$ and the straightening map $\hat g$ preserves combinatorics.  A symmetric argument proves the other inclusion, hence $\CG_{\hphi_1} = \CG_{\hphi_2}$, and we are done.
\end{proof}

The following is an immediate consequence of Lemma~\ref{L:hat combinatorics}.  Recall that $\CG(\hphi_i,\zeta)$ is the set of all basic $\hphi_i$--geodesics through a point $\zeta\in\hC$.

\begin{corollary}\label{C:hat concurrency}
Suppose $\CG_{\tphi_1} =\CG_{\tphi_2}$, $C$ is a concurrency set, and $\varphi_2$ is $C$--uber-normalized.  Then we have $\hat g(\CG(\hphi_1,\zeta)) = \CG(\hphi_2,\zeta)$ for all $\zeta\in\hC$. \qed
\end{corollary}

For this reason, we will refer to the points of $\hC$ as $(\hphi_1,\hphi_2)$--\textit{concurrence points} and $\hC$ a concurrency set for $(\hphi_1,\hphi_2)$.
Another useful fact is the following, which relates angles between intersecting basic geodesic rays.

\begin{lemma} \label{L:pi intervals} Given oriented basic $\hphi_1$--geodesic rays $\eta,\eta'$ emanating from a cone point $\zeta \in \hC_0$ making angle $\theta_1 > 0$, let $\theta_2 > 0$ be the angle between $\hat g(\eta)$ and $\hat g(\eta')$.  Then $\frac{\theta_1}\pi \in \mathbb Z$ if and only if $\frac{\theta_2}{\pi} \in \mathbb Z$, and in this case the angles are equal.  In general, we have
$\left\lfloor \frac{\theta_1}{\pi} \right\rfloor = \left\lfloor \frac{\theta_2}{\pi} \right\rfloor$.

As a consequence, if two oriented $\hphi_1$--geodesics $\eta, \eta'$ intersect at a cone point, and $D_1 (\eta) = D_1 (\eta')$, then $D_2 (\hat{g} (\eta))=D_2 (\hat{g} (\eta'))$.
\end{lemma}

\begin{proof} First note that the angles $\theta_1$ and $\theta_2$ depend only on the initial segments of $\eta$ and $\eta'$ (the case of $\theta_2$ requires an application of Lemma~\ref{L:hat combinatorics}).  In particular, $\theta_1$ remains unchanged by modifying the side on which $\eta$  and $\eta'$ make angle $\pi$ at various cone points they meet, since their initial segments are preserved by this modification.

First suppose the angle between $\eta$ and $\eta'$ is $\pi$.  Without loss of generality, we may assume that the angle from $\eta$ to $\eta'$ is $\pi$ (i.e.~counterclockwise) and that $\eta$ makes angle $\pi$ on the left and $\eta'$ makes angle $\pi$ on the right at every cone point they encounter.  Then reversing the orientation on $\eta$ and concatenating with $\eta'$ we get $\eta\cdot\eta'$ is a basic $\hphi_1$--geodesic making angle $\pi$ on the right at every cone point encountered.  Straightening, we see that $\hat g(\eta \cdot \eta') = \hat g(\eta) \cdot \hat g(\eta')$ is also a basic geodesic, implying that $\hat g(\eta)$ and $\hat g(\eta')$ also make angle $\pi$ (in fact, it's counter clockwise from $\hat g(\eta)$ to $\hat g(\eta')$ using the cyclic ordering on $\hS^1_\infty$).

Now given any two rays $\eta$ and $\eta'$, without loss of generality assume the angle from $\eta$ to $\eta'$ is positive (counterclockwise) so that the angle from $\hat g(\eta)$ to $\hat g(\eta')$ is also positive.  Let $\eta = \eta_0,\eta_1,\ldots,\eta_k,\eta_{k+1}$ be a finite sequence of geodesic rays emanating from $\zeta$ so that the angle from $\eta_i$ to $\eta_{i+1}$ is $\pi$ (counterclockwise) and the angle from $\eta_k$ to $\eta'$ is $\theta_1^0 \in [0,\pi)$.  If $\eta_k = \eta'$ then $\theta_1^0 = 0$, but otherwise, let $\eta_{k+1}$ be such that the angle from $\eta'$ to $\eta_{k+1}$ is $\pi - \theta^1_0 >0$.
Then $\theta_1 = k \pi + \theta_1^0$ and $\theta_1^0 \in [0,\pi)$.  By the first part of the proof, the angle from $\hat g(\eta_i)$ to $\hat g(\eta_{i+1})$ is $\pi$.  If $\theta_1^0 = 0$, then $\hat g(\eta_k) = \hat g(\eta')$ and thus $\theta_2 = k \pi = \theta_1$.  In this case $\tfrac{\theta_1}\pi = \tfrac{\theta_2}\pi$ and these are integers.  On the other hand, if $\theta_1^0 \in (0,\pi)$, then $\theta_2$ must be between $k\pi$ (the measure of the angle from $\hat g(\eta_0)$ to $\hat g(\eta_k)$) and $(k+1)\pi$ (the measure of the angle from $\hat g(\eta_0)$ to $\hat g(\eta_{k+1})$).  Therefore $\left\lfloor \frac{\theta_1}\pi \right\rfloor = k\pi = \left\lfloor \frac{\theta_2}\pi \right\rfloor$.

Finally note that if $\eta,\eta' \in \CG(\hphi_1,\zeta)$ and $D_1(\eta) = D_1(\eta')$, then $\eta,\eta'$ must make angle at $\zeta$ which is an integral multiple of $\pi$.  By the first part of the lemma the same is true for $\hat g(\eta),\hat g(\eta')$, and thus $D_2(\hat g(\eta)) = D_2(\hat g(\eta'))$.
\end{proof}

\subsection{Conjugating circle actions}\label{S:conjugating}

The fundamental group, $\pi_1 \dot S$, acts on $\hS$ by isometries with respect to $\hphi_1$ and $\hphi_2$. More generally we consider any group $G$ acting by homeomorphisms on $\hS$ that are isometries with respect to both $\hphi_1$ and $\hphi_2$. In this situation, there are holonomy homomorphisms $\rho_i: G \to\PSL_2(\mathbb{R})$, for $i=1, 2$.

The goal of this subsection is to use the developing maps and behavior of basic geodesics to prove the following, which is a key ingredient in the proof of the \currentsupport\!.

\begin{proposition} \label{P:holonomy conjugation} Suppose $\CG_{\tphi_1} = \CG_{\tphi_2}$, $C$ is a concurrency set and $\varphi_2$ is $C$--uber-normalized.  Suppose a group $G$ acts on $\hS$ by homeomorphisms that are isometries with respect to $\hphi_1$ and $\hphi_2$. Then there is an orientation preserving homeomorphism $h \colon \partial \mathbb H \to \partial \mathbb H$ topologically conjugating $\rho_1$ to $\rho_2$.  That is, for all $x \in \partial \mathbb H$ and $\gamma \in G$, we have
\begin{equation}
\label{E:conjugating homs}
h(\rho_1(\gamma) \cdot x) = \rho_2(\gamma) \cdot h(x).
\end{equation}
Moreover, for all $\eta \in \CG(\hphi_1)$, $h(\partial D_1(\eta)) = \partial D_2(\hat g(\eta))$.
\end{proposition}

We will define the map $h: \partial \mathbb{H} \rightarrow \partial \mathbb{H}$ as follows. Suppose $x \in \partial \mathbb{H}$, and take any oriented geodesic $\alpha$ in $\mathbb H$ with forward end point $x$.  Let $\eta$ be any basic $\hphi_1$--geodesic with $D_1(\eta) = \alpha$  (see Lemma~\ref{L:Dev surjective}) and define $h(x)$ to be the forward endpoint of $D_2(\hat{g}(\eta))$ in $\partial \mathbb{H}$. An immediate concern is whether $h$ is well-defined, and \S\ref{hwelldefined} is dedicated to showing that it is.

\subsubsection{Developed asymptotic geodesics}
\label{hwelldefined}

Given oriented basic $\hphi_i$--geodesics $\eta_1,\eta_2$, we say that they are \emph{developed forward asymptotic} if $D_i(\eta_1)$ and $D_i(\eta_2)$ are forward asymptotic in $\mathbb{H}$ (note we allow the possibility that $D_i(\eta_1) = D_i(\eta_2)$).

In this section, our goal is to show that if $\eta$ and $\eta'$ are basic $\hphi_1$--geodesics that are developed forward asymptotic, then $\hat{g}(\eta)$ and $\hat{g}(\eta')$ are developed forward asymptotic as well, which is an important ingredient in verifing that $h$ is well-defined.  The proof of this preservation of developed forward asymptoticity by $\hat g$ is somewhat complicated.  In fact, it turns out to be easier to prove that $\hat g$ preserves the property that developed geodesics {\em intersect} in the forward direction, and then observe that being developed forward asymptotic occurs as the limiting case of developed geodesics intersecting.  The intersecting behavior we prove is described by the next proposition.


\begin{proposition}\label{prop referee}
\label{prop:forward_intersecting}
Suppose $\eta$ and $\eta'$ are oriented basic $\hphi_1$--geodesics and let $\sigma$ be a $\hphi_1$--saddle connection with one end point on $\eta$ and the other on $\eta'$. Suppose further that
	\begin{enumerate}[(i)]
	\item the angle between $\sigma$ and the forward rays of each of $\eta$ and $\eta'$ measures strictly between $0$ and $\pi$, and
	\item $D_1(\eta)$ and $D_1(\eta')$ intersect forward of the end points of $D_1(\sigma)$.
	\end{enumerate}
Then $\hat g(\eta)$, $\hat g(\eta')$, and $\hat g(\sigma)$ satisfy (i) and (ii) with respect to $D_2$ as well.
\end{proposition}

\begin{proof} Observe that (i) holds for $\hat g(\eta)$, $\hat g(\eta')$, and $\hat g(\sigma)$ by Lemma~\ref{L:pi intervals}, so we need only prove (ii).
Let $\zeta$ and $\xi$ be the cone points at which $\sigma$ intersects $\eta$ and $\eta'$ respectively, and we orient $\sigma$ in the direction from $\zeta$ to $\xi$ so that $\eta$ and $\eta'$ cross $\sigma$ from left to right; see Figure~\ref{F:first developing pic}. Without loss of generality we can replace $\eta$ and $\eta'$ by geodesics that agree on their initial segments forward of $\zeta$ and $\xi$, respectively, so that they make angles $\pi$ on the left and right, respectively, at every cone point they encounter (by Lemma~\ref{L:pi intervals} if the conclusion holds for these replacements, it must hold for the original geodesics).

\begin{figure}[h]
\begin{center}
  \captionsetup{width=.85\linewidth}
\begin{tikzpicture}[scale=.9]
\draw[directedd](-6,-1.54) -- (-6,1.54);
\draw[directedd] (-7,1.8) .. controls (-4, 1) .. (-1,2+1);
\draw[directedd] (-7,-1.8) .. controls (-4,-1) .. (-1,-2+1);
\filldraw (-6, 1.54) circle (1pt);
\node at (-6.2,1.85) {\small $\xi$};
\filldraw (-6, -1.54) circle (1pt);
\node at (-6.2,-1.8) {\small $\zeta$};
\node at (-6.2,0) {\small $\sigma$};
\node at (-4,1.8) {\small $\eta'$};
\node at (-4,-1.5) {\small $\eta$};
\node at (-4,-2) {\small $(\hS, \hphi_1)$};
\draw[directedd] (2,-1.35) -- (2,1.35);
\draw[directedd] (1,1.8) .. controls (5,0) .. (7,-0.2+1);
\draw[directedd] (1,-1.8) .. controls (5,0) .. (7,0.2+1);
\filldraw (2, 1.35) circle (1pt);
\node at (2.6,1.5) {\small $D_1(\xi)$};
\filldraw (2, -1.35) circle (1pt);
\node at (2.6,-1.5) {\small $D_1(\zeta)$};
\node at (1.4,0) {\small $D_1(\sigma)$};
\node at (4,1) {\small $D_1(\eta')$};
\node at (4,-1) {\small $D_1(\eta)$};
\node at (4,-2) {\small Developed in $\mathbb{H}$};
\end{tikzpicture}
\caption{A pair of developed forward intersecting geodesics connected by a saddle connection and their developed image.}
\label{F:first developing pic}
\end{center}
\end{figure}
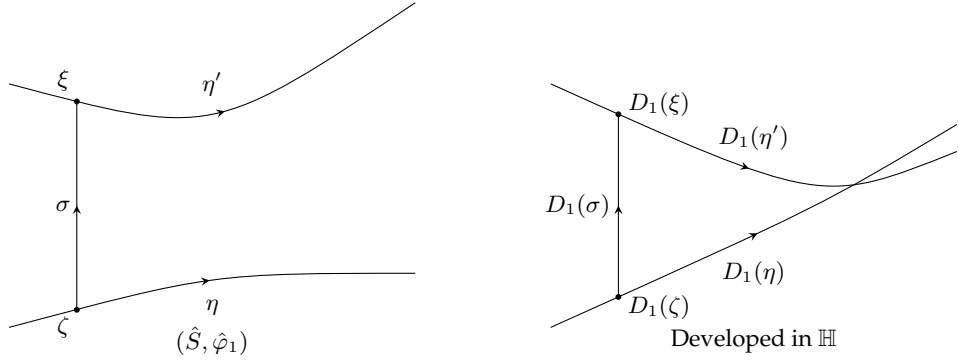

\paragraph{Creating a foliated triangle in $\mathbb{H}$.} The assumptions of the proposition imply that $D_1 (\sigma)$, $D_1(\eta)$ and $D_1(\eta')$ form a triangle $T$. Let $a$ be the point of intersection $D_1(\eta) \cap D_1(\eta')$. Foliate $T \subseteq \mathbb{H}$ by geodesic segments from $D_1(\sigma)$ to $a$, see Figure~\ref{figureT}. Orient each segment from $D_1(\sigma)$ to $a$.  The set of leaves can be parametrized by their intersection with $D_1(\sigma)$.\\ 

\begin{figure}[h]
\begin{center}
  \captionsetup{width=.85\linewidth}
\begin{tikzpicture}[scale=.9]
\draw (2,1.35) -- (2,-1.35);
\draw (1,1.8) .. controls (5,0) .. (7,-0.2+1);
\draw (1,-1.8) .. controls (5,0) .. (7,0.2+1);
\filldraw (2, 1.35) circle (1pt);
\node at (2.6,1.6) {\small $D_1(\xi)$};
\filldraw (2, -1.35) circle (1pt);
\node at (2.6,-1.6) {\small $D_1(\zeta)$};
\node at (1.3,-0.23) {\small $D_1(\sigma)$};
\node at (4,1) {\small $D_1(\eta')$};
\node at (4,-1) {\small $D_1(\eta)$};
\draw (2,.9) .. controls (4.2,.2) .. (5.48, 0.31);
\draw (2,0.4) .. controls (4,.1) .. (5.48, 0.31);
\draw (2,0) .. controls (4,0) .. (5.48, 0.31);
\draw (2,-0.4) .. controls (4,-.15) .. (5.48, 0.31);
\draw (2,-.9) .. controls (4,-.3) .. (5.48, 0.31);
\filldraw (5.48, 0.31) circle (1pt);
\node at (5.5,-0.3+1) {\small $a$};
\end{tikzpicture}
\caption{Foliated triangle $T$ in $\mathbb{H}$.}
\label{figureT}
\end{center}
\end{figure}
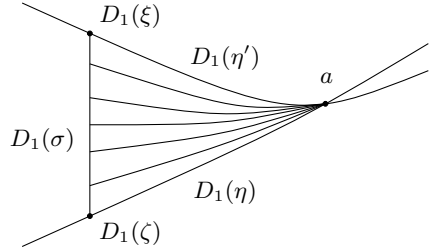

The goal for what follows is to "pull this foliation back" to a foliation of a region containing $\sigma$ whose leaves are basic $\hphi_1$--geodesic segments. We will then argue that a sufficiently robust subset of these basic geodesics are configured in a similar way after $\hphi_2$--straightening so that we may conclude that $\hat{g}(\eta)$ and $\hat g(\eta')$ are developed forward intersecting.\\

\paragraph{Pulling back the foliation.} Now we describe a foliated region in $\hat S$ that maps to the foliated triangle $T$ by $D_1$.  Consider a point $z \in \sigma$.  If there is a unique basic $\hphi_1$--geodesic arc starting at $z$ whose $D_1$--image is a leaf of the foliation of $T$ through $D_1(z)$, then denote this arc $f_z$.  If not, then all such arcs must agree on some initial subarc, $\delta$, from $z$ to a cone point.  In this case we let $f_z^l$ and $f_z^r$ be the arcs of the unique basic $\hphi_1$--geodesics through $z$, containing $\delta$, making angle $\pi$ on the left and right, respectively, at every cone point they encounter and that each map to the leaf through $D_1(z)$.  Now let $\hat T$ be the union of the arcs $f_z$ (or $f_z^l$ and $f_z^r$) over all $z \in \sigma$, which is a region in $\hat S$ (singularly) foliated by these arcs.

\paragraph{Finitely many cone points.} Next we want to show that $\hat T$ contains only finitely many cone points, and hence only finitely many leaves of the foliation of $\hat T$ encounter cone points.  To see this, we first note that $\hat T$ has finite diameter: if $d$ is the maximal length of a leaf of $T$, then every point of $\hat T$ is within $d$ of the compact segment $\sigma$.  Next, we claim that $\hat T$ projects injectively to $\tS$ by $\tilde p$. Assuming this claim, observe that since there are only finitely many cone points in the bounded image of $\hat T$ in $\tS$, there are only finitely many cone points in $\hat T$.  We are left to prove the claim.

To prove $\hat T$ projects injectively to $\tS$, we argue by contradiction. If it does not project injectively, there must be two distinct points $x$ and $y$ in $\hat T$ that project to a single point in $\tS$. The $\hphi_1$--geodesic leaves through $x$ and $y$ must be distinct since basic geodesics project injectively from $\hS$ to $\tS$. In the case where the two leaves are $f_z^l$ and $f_z^r$ for some $z$, $x$ and $y$ projecting to the same point of $\tS$ would contradict the fact that the total cone angle at each cone point of $\tS$ exceeds $2\pi$. In the case where the two leaves through $x$ and $y$ intersect $\sigma$ at distinct points $z$ and $z'$, we call the leaves $f_z$ and $f_{z'}$, suppressing the $r,l$ superscripts that may decorate these since they do not make a difference in the argument that follows.

The subsegment of $\sigma$ from $z$ to $z'$, together with the segments of $f_z$ and $f_{z'}$ from $z$ to $x$ and $z'$ to $y$, respectively, together project to a geodesic triangle $\Delta$ in $\tS$; see Figure~\ref{F:projecting to a triangle}.  Note that this triangle must either have a cone point in its interior, or a cone point on the boundary which makes an angle greater than $\pi$ toward the inside: otherwise, the entire triangle would lift to $\hS$, and hence $x$ and $y$ would be the same point of $\hat T$, a contradiction.  We will say that such cone points are {\em fully inside} $\Delta$.

\begin{figure}[h]
\begin{center}
  \captionsetup{width=.85\linewidth}
\begin{tikzpicture}[scale=1.1]
\draw (0,-.5) -- (0,2.2);
\draw (0,.3) -- (3,.3);
\draw (0,1.8) -- (3,1.8);
\filldraw (0,.3) circle (1pt);
\filldraw (0,1.8) circle (1pt);
\node at (-.25,-.3) {$\sigma$};
\node at (-.25,.4) {$z'$};
\node at (-.25,1.8) {$z$};
\filldraw (2.5,.3) circle (1pt);
\filldraw (2.5,1.8) circle (1pt);
\node at (2.5,2.05) {$x$};
\node at (2.5,0) {$y$};
\draw[->] (3.7,1) .. controls (4.4,1.25) and (5.5,1.25) .. (6.2,1);
\node at (5,1.45) {$\tilde p$};
\draw (8,-.5) -- (8,2.2);
\draw (8,.3) -- (10.9,1);
\draw (8,1.8) -- (10.9,.8);
\filldraw (8,.3) circle (1pt);
\filldraw (8,1.8) circle (1pt);
\node at (7.5,.3) {$\tilde p(z')$};
\node at (7.5,1.8) {$\tilde p(z)$};
\node at (7.5,-.3) {$\tilde p(\sigma)$};
\node at (10.5,1.4) {$\tilde p(x)$};
\node at (10.5,.5) {$\tilde p(y)$};
\filldraw (10.55,.92) circle (1pt);
\node at (9,1) {$\Delta$};
\end{tikzpicture}
\caption{Failure of injectivity produces a triangle $\Delta$ in $\tS$.}
\label{F:projecting to a triangle}
\end{center}
\end{figure}
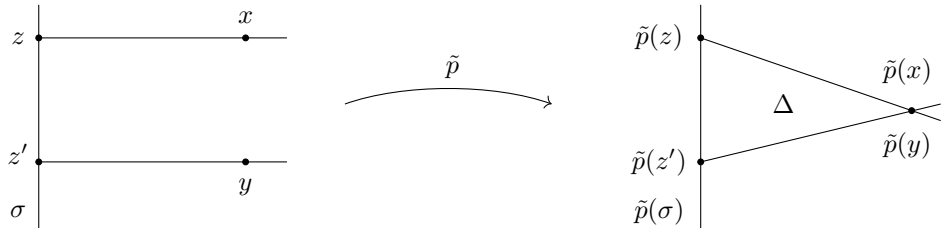

Because $\Delta$ is compact, there are only finitely many cone points fully inside $\Delta$.  Observe that there must be some $w$ in $\sigma$ between $z$ and $z'$ (possibly equal to one of $z$ or $z'$) so that $f_w^r$ and $f_w^l$ pass through one of these cone points fully inside $\Delta$.  One or both of these must project to a geodesic segment in $\tS$ that cuts off a ``sub-triangle", coming from two more distinct points $x'$ and $y'$ in $\hat T$ that project to the same point; see Figure~\ref{F:cone points in the triangle}.  This new triangle $\Delta'$ contains fewer cone points.  We may repeat this procedure, replacing $x,y$ with $x',y'$ and eventually arrive at a triangle containing no cone points fully inside it, which is a contradiction.  Therefore, $\hat T$ must inject into $\tS$, and consequently $\hat T$ contains only finitely many cone points.

\begin{figure}[h]
\begin{center}
  \captionsetup{width=.85\linewidth}
\begin{tikzpicture}[scale=1.5]
\filldraw[opacity=.15] (0,1.8) -- (1.45,1.3) -- (.7,1) -- (0,1) -- (0,1.8); 
\draw (0,1) -- (.7,1) -- (1.7,1.4);
\draw (.7,1) -- (1.3,.9);
\draw[dotted] (1.3,.9) -- (1.6,.85);
\filldraw (1.45,1.3) circle (1pt);
\draw (0,-.5) -- (0,2.2);
\draw (0,.3) -- (2.9,1);
\draw (0,1.8) -- (2.9,.8);
\filldraw (0,.3) circle (1pt);
\filldraw (0,1.8) circle (1pt);
\filldraw (0,1) circle (1pt);
\node at (-.35,.3) {$\tilde p(z')$};
\node at (-.35,1.8) {$\tilde p(z)$};
\node at (-.35,1) {$\tilde p(w)$};
\node at (-.35,-.3) {$\tilde p(\sigma)$};
\node at (2,1.85) {$\tilde p(x') = \tilde p(y')$};
\draw[->] (1.65,1.7) -- (1.45,1.35);
\node at (2.5,1.2) {$\tilde p(x)$};
\node at (2.5,.7) {$\tilde p(y)$};
\node at (.6,1.2) {$\tilde p(f_w^l)$};
\node at (.6,.8) {$\tilde p(f_w^r)$};
\filldraw (2.55,.92) circle (1pt);
\node at (.3,1.5) {$\Delta'$};
\filldraw[opacity=.15] (5,1.8) -- (6.45,1.3) -- (6.835,.742) -- (5,.3) -- (5,1.8);
\draw (5,-.5) -- (5,2.2);
\draw (5,.3) -- (7.9,1);
\draw (5,1.8) -- (7.9,.8);
\filldraw (5,.3) circle (1pt);
\filldraw (5,1.8) circle (1pt);
\draw (5,1.8) -- (6.45,1.3) -- (7,.5);
\filldraw (6.835,.742) circle (1pt);
\draw[->] (6.3,.2) -- (6.78,.68);
\node at (6.55,0) {$\tilde p(x') = \tilde p(y')$};
\node at (4.65,.3) {$\tilde p(z')$};
\node at (4.3,1.8) {$\tilde p(w) = \tilde p(z)$};
\node at (4.65,-.3) {$\tilde p(\sigma)$};
\node at (6.3,1) {$\tilde p(f_w^r)$};
\node at (6.7,1.5) {$\tilde p(f_w^l)$};
\node at (7.5,1.2) {$\tilde p(x)$};
\node at (7.5,.7) {$\tilde p(y)$};
\filldraw (7.55,.92) circle (1pt);
\node at (5.5,1.1) {$\Delta'$};
\end{tikzpicture}
\caption{Two possibilities for a subtriangle $\Delta'$ of $\Delta$ in $\tS$ (shaded) with fewer cone points fully inside.}
\label{F:cone points in the triangle}
\end{center}
\end{figure}
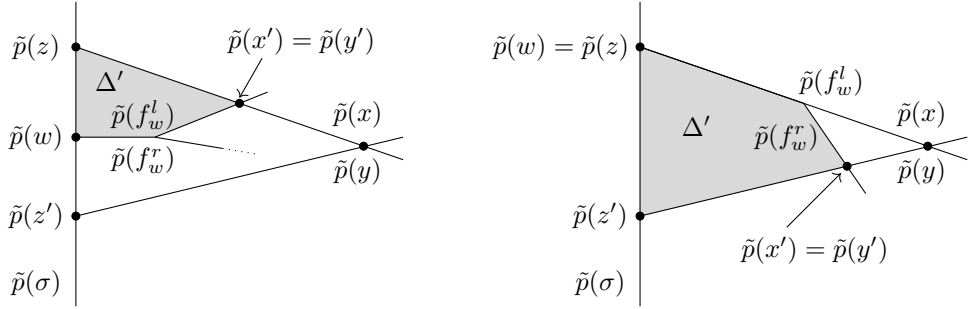


\paragraph{$\hat T$ is a union of hyperbolic triangles} We note that $\hat T$ has piecewise $\hphi_1$--geodesic boundary (consisting of subarcs of the arcs $f_z^l$ and $f_z^r$, together with arcs of $\eta$, $\eta'$, and $\sigma$), does not contain cone points in the interior, and the angle between consecutive boundary segments is less than or equal to $2\pi$; see Figure~\ref{hatT}.  Therefore, $D_1$ maps $\hat{T}$ to $T$, and is injective except along consecutive boundary arcs that make angle $2\pi$: these are arcs of $f_z^l$ and $f_z^r$ which are "zipped together" by $D_1$.  The leaves of $\hat T$ are sent to leaves of $T$ by construction.

The orientation of $\sigma$ in the direction from $\zeta$ to $\xi$ orders the leaves of the foliation of $\hat T$, except pairs $f_z^l$ and $f_z^r$ that meet cone points, which we order so that $f_z^r < f_z^l$.  Now write $f_1^r, f_1^l, f_2^r, f_2^l, \ldots, f_n^r, f_n^l$ for the (possibly empty) set of all leaves in $\hat{T}$, in order, that encounter cone points away from $\sigma$, and for convenience, write $f_0^l$ and $f_{n+1}^r$ for the leaves contained in $\eta$ and $\eta'$, respectively.
For each $0 \leq i \leq n$, let $\hat T_i \subset \hat T$ be the region bounded by $f_i^l$, $f_{i+1}^r$, and $\sigma$. 
Note that this region contains no cone points in its interior, and maps isometrically by $D_1$ to a sub-triangle of $T$ in $\mathbb H$ which is a union of leaves.  We view $\hat T$ as a "polygon" decomposed into a union of hyperbolic triangles in this way. See figure \ref{hatT}. The special case that $n=0$ is possible (then $\hat T$ is itself a hyperbolic triangle) and our argument is valid in that case as well, though one could also argue it directly.


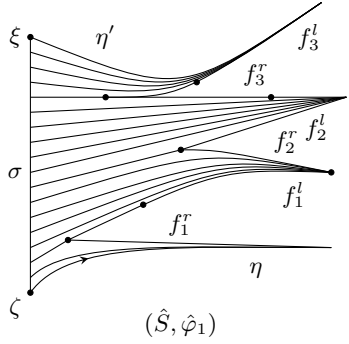
\begin{figure}[h]
\begin{center}
  \captionsetup{width=.85\linewidth}
\begin{tikzpicture}
\draw (-6,1.8) -- (-6,-1.6);
\draw (-6,1.8) .. controls (-4, 1) .. (-2.13,2.26);
\draw (-6,1.6) .. controls (-4,0.99) .. (-2.13,2.26);
\draw (-6,1.4) .. controls (-4,0.98) .. (-2.13,2.26);
\draw (-6,1.2) .. controls (-4,0.97) .. (-2.13,2.26);
\draw (-6,1) -- (-5,1); 
\draw (-5,1) .. controls (-4,1) .. (-2.13,2.26); 
\draw (-5,1) -- (-1.8,1); 
\draw (-6,0.8) -- (-1.8,1);
\draw (-6,0.6) -- (-1.8,1);
\draw (-6,0.4) -- (-1.8,1);
\draw (-6,0.2) -- (-1.8,1);
\draw (-6,0) -- (-1.8,1);
\draw (-6,-0.2) -- (-4,0.3); 
\draw (-4,0.3) -- (-1.8,1); 
\draw (-4,0.3) .. controls (-3.5, 0.3) .. (-2,0);
\draw (-6,-0.4) .. controls (-3.7, 0.3) .. (-2,0);
\draw (-6,-0.6) .. controls (-3.5, 0.2) .. (-2,0);
\draw (-6,-0.8) .. controls (-3.5, 0.15) .. (-2,0);
\draw (-6,-1) .. controls (-3.5, 0.1) .. (-2,0);
\draw (-6,-1.2) -- (-5.5,-0.9); 
\draw (-5.5,-0.9) .. controls (-3.5, 0.05) .. (-2,0);
\draw (-5.5,-0.9) -- (-2,-1);
\draw (-6,-1.4) .. controls (-5.5,-0.9) and (-4,-1) .. (-2,-1);
\draw [directed](-6,-1.6) .. controls (-5.5,-0.9) and (-4,-1) .. (-2,-1);
\filldraw (-2,0) circle (1pt);
\filldraw (-6, 1.8) circle (1pt);
\node at (-6.2,1.8) {\small $\xi$};
\filldraw (-6, -1.6) circle (1pt);
\node at (-6.2,-1.8) {\small $\zeta$};
\filldraw (-5,1) circle (1pt);
\filldraw (-4,0.3) circle (1pt);
\filldraw (-5.5,-0.9) circle (1pt);
\filldraw (-3.789,1.2) circle (1pt);
\filldraw (-2.8,1) circle (1pt);
\filldraw (-4.5,-.43) circle (1pt);
\node at (-6.2,0) {\small $\sigma$};
\node at (-5, 1.8) {\small $\eta'$};
\node at (-2.3, 1.8) {\small $f_3^l$};
\node at (-3, 1.3) {\small $f_3^r$};
\node at (-2.2, 0.6) {\small $f_2^l$};
\node at (-2.6, 0.4) {\small $f_2^r$};
\node at (-2.5, -0.3) {\small $f_1^l$};
\node at (-4, -0.7) {\small $f_1^r$};
\node at (-3,-1.8+0.5) {\small $\eta$};
\node at (-4,-2) {\small $(\hat{S}, \hat{\varphi}_1)$};
\end{tikzpicture}
\caption{the "polygon" $\hat{T}$ created in $(\hat{S}, \hat{\varphi}_1)$.}
\label{hatT}
\end{center}
\end{figure}

\paragraph{Intersections and orientations preserved.}  Going forward, we abuse notation and let
\[ \sigma, f_0^l,f_1^r, f_1^l,\ldots,f_n^r,f_n^l,f_{n+1}^r\]
denote oriented basic $\hphi_1$--geodesic containing the arc of the same name.  For $f_0^r$ and $f_{n+1}^l$, this basic $\hphi_1$--geodesic is $\eta$ and $\eta'$, respectively.  For $1 \leq i \leq n$, we assume the basic $\hphi_1$--geodesics $f_i^l$ and $f_i^r$ are as described above, thus making angles $\pi$ on the left and right, respectively, at every cone point they meet.  Finally, we choose the extension for $\sigma$ so that it crosses $\eta$ and $\eta'$ at $\zeta$ and $\xi$, respectively.


The pairs $(f_i^l, f_i^r)$ have the same $D_1$--image in $\mathbb{H}$ because they share a common subarc from their point of intersection with $\sigma$ to their first cone point after that in the forward direction. Therefore the basic $\hat{\varphi}_2$--geodesics $\hat{g}(f_i^l)$ and $\hat{g}(f_i^r)$ must also have the same $D_2$--images since they also share a common arc (since $f_i^l$ and $\hat g(f_i^l)$ have the same combinatorics, as do $f_i^r$ and $\hat g(f_i^r)$, by Lemma~\ref{L:hat combinatorics}).

We also observe that $f_i^r$ and $f_i^l$ cross $\sigma$ from left to right and $f_i^l$ crosses $f_{i+1}^r$ from right to left, for all indices $i$ for which the corresponding geodesics are defined.  By construction, these geodesics clearly intersect in a single point (possibly a cone point) and cross because the intersection is transverse or because of the choices of the sides on which geodesics make angle $\pi$ if at a cone point.  Since the crossing of geodesics and orientations of such (e.g.~left to right or right to left) is encoded by the endpoints at infinity, it follows that
 $\hat g(f_i^r)$ and $\hat g(f_i^l)$ cross $\hat g(\sigma)$ from left to right and $\hat g(f_i^l)$ crosses $\hat g(f_{i+1}^r)$ from right to left, for all  $i$ as above.  Moreover, because $\delta$ and $\hat g(\delta)$ have the same combinatorics for every $\delta \in \CG(\hphi_1)$ (Lemma~\ref{L:hat combinatorics}), these intersections occur at a single point.


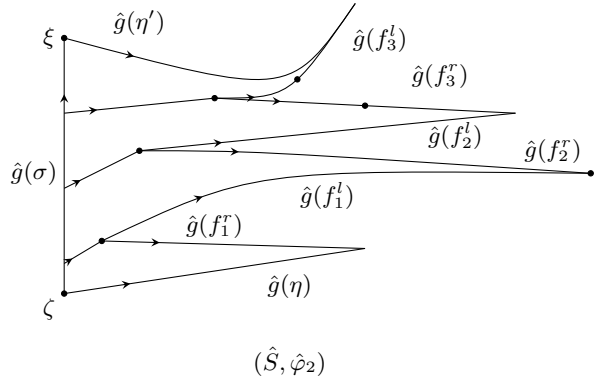
\begin{figure}[h]
\begin{center}
  \captionsetup{width=.85\linewidth}
\begin{tikzpicture}
\draw [reverse 
directed] (-6,1.8) -- (-6,-1.6);
\draw [directed] (-6,1.8) .. controls (-3, 1) .. (-2.13,2.26);
\draw [directed] (-6,0.8) -- (-4,1); 
\draw [directed](-4,1) .. controls (-3,1) .. (-2.13,2.26); 
\draw [directed](-4,1) -- (0,0.8); 
\draw [directed](-6,-0.2) -- (-5,0.3); 
\draw [directed](-5,0.3) -- (0,0.8); 
\draw [directed](-5,0.3) .. controls (-3.5, 0.3) .. (1,0);
\draw [directed](-6,-1.2) -- (-5.5,-0.9); 
\draw [directed](-5.5,-0.9) .. controls (-3.5, 0.05) .. (1,0);
\draw [directed](-5.5,-0.9) -- (-2,-1);
\draw [directed](-6,-1.6) -- (-2,-1);
\filldraw (1,0) circle (1pt);
\filldraw (-6, 1.8) circle (1pt);
\node at (-6.2,1.8) {\small $\xi$};
\filldraw (-6, -1.6) circle (1pt);
\node at (-6.2,-1.8) {\small $\zeta$};
\filldraw (-4,1) circle (1pt);
\filldraw (-5,0.3) circle (1pt);
\filldraw (-5.5,-0.9) circle (1pt);
\filldraw (-2.9,1.25) circle (1pt);
\filldraw (-2, 0.9) circle (1pt);
\node at (-6.4,0) {\small $\hat g(\sigma)$};
\node at (-5, 2) {\small $\hat{g}(\eta')$};
\node at (-1.8, 1.8) {\small $\hat{g}(f_3^l)$};
\node at (-1, 1.3) {\small $\hat{g}(f_3^r)$};
\node at (-0.8, 0.5) {\small $\hat{g}(f_2^l)$};
\node at (0.5, 0.3) {\small $\hat{g}(f_2^r)$};
\node at (-2.5, -0.3) {\small $\hat{g}(f_1^l)$};
\node at (-4, -0.7) {\small $\hat{g}(f_1^r)$};
\node at (-3,-1.5) {\small $\hat{g}(\eta)$};
\node at (-3,-2.5) {\small $(\hat{S}, \hat{\varphi}_2)$};
\end{tikzpicture}
\caption{The $\hphi_2$--straightenings of the important $\hphi_1$--geodesics.}
\label{hatT2}
\end{center}
\end{figure}

\paragraph{When developed, $\hat{g}(\eta)$ and $\hat{g}(\eta')$ intersect.}
As noted above, $D_2(\hat{g}(f_i^l)) = D_2(\hat{g}(f_i^r))$ since they share an arc.  We denote this image simply as $D_2(\hat{g}(f_i))$, noting that $D_2(\hat g(f_0)) = D_2(\hat g(\eta))$ and $D_2(\hat g(f_{n+1})) = D_2(\hat g(\eta'))$.  Since $\hat g(f_i^l)$ crosses $\hat g(f_{i+1}^r)$ from right to left for each $0 \leq i \leq n$ it follows that
\begin{enumerate}
\item[(1)] $D_2(\hat g(f_i))$ crosses $D_2(\hat g(f_{i+1}))$ from right to the left, for all $0 \leq i \leq n$.
\end{enumerate}
Similarly, since $\hat g(f_i^r)$ and $\hat g(f_i^l)$ cross $\hat g(\sigma)$ from left to right for all $1 \leq i \leq n$, as do $\hat g(\eta) = \hat g(f_0^l)$ and $\hat g(\eta') = \hat g(f_{n+1}^r$), we also have
\begin{enumerate}
\item[(2)]  $D_2(\hat g(f_i))$ crosses $D_2(\hat g(\sigma))$ from left to right for all $0 \leq i \leq n+1$.
\end{enumerate}

We now claim that condition (ii) in Proposition \ref{prop referee} holds for $\hat g(\eta)$, $\hat g(\eta')$, and $\hat g(\sigma)$. For this, note that the end points of $D_2 ( \hat{g} (\sigma))$ divide $\partial \mathbb{H}$ into two intervals: the left interval (left of $D_2 ( \hat{g} (\sigma))$) and the right interval (right of $D_2(\hat{g}(\sigma))$). Statement (2) implies that $D_2(\hat g(f_i))$ must have initial endpoint on the left interval and final endpoint on the right interval, for all $0 \leq i \leq n+1$.   Statement (1) implies that for all $0 \leq i \leq n$, the initial (final) endpoint of $D_2(\hat g(f_i))$ must occur clockwise of the initial (final) endpoint of $D_2(\hat g(f_{i+1}))$ within the left (right) intervals.  Since clockwise ordering is a total order on each of the left and right intervals, transitivity implies that the initial (final) endpoint of $D_2(\hat g(\eta')) = D_2(\hat g(f_{n+1}))$ must occur clockwise of the initial (final) endpoint of $D_2(\hat g(\eta)) = D_2(\hat g(f_0))$ within the left (right) intervals.
See figure \ref{linkedpairs}. This implies not only
that $D_2(\hat{g}(\eta))$ and $D_2(\hat{g}(\eta'))$ intersect (because their endpoints are linked), but that the intersection point must be forward of the intersections of $D_2 (\hat{g}(\eta))$ and $D_2 (\hat{g}(\eta'))$ with $D_2(\hat{g}(\sigma))$. The latter is because $D_2(\xi)$ is forwards of $D_2(\zeta)$ along $D_2(\hat{g}(\sigma))$. This proves condition (ii), and thus completes the proof.
\end{proof}

\newcommand{\segments}[4]{
\filldraw ({4*cos(#1)},{4*sin(#1)}) circle (#4 pt);
\filldraw ({4*cos(#2)},{4*sin(#2)}) circle (#4 pt);
\draw[#3] ({4*cos(#1)},{4*sin(#1)}) -- ({4*cos(#2)},{4*sin(#2)});}
\begin{figure}[htb]
\begin{center}
  \captionsetup{width=.85\linewidth}
\begin{tikzpicture}
\draw [gray] (0,0) circle (4);
\segments{230}{130}{directedd,ultra thick}{1.5}
\segments{220}{80}{directed}{1}
\segments{205}{40}{directedd}{1}
\segments{185}{0}{directed}{1}
\segments{160}{-35}{directedd}{1}
\segments{140}{-80}{directed}{1}

\node at (-1.8,-2.5) {\small $D_2(\hat g(\sigma))$};
\node at (-.5,3.3) {\small $D_2(\hat g(\eta))$};
\node at (1.9,2.5) {\small $D_2(\hat g(f_1))$};
\node at (3.1,.25) {\small $D_2(\hat g(f_2))$};
\node at (2.8,-1.4) {\small $D_2(\hat g(f_3))$};
\node at (1.2,-3.1) {\small $D_2(\hat g(\eta'))$};

\draw [domain=140:160, arrows=<-] plot ({4.2*cos(\x)}, {4.2*sin(\x)});
\draw [domain=160:185, arrows=<-] plot ({4.2*cos(\x)}, {4.2*sin(\x)});
\draw [domain=185:205, arrows=<-] plot ({4.2*cos(\x)}, {4.2*sin(\x)});
\draw [domain=205:220, arrows=<-] plot ({4.2*cos(\x)}, {4.2*sin(\x)});
\draw [domain=220:230, arrows=<-] plot ({4.2*cos(\x)}, {4.2*sin(\x)});

\draw [domain=-80:-35, arrows=<-] plot ({4.2*cos(\x)}, {4.2*sin(\x)});
\draw [domain=-35:0, arrows=<-] plot ({4.2*cos(\x)}, {4.2*sin(\x)});
\draw [domain=0:40, arrows=<-] plot ({4.2*cos(\x)}, {4.2*sin(\x)});
\draw [domain=40:80, arrows=<-] plot ({4.2*cos(\x)}, {4.2*sin(\x)});
\draw [domain=80:130, arrows=<-] plot ({4.2*cos(\x)}, {4.2*sin(\x)});
\end{tikzpicture}
\caption{Linking and ordering of end points of $D_2$--images of the important $\hphi_2$--straightened geodesics (shown here in the Klein model of $\mathbb{H}$).}
\label{linkedpairs}
\end{center}
\end{figure}
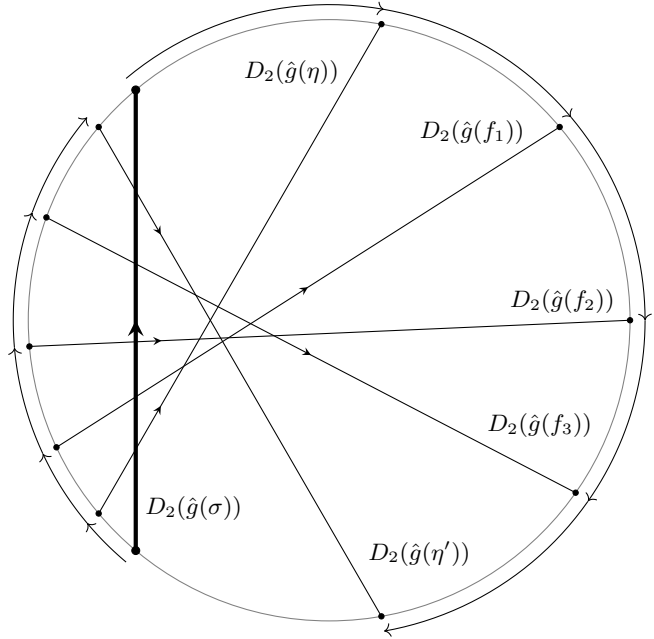


We now prove an asymptotic version of Proposition \ref{prop:forward_intersecting}.

\begin{proposition}
\label{prop:forward_asyptotic}
Suppose $\eta$ and $\eta'$ are oriented basic $\hphi_1$--geodesics that are developed forward asymptotic. Suppose there is a saddle connection $\sigma$ with one end point on $\eta$ and one end point on $\eta'$, such that the angles between $\sigma$ and the forward rays of $\eta$ and $\eta'$ lie in the interval $[0,\pi]$. Then $\hat{g}(\eta)$ and $\hat{g}(\eta')$ are developed forward asymptotic as well.
\end{proposition}
\begin{proof} Let $\theta_1, \theta_1' \in [0,\pi]$ be the angles between $\sigma$ and the forward rays of $\eta$ and $\eta'$, respectively, and $\theta_2,\theta_2'$ be the corresponding angles between $\hat g(\sigma)$ and the forward rays of $\hat{g}(\eta)$ and $\hat{g}(\eta')$, respectively.  If $\theta_1 = 0$ or $\pi$, then $\theta_1' = \pi$ or $0$, respectively, and then the $D_1$--image of $\sigma$, $\eta$, and $\eta'$ are all contained in a single geodesic. Appealing to Lemma~\ref{L:pi intervals} we see that the same is true of $\hat g(\sigma)$, $\hat g(\eta)$, and $\hat g(\sigma)$, completing the proof in these cases.

Now assume that $\theta_1,\theta_1' \in (0,\pi)$, so by Lemma~\ref{L:pi intervals}, $\theta_2,\theta_2' \in (0,\pi)$.  Let $\zeta$ and $\xi$ be the end points of $\sigma$ where $\eta$ and $\eta'$ cross it, respectively.  Observe that for all $\theta \in (0,\theta_1)$, any geodesic $\eta_\theta$ through $\zeta$ whose forward ray makes angle $\theta$ with $\sigma$ must have $D_1(\eta_\theta)$ intersecting $D_1(\eta')$ forward of $\sigma$.  By Proposition~\ref{prop:forward_intersecting}, $D_2(\hat g(\eta_\theta))$ must intersect $D_2(\hat g(\eta'))$ forward of $D_2(\hat g(\sigma))$.  

We may assume that as $\theta \to \theta_1$, $\eta_\theta \to \eta$, after replacing $\eta$ and each $\eta_\theta$ by geodesics through $\zeta$ with the same $D_1$--images, if necessary.  Therefore as $\theta \to \theta_1$, we have $\hat g(\eta_\theta) \to \hat g(\eta)$, and so also $D_2(\hat g(\eta_\theta)) \to D_2(\hat g(\eta))$.  Since $D_2(\hat g(\eta_\theta))$ intersects $D_2(\hat g(\eta'))$ forward of $D_2(\hat g(\sigma))$ for all $\theta \in (0,\theta_0)$, we see that either $D_2(\hat g(\eta))$ intersects $D_2(\hat g(\eta'))$ forward of $D_2(\hat g(\sigma))$ or else $D_2(\hat g(\eta))$ and $D_2(\hat g(\eta'))$ are forward asymptotic. We cannot be in the former situation since reversing the roles of $\hphi_1$ and $\hphi_2$ and applying Proposition~\ref{prop:forward_asyptotic} to $\hat g(\sigma)$, $\hat g(\eta)$, and $\hat g(\eta')$, would imply $D_1(\eta)$ and $D_1(\eta')$ intersect forward of $D_1(\sigma)$, a contradiction.  We are thus in the latter situation, and so $\hat g(\eta)$ and $\hat g(\eta')$ are developed forward asymptotic. This completes the proof.
\end{proof}

Now, given a $\hphi_i$--saddle connection $\sigma$ and $x \in \partial \mathbb H$, we consider the family $\Lambda(\sigma,x) \subset \CG(\hphi_i)$ which consists of the set of basic $\hphi_i$--geodesics $\eta$ intersecting $\sigma$ that may be oriented so that $D_i(\eta)$ limits to $x$ in the forward direction.  We view $\Lambda(\sigma,x)$ as a subset of $\CG(\hphi_i)$, together with a choice of orientation for each geodesic.

If $D_i(\sigma)$ can be extended to an oriented geodesic $\eta$ with $x$ as its forward end point, then $\Lambda(\sigma,x)$ is the set of all oriented basic $\hphi_i$--geodesics meeting $\sigma$ with developed image equal to $\eta$. Note that in this case $\Lambda(\sigma,x)$ is countable.

If not, then note that $\Lambda(\sigma,x)$ contains the pullback (as in the proof of Proposition~\ref{prop:forward_intersecting}) of the family of geodesics in $\mathbb{H}$ that limit to $x$ in the forward direction and intersect $D_i(\sigma)$ transversely. In particular, in this case, $\Lambda(\sigma,x)$ is uncountable.

\begin{proposition}\label{asymptotic_family} For any $\hphi_1$--saddle connection $\sigma$ and $x \in \partial \mathbb H$, there exists a point $h(\sigma,x) \in \partial \mathbb H$ so that
\[ \hat g(\Lambda(\sigma,x)) = \Lambda(\hat g (\sigma),h(\sigma,x)).\]
\end{proposition}
We clarify that $\hat g(\Lambda(\sigma,x)) \subset \CG(\hphi_2)$ and each geodesic in the $\hat g$--image is assigned the orientation coming from its orientation in $\Lambda(\sigma,x)$. 

\begin{proof}
We begin by addressing the case where  $\Lambda(\sigma,x)$ is countable.  In this situation, $D_1( \Lambda(\sigma,x))$ is a single oriented geodesic in $\mathbb{H}$. By Lemma~\ref{L:pi intervals}, $\hat{g} ( \Lambda(\sigma,x))$ consists of oriented $\hphi_2$--geodesics that develop to a single geodesic $\eta$ in $\mathbb{H}$ which contains $D_2(\hat g(\sigma))$. Letting the forward end point of this geodesic be $h(\sigma, x)$ gives us that $\hat{g} ( \Lambda(\sigma,x)) \subset \Lambda(\hat g(\sigma),h(\sigma,x))$.
Reversing the roles of $\hphi_1$ and $\hphi_2$ proves the other inclusion.

In the case where $\Lambda(\sigma,x)$ is uncountable, $D_1 (\Lambda(\sigma, x))$ consists of all oriented geodesics in $\mathbb H$ that intersect $D_1(\sigma)$ and have $x$ as their forward end point. As in Proposition \ref{prop:forward_intersecting}, we systematically pull back each of these geodesics to one or two $\hphi_1$--geodesics to get a family $\mathcal{R} \subseteq \Lambda(\sigma,x)$.

More precisely, we begin by pulling back short segments of forward asymptotic geodesics that intersect the interior of $D_1 (\sigma)$ and extending them to basic geodesics in $(\hat{S}, \hat{\varphi}_1)$. Again, as in Proposition~\ref{prop:forward_intersecting}, if in the extension process, a cone point is encountered, we continue extending in two ways: by always making angle $\pi$ on the right, and by always making angle $\pi$ on the left. At most countably many of these pullback geodesics, $\{f_i^l,f_i^r\}_{i \in A}$, where $A$ is a countable indexing set, are such extensions (because there are only countably many cone points in $(\hat{S}, \hat{\varphi}_1)$). The superscript tells us which side of the oriented basic geodesic makes angle $\pi$. Again as in Proposition \ref{prop:forward_intersecting}, this construction always has $D_1 (f_i^l) = D_1 (f_i^r)$. For the two geodesics in $\mathbb{H}$ intersecting the end points of $D_{\sigma}$, we also include in $\mathcal{R}$ the two geodesics $\eta_0$ and $\eta_1$ in $\Lambda(\sigma,x)$ that intersect the end points of $\sigma$ and whose forward rays make an angle less than $\pi$ with $\sigma$.

We show that $\hat g(\mathcal{R})$ consists of basic $\hphi_2$--geodesics that are developed forward asymptotic, and use this to prove the desired result. The geodesics of $\hat{g}(\mathcal R)$  in $(\hat{S}, \hat{\varphi}_2)$ must also be a family of pairwise non-intersecting geodesics since the end points of any two are not linked, and because the cone points visited by each geodesic are preserved by the straightening map $\hat{g}$. Betweenness of geodesics is also preserved by $\hat{g}$. This implies that the geodesics of $\hat{g}(\mathcal R)$ intersect $\hat{g}(\sigma)$ in the same order as in $(\hat{S}, \hat{\varphi}_1)$. In particular, $\hat{g}(\mathcal R)$ is a family of geodesics between $\hat{g}(\eta_0)$ and $\hat{g}(\eta_1)$. Finally, $D_2  (\hat{g} (f_i^l)) = D_2 (\hat{g}(f_i^r))$, by Lemma~\ref{L:pi intervals}, and we denote this geodesic by $D_2(\hat{g}(f_i))$.

Consider the region $\mathcal{S}$ of $(\hat{S}, \hat{\varphi}_2)$ given by the union of geodesics in $\hat{g}(\mathcal{R})$. Using an argument similar to the proof of  Proposition~\ref{prop:forward_intersecting} (finitely many cone points), we can show that $\mathcal{S}$ projects injectively to $\tS$ and therefore has the following property. Given any distance $d$, there are only finitely many cone points in $\mathcal{S}$ distance $d$ away from $\sigma$. This implies that $\mathcal{S}$ is a hyperbolic surface with boundary. $\mathcal{S}$ does not contain any cone points in its interior, but it does have at most countably many cone points on its boundary. Through the $i$th cone point, there are two geodesics of $\hat{g}(\mathcal{R})$: $\hat{g}(f_i^l)$ and $\hat{g}(f_i^r)$. Extending the overlap of $\hat{g}(f_i^l)$ and $\hat{g}(f_i^r)$ to an isometric identification of the two geodesics (``zipping up'' the slits of $\mathcal{S}$), we obtain a new surface $\mathcal{S}_{\text{zip}}$. Note that the metric on $\mathcal{S}_{\text{zip}}$ is nonsingular, since after the identification, the total angle at each cone point is $2\pi$. Therefore, the restriction of $D_2$ to $\mathcal S$ descends to an isometry $\mathcal S_{\text{zip}} \to D_2(\mathcal{S})$. Since $\hat{g}(\mathcal{R})$ is a set of non-intersecting geodesics between these two, the image of $\hat{g}(\mathcal{R})$ in $\mathcal{S}_{\text{zip}}$ is a set of parallel geodesics between the images of $\hat{g}(\eta_0)$ and $\hat{g}(\eta_1)$. Further since $\hat{g}(\eta_0)$ and $\hat{g}(\eta_1)$ are developed forward asymptotic, the image of $\hat{g}(\mathcal{R})$ must be forward asymptotic (a family of parallel geodesics between two forward asymptotic geodesics in the hyperbolic plane must all be forward asymptotic to each other).

Note that every $\hphi_1$-geodesic $\eta \in \Lambda(\sigma,x)$ is either contained in $\mathcal{R}$ or there exists $\eta' \in \mathcal R$ such that $\eta$ and $\eta'$ intersect at a cone point and $D_1(\eta) = D_1(\eta')$. Using Lemma~\ref{L:pi intervals}, we then have $D_2(\hat g(\Lambda(\sigma,x))) = D_2(\hat g(\mathcal{R}))$. Therefore, $\hat g(\Lambda(\sigma,x))$ is a developed forward asymptotic family, and we let $h(\sigma, x)$ be the forward end point of $D_2(\hat g(\Lambda(\sigma,x)))$. From here, since every geodesic of $\hat g(\Lambda(\sigma,x))$ must instersect $\hat g (\sigma)$, the inclusion $\hat g(\Lambda(\sigma,x)) \subseteq \Lambda(\hat g (\sigma),h(\sigma,x))$ follows by definition, and the reverse inclusion follows by a symmetric argument with the roles of $\hphi_1$ and $\hphi_2$ exchanged.
\end{proof}

Finally, we are able to show the desired result that developed forward asymptoticity is preserved by the straightening map

\begin{proposition}
\label{P:developed_forward_asymptotic}
Suppose $\eta_1$ and $\eta_2$ are developed forward asymptotic $\hphi_1$--geodesics. Then $\hat{g}(\eta_1)$ and $\hat{g}(\eta_2)$ are developed forward asymptotic as well.
\end{proposition}

\begin{proof}
If $\eta_1$ and $\eta_2$ meet in a segment or cone point, then $\hat g(\eta_1)$ and $\hat g(\eta_2)$ are developed forward asymptotic by Lemma~\ref{L:pi intervals}.  For the rest of the proof, suppose that this is not the case.
The geodesics $\eta_1$ and $\eta_2$ must not intersect transversely in $(\hat{S}, \hat{\varphi}_1)$ (otherwise they would intersect transversely when developed). Therefore, they partition $\hat{S}$ into three regions: one bounded by $\eta_1$, one bounded by $\eta_2$ and one bounded by both. Pick a cone point $\zeta$ in the region bounded by $\eta_1$ alone and another cone point $\xi$ in the region bounded by $\eta_2$ alone. Take the geodesic segment between $\zeta$ and $\xi$. It must be a concatenation of saddle connections. After forgetting some initial and final saddle connections, we may assume the remaining concatenation $\sigma_1, \sigma_2, \ldots \sigma_k$ has only $\sigma_1$ intersecting $\eta_1$ and only $\sigma_k$ intersecting $\eta_2$.

Let $z$ be the forward end point of $D_1 (\eta_1)$ and $D_1(\eta_2)$ in $\partial \mathbb{H}$. Then consider the developed forward asymptotic families $\Lambda(\sigma_i, z)$. We know that $\eta_1 \in \Lambda(\sigma_1, z)$ and $\eta_2 \in \Lambda(\sigma_k, z)$, and from Proposition~\ref{asymptotic_family} we know that $\hat g (\Lambda(\sigma_i, z))= \Lambda(\hat g (\sigma_i), h(\sigma_i, z))$. To complete our proof, we will show that $\cup_i \Lambda(\hat g (\sigma_i), h(\sigma_i, z))$ is a developed forward asymptotic family by showing that $h(\sigma_1, z)=h(\sigma_2, z) = \ldots = h(\sigma_n, z)$.

Let $c_i$ be the cone point $\sigma_i \cap \sigma_{i+1}$, the intersection of two consecutive saddle connections. The intersection of any two consecutive families $\Lambda(\sigma_i, z) \cap \Lambda(\sigma_{i+1}, z)$ consists of the basic $\hphi_1$--geodesics through $c_i$ that develop to have $z$ as their forward end point. This set is countably infinite, so in particular, it is not empty. For any $\eta \in \Lambda(\sigma_i, z) \cap \Lambda(\sigma_{i+1}, z)$, the $\hphi_2$--straightening $\hat g (\eta)$ is contained in $\Lambda(\hat g(\sigma_i), h(\sigma_i z)) \cap \Lambda(\hat g (\sigma_{i+1}), h(\sigma_{i+1}, z))$. The geodesic $\hat g (\eta)$ cannot have two distinct forward end points, and so $h(\sigma_{i}, z)=h(\sigma_{i+1}, z)$. This argument holds for all $i$, and so $\cup_i \Lambda(\hat g (\sigma_i), h(\sigma_i, z))$ is a developed forward asymptotic family. In particular, this implies that $\hat{g}(\eta_1)$ and $\hat{g}(\eta_2)$ are developed forward asymptotic.
\end{proof}

As proposed earlier, we may now define $h: \partial \mathbb{H} \rightarrow \partial \mathbb{H}$ by requiring $h(x)$ to be the forward endpoint of $D_2(\hat g(\eta))$ where $\eta \in \CG(\hphi_1)$ is a basic geodesic oriented so that $x$ is the forward endpoint of $D_1(\eta)$, and Proposition~\ref{P:developed_forward_asymptotic} implies that this is indeed well-defined.  Moreover, by the proposition, for any $\hphi_1$--saddle connection $\sigma$, we have $h(x) = h(\sigma,x)$ where $h(\sigma,x)$ is as in Proposition~\ref{asymptotic_family}.

At long last, we are able to prove Proposition~\ref{P:holonomy conjugation}.

\begin{proof}[Proof of Proposition \ref{P:holonomy conjugation}]
The definition of $h$ implies that for any biinfinite basic $\hat \varphi_1$--geodesic $\eta$, $h(\partial D_1(\eta)) = \partial D_2(\hat g (\eta))$. It remains to show that $h$ is an orientation preserving homeomorphism, and that it topologically conjugates the holonomies.

Recall that the developing maps $D_i$ can be extended to $\bar D_i \colon \hS \cup \mathcal K_i \to \overline{\mathbb H}$, where $\mathcal K_i$ is the set of end points of basic $\hphi_i$--geodesics (see discussion after Lemma~\ref{L:Dev surjective}). Note that $\mathcal K_1= \mathcal K_2=: \mathcal K$ since $\CG_{\hphi_1}=\CG_{\hphi_2}$. Using this extension, we may write $h(\bar D_1(k))=\bar D_2(k)$, for any $k \in \mathcal K \subseteq \hat S_{\infty}^1$.

If $h$ were not injective, there would be two developed forward asymptotic families of oriented basic $\hphi_1$--geodesics that are not developed forward asymptotic to each other, however their images by $\hat g$ are developed forward asymptotic to each other. This contradicts Proposition \ref{P:developed_forward_asymptotic}, and thus $h$ is injective. The map $h$ is surjective because $\bar D_2$  maps $\mathcal K$ onto $\partial \mathbb H$ by Corollary~\ref{C:extendD}. So, for every $y \in \partial \mathbb{H}$, there exists $k \in \mathcal K$ such that $\bar D_2(k) = y$. Then $h(\bar D_1(k))=\bar D_2(k)$, so we have $h(\bar D_1(k))=y$, as desired.

In order to show $h$ is a homeomorphism, we show that $h$ and $h^{-1}$ are both open maps. Let $U$ be an open interval in $\partial \mathbb{H}$. Choose a point $\zeta_1 \in \mathbb{H}$ such that $\zeta_1 = D_1(\zeta)$, where $\zeta$ is a cone point in $(\hat{S}, \hat{\varphi}_1)$. Consider the set of oriented geodesics through $\zeta_1$ that have forward end point in $U$. These geodesics sweep out some interval of angles $(\alpha, \beta)$ measured at $\zeta_1$. From the description of neighborhoods of cone points in $\hS$ in \S\ref{S:puncturing surfaces}, we may choose a family of basic $\hphi_1$-geodesics $\Lambda$ through $\zeta$ that sweep out an interval of angles of the same measure as $(\alpha,\beta)$ which develop to the above mentioned set of geodesics in $\mathbb{H}$.
The $\hat{\varphi}_2$--straightening of this family, $\hat{g}(\Lambda)$, is a family of basic geodesics in $(\hat{S}, \hat{\varphi}_2)$ that are concurrent at $\zeta$ and sweep out an open interval of angles at $\zeta$, though the measure of this interval is not necessarily $\beta - \alpha$. The image $D_2 (\hat{g}(\Lambda))$ must be a family of concurrent geodesics, passing through $\zeta_2 = D_2(\zeta)$, that sweep out an open interval of angles at $\zeta_2$. Therefore the forward end points of the geodesics in $D_2 (\hat{g}(\Lambda))$ must form an open interval in $\partial \mathbb{H}$. This implies that $h$ is an open map.
The same argument with the subscripts 1 and 2 interchanged shows that $h^{-1}$ is an open map. Therefore $h$ is a homeomorphism. Moreover, sweeping through angles counterclockwise with respect to $\hphi_1$ is the same as sweeping through them counterclockwise with respect to $\hphi_2$, and thus $h$ is orientation preserving.

We now show that $h$ topologically conjugates $\rho_1$ to $\rho_2$. That is, for $\gamma \in G$ and $x \in \partial \mathbb H$, we prove Equation~\eqref{E:conjugating homs}.  
Recall that since $\gamma$ acts on $\hS$ by isometries with respect to both $\hphi_1$ and $\hphi_2$, it extends to $\hS^1_\infty$, and thus $\mathcal K$.  Since $D_i$ also extends to $\mathcal K$, we have:
\[ \begin{array}{rclll} 
h(\rho_1(\gamma) \cdot x) & = & h(\rho_1(\gamma) \cdot \bar D_1(k)), &\quad & \text{ for some } k \in \mathcal K.\\
&=& h(\bar D_1(\gamma \cdot k)) & & \text{ by definition of the holonomy homomorphism.}\\
&=& \bar D_2(\gamma \cdot k) & & \text{ by definition of $h$.}\\
&=& \rho_2 (\gamma) \cdot \bar D_2(k) & &  \text{ by definition of the holonomy homomorphism.}\\
&=& \rho_2 (\gamma) \cdot h(\bar D_1 (k)) & & \text{ by definition of } h \text{ and extensions of } D_i.\\
&=& \rho_2 (\gamma) \cdot h(x) & & \text{ by definition of }k.
\end{array}
\]
\end{proof}

\subsection{Rigidity from $h \in \PSL_2(\mathbb R)$}

We end this section by noting that for the rigidity statement in the \currentsupport\!\!, we want the homeomorphism $h$ from Proposition~\ref{P:holonomy conjugation} to be in $\PSL_2(\mathbb R)$.

\begin{proposition} \label{P:first rigidity}
If $\CG_{\tphi_1} = \CG_{\tphi_2}$ and $h \colon \partial \mathbb H \to \partial \mathbb H$ from Proposition~\ref{P:holonomy conjugation} lies in $\PSL_2(\mathbb R)$, then $\varphi_1\!\sim\!\varphi_2$. 
\end{proposition} 

This proposition will follow easily from the next lemma. 

\begin{lemma} \label{L:all angles preserved} Suppose $\CG_{\tphi_1} = \CG_{\tphi_2}$ and $h \colon \partial \mathbb H \to \partial \mathbb H$ from Proposition~\ref{P:holonomy conjugation} lies in $\PSL_2(\mathbb R)$.  If $\varphi_2$ is $C_0$--uber-normalized, with $C_0 = \cone(\varphi_1) = \cone(\varphi_2)$, then for any two basic $\hphi_1$--geodesics $\eta,\eta' \in \CG(\hphi_1)$ that intersect making angle $\theta<\pi$, $\hat g(\eta)$ and $\hat g(\eta')$ also intersect making angle $\theta$.
\end{lemma}
\begin{proof} From the hypotheses and Lemma~\ref{L:hat combinatorics}, we have $\CG_{\hphi_1} = \CG_{\hphi_2}$, $\cone(\hphi_1) = \cone(\hphi_2) = \hC_0$, and since $\eta$ and $\eta'$ intersect in a single point making angle less than $\pi$, the same is true of $\hat g(\eta)$ and $\hat g(\eta')$ (by Lemma~\ref{L:pi intervals} if $\eta$ intersects $\eta'$ in a point of $\hC_0$).  Moreover the angle between $\eta$ and $\eta'$ (respectively, $\hat g(\eta)$ and $\hat g(\eta')$) are the same as their images by $D_1$ (respectively, $D_2$).  By Proposition~\ref{P:holonomy conjugation}, $h(\partial D_1(\eta)) = \partial D_2(\hat g(\eta))$ and $h(\partial D_1(\eta')) = \partial D_2(\hat g(\eta'))$.   Since $h \in \PSL_2(\mathbb R)$, we extend it to a map $\bar h \colon \overline {\mathbb H} \to \overline {\mathbb H}$ which is an isometry in $\mathbb{H}$ and
\[ \bar h(D_1(\eta)) = D_2(\hat g(\eta)) \quad \mbox{ and } \quad \bar h(D_1(\eta')) = D_2(\hat g(\eta')).\]
Since $\bar h$ is an isometry in $\mathbb{H}$, the angle between $D_1(\eta)$ and $D_1(\eta')$ is equal to the angle between their $\bar h$--images, $D_2(\hat g(\eta))$ and $D_2(\hat g(\eta'))$.
\end{proof}

\begin{proof}[Proof of Proposition~\ref{P:first rigidity}] After adjusting $\varphi_2$ if necessary, we may assume it has been $C_0$--uber-normalized with respect to $\varphi_1$, so that $\cone(\tphi_1)=\cone(\tphi_2)=\tC_0$, and thus $\CG_{\hphi_1} = \CG_{\hphi_2}$ and $\cone(\hphi_1) = \cone(\hphi_2) = \hC_0$ by Lemma~\ref{L:hat combinatorics}.  Let $\mathcal{T}$ be a $(C_0, \varphi_1)$--triangulation of $S$ and further adjust $\varphi_2$ by an isotopy if necessary so that it is a $(C_0, \varphi_1, \varphi_2)$--triangulation. Let $\widehat{\mathcal{T}}$ be the lift of $\mathcal{T}$ to $\hat S$. 

Since $\widehat{\mathcal{T}}$ is simultaneously a $(C_0, \tphi_1)$--triangulation and a $(C_0, \tphi_2)$--triangulation, the image of each triangle $\widehat T$ in $\widehat{\mathcal{T}}$ under both developing maps is a geodesic triangle in $\mathbb{H}$. Further $D_1(\widehat T)$ is isometric to $(\widehat T, \hphi_1)$ and $D_2(\widehat T)$ is isometric to $(\widehat T, \hphi_2)$. This is because the developing maps are local isometries on sets that do not contain cone points in their interiors, and since all interior angles of triangles measure less than $\pi$, this can be promoted to an isometry in the case of a triangle.  

For each vertex of each triangle of $\widehat{\mathcal{T}}$, extending the two sides adjacent to that vertex to basic $\hphi_1$--geodesics $\eta$ and $\eta'$, Lemma~\ref{L:all angles preserved} implies the angle between $\eta$ and $\eta'$ agrees with the angle between $\hat g(\eta)$ and $\hat g(\eta')$.  Since $\hat g(\eta)$ and $\hat g(\eta')$ are extensions of the sides of the triangle to basic $\hphi_2$--geodesics, it follows that the interior angle of the triangle is the same when measured in both metrics.  
Since the identity on $\hat S$ is a canonical map on each triangle with respect to $\hphi_1$ on the domain and $\hphi_2$ on the range, and since the interior angles of the triangles are equal, this map is an isometry.  Therefore, the identity $(\hS,\hphi_1) \to (\hS,\hphi_2)$ is an isometry, and hence so is the identity $(S,\varphi_1) \to (S,\varphi_2)$, proving $\varphi_1 \sim \varphi_2$.
\end{proof}

\section{Proof of the \currentsupport} \label{S:proof of current}

The conclusion of the \currentsupport has two parts: a rigidity statement and a flexibility statement.  Proposition~\ref{P:first rigidity} says that to prove the rigidity statement---that is, if $\CG_{\tphi_1} = \CG_{\tphi_2}$ then $\varphi_1 \sim \varphi_2$---it suffices to show that the homeomorphism $h$ from Proposition~\ref{P:holonomy conjugation} lies in $\PSL_2(\mathbb R)$.   Proposition~\ref{p:subgroups} easily implies that if $\Gamma_1=\rho_1(\pi_1\dot S)$ is indiscrete, then $h$ lies in $\PSL_2(\mathbb R)$.  In this section we first describe a weaker condition that still suffices to guarantee that $h$ lies in $\PSL_2(\mathbb R)$.  We then show that the only way this condition can fail is if we are in the situation described in  Theorem~\ref{T:deforming}, proving the flexibility statement.

Throughout this section, we will assume $\varphi_1,\varphi_2 \in \tHyp_c(S)$ with $\CG_{\tphi_1} = \CG_{\tphi_2}$ and that $C_0 = \cone(\varphi_1) = \cone(\varphi_2)$ with $\varphi_2$ being $C_0$--uber-normalized (so that $\CG_{\hphi_1} = \CG_{\hphi_2}$ by Lemma~\ref{L:hat combinatorics}).  For each $i=1,2$, recall that $D_i \colon \hat S \to S$ is the developing map for $\hphi_i$, $\rho_i$ is the holonomy homomorphism defined on $\pi_1\dot{S}$ (or any group acting by homeomorphisms, isometrically with respect to both $\hphi_1$ and $\hphi_2$), and $\Gamma_i = \rho_i(\pi_1\dot{S})$ with $q_i \colon (S,\varphi_i) \to \mathcal O_i = \mathbb H/\Gamma_i$, the quotient space by $\Gamma_i$ (which is an orbifold if $\Gamma_i$ is discrete).

\subsection{Rigidity}\label{S:rigidity}

The right condition to ensure rigidity turns out to involve a certain enlargement of $\Gamma_1$ which we describe below.  First, we need the following fact.  Given $\xi \in \mathbb H$, let $\tau_\xi$ denote the involution given by a rotation of order $2$ about $\xi$.

\begin{proposition} \label{P:extending involutions} 
Suppose $\CG_{\tphi_1} = \CG_{\tphi_2}$, $\varphi_2$ is $C_0$--uber-normalized, and $h \colon \partial \mathbb H \to \partial \mathbb H$ is the homeomorphism from Proposition~\ref{P:holonomy conjugation}.  Given $\zeta \in \hC_0$, set $\zeta_i = D_i(\zeta)$, for $i=1,2$.  Then $h \tau_{\zeta_1} h^{-1} = \tau_{\zeta_2}$.
\end{proposition}
\begin{proof}
Fix any $\zeta \in \hC_0$. By Corollary~\ref{C:hat concurrency}, $\hat g(\CG(\hphi_1,\zeta)) = \CG(\hphi_2,\zeta)$. Moreover, for any $\eta \in \CG(\hphi_1,\zeta)$ we have $h(\partial D_1(\eta)) = \partial D_2(\hat g(\eta))$ by Proposition~\ref{P:holonomy conjugation}.
Next observe that $\tau_{\zeta_1}$ interchanges the endpoints of $D_1(\eta)$, while $\tau_{\zeta_2}$ interchanges the endpoints of $D_2(\hat g(\eta))$.  Therefore, $\partial D_1(\eta)$ has the form $\{x,\tau_{\zeta_1}(x)\}$ for some $x \in \partial \mathbb H$, and $\partial D_2(\hat g(\eta)) = \{h(x),\tau_{\zeta_2}(h(x))\}$.  Combining these facts, we have $\tau_{\zeta_2}(h(x)) = h(\tau_{\zeta_1}(x))$.  Every $x \in \partial \mathbb H$ is in $\partial D_1(\eta)$ for some $\eta \in \CG(\hphi_1,\zeta)$ by Lemma~\ref{L:Dev surjective}, and thus $\tau_{\zeta_2} h = h \tau_{\zeta_1}$, as required.
\end{proof}

For each $\zeta \in \hC_0 \subset \hS$ and $i=1,2$, let $\zeta_i = D_i(\zeta)$, let $\tau_{\zeta_i}\in\PSL_2(\mathbb R)$ be the involution as above, and set 
\[\Gamma_i^0 = \langle \Gamma_i, \{\tau_{\zeta_i}\}_{\zeta \in \hat C_0} \rangle. \]

\begin{corollary} \label{C:extending by involutions}  Suppose $\CG_{\tphi_1} = \CG_{\tphi_2}$ and $\varphi_2$ is $C_0$--uber-normalized.  The homeomorphism $h \colon \partial \mathbb H \to \partial \mathbb H$ from Proposition~\ref{P:holonomy conjugation} topologically conjugates $\Gamma_1^0$ to $\Gamma_2^0$ sending $\tau_{\zeta_1}$ to $\tau_{\zeta_2}$ for all $\zeta \in \hC_0$.
\end{corollary}
\begin{proof} As a consequence of Proposition~\ref{P:holonomy conjugation}, $h$ topologically conjugates $\Gamma_1$ to $\Gamma_2$, and by Proposition \ref{P:extending involutions}, $h \tau_{\zeta_1} h^{-1}= \tau_{\zeta_2}$, for all $\zeta \in \hC_0$.  Therefore $h$ topologically conjugates $\Gamma^0_1$ to $\Gamma_2^0$. 
\end{proof}

For each $i=1,2$, we write $p_i \colon \mathbb H \to \mathcal O_i^0 = \mathbb H/\Gamma_i^0$ for the quotient space.  If $\Gamma_i^0$ is discrete, then $\mathcal O_i^0$ is a hyperbolic orbifold.  Otherwise, we view it simply as the quotient topological space.
Since $\rho_i(\pi_1\dot{S}) = \Gamma_i < \Gamma_i^0$, and the map $D_i \colon \hS \to \mathbb H$  is $\rho_i$--equivariant, it follows that $D_i$ descends to a continuous map $q_i^0 \colon (S,\varphi_i) \to \mathcal O_i^0$ fitting into the following diagram
\[ \xymatrix{ \hat S \ar[r]^{D_i} \ar[d]_{\hat p} & \mathbb H \ar[d]^{p_i} \\
S \ar[r]^{q_i^0} & \mathcal O_i^0.}\]
As with $q_i$, the map $q_i^0$ is a locally isometric branched cover with respect to $\varphi_i$ if $\Gamma_i^0$ is discrete, but is only a continuous map in general.

The next proposition provides a sufficient condition for rigidity.
\begin{proposition} \label{P:indiscrete is rigid} If $\CG_{\tphi_1} = \CG_{\tphi_2}$ and $\Gamma_1^0$ is indiscrete, then $\varphi_1\sim\varphi_2$.
\end{proposition}
\begin{proof}  Observe that $\mathcal O_1 = \mathbb H/\Gamma_1^0$ is compact, being the image of $S$ by $q_1^0$.  Adjusting $\varphi_2$ by a homeomorphism isotopic to the identity if necessary, we may assume that $\varphi_2$ is $C_0$--uber-normalized.  By Corollary~\ref{C:extending by involutions},  $\Gamma_1^0$ and $\Gamma_2^0$ are topologically conjugate by $h$, and since $\Gamma_1^0$ is indiscrete, Proposition~\ref{p:subgroups} implies $h \in \PSL_2(\mathbb R)$.  Therefore, $\varphi_1 \sim \varphi_2$ by Proposition~\ref{P:first rigidity}.
\end{proof}

The next corollary describes one situation where $\Gamma_1^0$ is discrete, and yet the homeomorphism $h$ topologically conjugating $\Gamma_1^0$ to $\Gamma_2^0$ is necessarily in $\PSL_2(\mathbb{R})$, giving us another case where we have rigidity.

\begin{corollary} \label{C:triangle is rigid} If $\CG_{\tphi_1} = \CG_{\tphi_2}$ and $\Gamma_1^0$ is a Fuchsian triangle group, then $\varphi_1 \sim \varphi_2$.
\end{corollary}
\begin{proof} Since $h$ conjugates $\Gamma_1^0$ to $\Gamma_2^0$, it follows from Lemma \ref{l:trianglegroup} that $h \in \PSL_2(\mathbb R)$, and Proposition~\ref{P:first rigidity} again implies $\varphi_1 \sim \varphi_2$.
\end{proof}

\subsection{Flexibility}\label{S:flexibility proof}

We now turn to the situation that $\CG_{\tphi_1} = \CG_{\tphi_2}$, but $\varphi_1 \not \sim \varphi_2$.  We continue with the assumptions on $\varphi_1$, $\varphi_2$, $C_0$, etc., and with the definitions of $\Gamma_i^0$, $\mathcal O_i^0$, etc., as above.  According to Corollary~\ref{C:extending by involutions}, $h$ topologically conjugates $\Gamma_1^0$ to $\Gamma_2^0$, and we let
\[ \psi^0 \colon \Gamma_1^0 \to \Gamma_2^0 \]
denote the isomorphism; that is, in terms of the action on $\partial \mathbb H$, $\psi^0(\gamma) = h \gamma h^{-1}$, for all $\gamma \in \Gamma_1^0$.

To prove the \currentsupport\!\!, we need to show that $\varphi_1$ comes from a branched cover of an orbifold and that $\varphi_2$ is (up to equivalence) obtained by deforming the orbifold as in Theorem~\ref{T:deforming}.  By Proposition~\ref{P:indiscrete is rigid}, we need only consider the case that $\Gamma_1^0$ (and hence also $\Gamma_2^0$) is discrete. The goal is to show that for the branched covers of orbifolds $q_i^0: S\to \mathcal O_i^0$, for $i=1,2$, there is a homeomorphism $F:\mathcal O_1^0 \to \mathcal O_2^0$ such that $q_2^0$ and $F\circ q_1^0$ differ by a homeomorphism of $S$ isotopic to the identity. Let $\mathcal{E}^0_i$ be the set of even order orbifold points in $\mathcal{O}_i^0$.  We begin with the following.

\begin{lemma} \label{L:preferred concurrence} Suppose $\CG_{\tphi_1} = \CG_{\tphi_2}$ and $\Gamma_1^0$ is discrete.  The set $C= (q^0_1)^{-1}(\mathcal{E}^0_1)$ is a concurrency set for $(\varphi_1,\varphi_2)$. 
\end{lemma}

\begin{proof} 
Note that by the definition of $\Gamma_1^0$, $C$ is a finite set and $\cone(\varphi_1)=C_0\subset C$. Now fix any $\zeta\in C$. If $\zeta\in C_0$ we know that it is a concurrency point and its lifts $\hat{p}^{-1}(\zeta)\subset \hat S$ are concurrency points by Corollary~\ref{C:hat concurrency}. So assume $\zeta\in (q_1^0)^{-1}(\mathcal{E}^0_1)\setminus C_0$ and consider any $\hat \zeta\in \hat{p}^{-1}(\zeta)\subset \hat S$.  Recall that $\CG(\hphi_1,\hat \zeta)$ consists of all basic $\hphi_1$--geodesics passing through $\hat \zeta$.  Since $\hat \zeta$ is not a cone point, any two geodesics in $\CG(\hphi_1,\hat \zeta)$ intersect in exactly the one point $\hat \zeta$.

\begin{claim} The set of basic $\hphi_2$--geodesics $\{ \hat g(\eta) \mid \eta \in \CG(\hphi_1,\hat \zeta)\}$ are concurrent.
\end{claim}
\begin{proof} First, for any two geodesics $\eta_1,\eta_2 \in \CG(\hphi_1,\hat \zeta)$, their $\hphi_2$--straightenings $\hat g(\eta_1),\hat g(\eta_2)$ must intersect in exactly one point, and that point is not a cone point.  To see this, note that the endpoints of $\hat g(\eta_i)$ are the same as the endpoints of $\eta_i$, for $i=1,2$, and the endpoints of $\eta_1$ and $\eta_2$ link.  The intersection is a single non-cone point because $\eta_i$ and $\hat g(\eta_i)$ have the same combinatorics, for $i=1,2$ by Lemma~\ref{L:hat combinatorics}.  Now, suppose $\eta_0,\eta_1,\eta_2 \in \CG(\hphi_1,\hat \zeta)$ are any three distinct geodesics, let $\hat g (\eta_0) \cap \hat g (\eta_1) = \hat\xi_1$ and $\hat g (\eta_0 ) \cap \hat g ( \eta_2) = \hat \xi_2$.  We must show that $\hat \xi_1 = \hat \xi_2$.  To do this, we first analyze the point $\hat \zeta$ in more detail.

Since $p_1(D_1(\hat\zeta)) \in \mathcal O_1^0$ is an even order orbifold point,  there exists an involution $\tau \in \Gamma_1^0$, rotating about the fixed point $D_1(\hat \zeta)$ through angle $\pi$. 
Since $\psi^0$ is induced by conjugation by $h$,  $\psi^0(\tau)$ is also a rotation through angle $\pi$ about some point $\xi \in \mathbb H$.

For each $i=0,1,2$, $D_1(\eta_i)$ is a geodesic through $D_1(\hat \zeta)$, and we let $x_i$ and $y_i=\tau \cdot x_i$ denote its endpoints in $\partial \mathbb H$. By Proposition~\ref{P:holonomy conjugation}, the endpoints of $D_2(\hat g(\eta_i))$ are $h(x_i)$ and $h(y_i)$. On the other hand, by Proposition~\ref{P:extending involutions}, we have
\[ h (y_i) = h(\tau \cdot x_i) = (h\tau h^{-1}) \cdot h(x_i) = \psi^0(\tau) \cdot h(x_i). \]
Therefore, $D_2(\hat g(\eta_i))$ is also invariant by $\psi^0(\tau)$, and thus passes through $\xi$, for $i=0,1,2$.  Since $\xi$ is the unique point of $D_2(\hat g(\eta_0)) \cap D_2(\hat g(\eta_i))$, for $i=1,2$, we must have $D_2(\hat \xi_1) = \xi = D_2(\hat \xi_2)$.  But $\hat \xi_1,\hat \xi_2 \in \hat g(\eta_0)$, and $D_2$ is injective on $\hat g(\eta_0)$, and therefore $\hat \xi_1=\hat \xi_2$, as required.
\end{proof}
Let $\hat \xi$ be the concurrent point of $\{ \hat g(\eta) \mid \eta \in \CG(\hphi_1,\hat \zeta)\}$.  We have thus shown that
\[ \CG_{\hphi_1}(\hat \zeta) \subset \CG_{\hphi_2}(\hat \xi).\]
We may reverse the roles of $\hphi_1$ and $\hphi_2$ since $D_2(\hat \xi)$ is the fixed point of an even order elliptic element of $\Gamma_2^0$, and is not a cone point.  Doing so proves the reverse inclusion, and shows that $\zeta$ is a concurrency point for $(\varphi_1,\varphi_2)$, as required.
\end{proof}

The next proposition is the final ingredient for the proof of the \currentsupport\!.

\begin{proposition} \label{P:induced orbifold homeo} Suppose $\CG_{\tphi_1} = \CG_{\tphi_2}$, $\Gamma_1^0$ is discrete, $C = (q_1^0)^{-1}(\mathcal E_1^0)$ is the concurrency set from Lemma~\ref{L:preferred concurrence}, and that $\varphi_2$ is $C$--uber-normalized.  Then there is a homeomorphism $f \colon S \to S$, isotopic to the identity relative to $C$, and a homeomorphism $F \colon \mathcal O_1^0 \to \mathcal O_2^0$ so that $q_2^0 \circ f = F \circ q_1^0$.
\end{proposition} 

Observe that the conclusion of this proposition implies that $\varphi_1$ and $\varphi_2$ arise as in Theorem~\ref{T:deforming}, and thus proving this proposition will essentially complete the proof of the \currentsupport\!.

\begin{proof} Let $\tilde \Delta_i$ be the cell structures on $\mathbb H$ from Lemma~\ref{L:Delaunay maps} applied to $G_i = \Gamma^0_i$ and the isomorphism $\psi^0$, let $\tilde \Delta_i'$ be the triangulation of $\mathbb H$ obtained by subdividing $\tilde \Delta_i$, for $i=1,2$, and $\tilde F \colon (\mathbb H,\tilde \Delta_1') \to (\mathbb H,\tilde \Delta_2')$ the canonical map given by the second half of Lemma~\ref{L:Delaunay maps}. We write $F \colon \mathcal O_1^0 \to \mathcal O_2^0$ to denote the descent of $\tilde F$, by the universal (orbifold) coverings $p_i \colon \mathbb H \to \mathcal O_i^0$, for $i=1,2$.  The map $F$ is the canonical map with respect to the descended triangulations, $\Delta_1'$ and $\Delta_2'$ on $\mathcal O_1^0$ and $\mathcal O_2^0$, respectively.

We will adjust $\varphi_2$ by a homeomorphism $f \colon S \to S$ isotopic to the identity relative to $C$ (without changing its name, as usual).  Thus to prove the proposition, it suffices to show that after this adjustment, the following diagram commutes.
\[
\begin{tikzcd}[row sep=15]
(\hat S,\hphi_1) \ar[r, "id_{\hS}"] \ar[dd, "\hat p"] 		& (\hat S,\hphi_2) \ar[drr, "D_2" ]  \ar[dd, "\hat p"] 	& \quad \quad \quad & \quad \quad \quad\\
										&										& \mathbb H	\ar[from=ull, crossing over, "D_1" near end] \ar[r,"\tilde F" below]					& \mathbb H \ar[dd, "p_2"] \\
(S,\varphi_1) \ar[r, "id_S"] 						& (S,\varphi_2)	\ar[drr, "q_2^0" near start]  \\
										&						&\mathcal O_1^0  \ar[r,"F" below] \ar[from=uu, crossing over, "p_1"] \ar[from=ull, crossing over, "q_1^0" below]	&  \mathcal O_2^0
\end{tikzcd}
\]

The "back" square obviously commutes ($\hat p = \hat p$), and since $F$ is the descent of the equivariant map $\tilde F$ this implies that the "front" square commutes ($p_2 \circ \tilde F = F \circ p_1$).  Similarly, the two "sides'' commute  ($p_i \circ D_i = q_1^0 \circ \hat p$, for $i = 1,2$) by equivariance and descent.  Since the "bottom'' square is what we want to prove, and is the descent of the "top'' square, it suffices to prove the top commutes; that is, we must show that $\tilde F \circ D_1 = D_2$.

Let $\hat \Delta$ be the cell structure on $\hS$ obtained by pulling back $\tilde \Delta_1$ with respect to $D_1 : \hS \to \mathbb H$.
The vertex set of $\hat \Delta$ is the concurrency set, that is, $\hat \Delta^{(0)} = \hat C$.  Since the $1$--skeleton of $\tilde \Delta_1$ is a union of biinfinite geodesics, the $1$--skeleton of $\hat \Delta$ is a union of basic $\hphi_1$--geodesics.  Specifically, there is a maximal set of basic geodesics $\Upsilon_1 \subset \CG(\hphi_1)$ so that
\[ \hat \Delta^{(1)} = \bigcup_{\eta \in \Upsilon_1} \eta.\]
Since $D_1$ is $\rho_1$--equivariant, and $\rho(\pi_1\dot S) = \Gamma_1 < \Gamma_1^0$, it follows that $\hat \Delta$ is $\pi_1\dot S$--invariant, and hence by maximality, $\Upsilon_1$ is $\pi_1\dot S$--invariant.

The $2$--cells of $\hat \Delta$ are convex hyperbolic polygons that map isometrically by $D_1$ to their image $2$--cells in $\mathbb H$.  Since the action of $\pi_1\dot S$ on $\hat S$ is free except at the points of $\hat C_0$, it follows that the stabilizer of each $2$--cell (a compact polygon) is trivial: otherwise it would have a fixed point in the interior.  It follows that $\hat p$ maps the interior of every $2$--cell homeomorphically to $S$, and thus $\hat \Delta$ descends to a cell structure $\Delta$ on $S$, all of whose $2$--cells are hyperbolic polygons with respect to $\varphi_1$ (possibly with some edges identified).  The zero-skeleton is $\Delta^{(0)} = \hat p(\hat C) = C$.

Subdivide $\Delta$ to a triangulation $\mathcal T$, then adjust $\varphi_2$ by a homeomorphism isotopic to the identity so that $\mathcal T$ is a $(C,\varphi_1,\varphi_2)$--triangulation.  Since $\varphi_2$ is $C$--uber-normalized, the adjustment is by a homeomorphism isotopic to the identity relative to $C$.   This pulls back to a $\pi_1\dot S$--invariant triangulation $\widehat{\mathcal T}$.  Any geodesic $\eta \in \Upsilon_1$ is a concatenation of edges of $\hat \Delta$, and hence edges in $\widehat{\mathcal T}$, and thus it follows that $\hat g(\eta) = \eta$.  Moreover, the identity restricts to the canonical map on each $C$--saddle connection of $\eta$, with respect to $\hphi_1$ on the domain and $\hphi_2$ on the range.  

Now we see how $\tilde F$ is related to $D_1$ and $D_2$.  For any vertex $\zeta \in \hat \Delta^{(0)} = \hC$, if $\zeta_1 = D_1(\zeta)$ and $\zeta_2 = D_2(\zeta)$, then because
\[ \psi^0(\tau_{\zeta_1}) = h \tau_{\zeta_1} h^{-1} = \tau_{\zeta_2},  \]
we have $\tilde F(\zeta_1) = \zeta_2$.  That is
\[ \tilde F \circ D_1(\zeta) = D_2(\zeta).\]
Furthermore, for any $\eta \in \Upsilon_1$, since $\eta$ is a basic geodesic for both $\hphi_1$ and $\hphi_2$, $D_1(\eta)$ and $D_2(\eta)$ are both geodesics in $\mathbb H$.  The geodesic $\eta$ passes through infinitely many points of $\hC$.  Let $\zeta,\zeta' \in \eta \cap \hC$ be two distinct such points, and observe that by the equation above we have
\[ \tilde F \circ D_1(\zeta) = D_2(\zeta) \quad \mbox{ and } \quad \tilde F \circ D_1(\zeta') = D_2(\zeta').\]
Since $D_1(\eta)$ and $D_2(\eta)$ are determined by two points on them, it follows that $\tilde F \circ D_1(\eta) = D_2(\eta)$.  Moreover, for every edge $e$ of $\hat \Delta^{(1)}$ contained in $\eta$, the identity in $\hat S$ is the canonical geodesic segment map when restricted to $e$, and since $\tilde F$ is also the canonical map from $D_1(e)$ to $D_2(e)$, naturality of canonical maps implies $\tilde F \circ D_1$ and $D_2$ agree on $e$.  So, $\tilde F \circ D_1$ and $D_2$ agree on $\hat \Delta^{(1)}$.  

For any $2$--cell $\sigma$ of $\hat \Delta$, the restrictions of $D_1$ and $D_2$ to $\sigma$ are isometries with respect to $\hphi_1$ and $\hphi_2$, respectively, and the composition
\[ D_2|_{\sigma}^{-1} \circ \tilde F \circ D_1|_{\sigma} \]
is the identity on $\partial \sigma$.  We may therefore adjust $\varphi_2$ by a homeomorphism $S \to S$ isotopic to the identity by an isotopy which is the identity outside $\hat p(\sigma)$, so that $D_2|_{\sigma}^{-1} \circ \tilde F \circ D_1|_{\sigma}$ is the identity.  Equivalently, $\tilde F \circ D_1|_{\sigma} = D_2|_\sigma$.  Performing this adjustment on each of the finitely many $2$--cell of $\Delta$ we have $\tilde F \circ D_1 = D_2$, as required.  This completes the proof.
\end{proof}

\begin{proof}[Proof of the \currentsupport\!\!] Suppose $\varphi_1,\varphi_2 \in \tHyp_c(S)$ and $\CG_{\tphi_1} = \CG_{\tphi_2}$.  If $\Gamma_1^0$ is indiscrete, then Proposition~\ref{P:indiscrete is rigid}  implies $\varphi_1 \sim \varphi_2$.  We may therefore assume $\Gamma_1^0$ is discrete, in which case Proposition~\ref{P:induced orbifold homeo} guarantees that $\varphi_1$ and $\varphi_2$ come from branched covers of orbifolds, as in Theorem~\ref{T:deforming} completing the proof.
\end{proof}

\begin{remark}
Given $\varphi_1,\varphi_2 \in \tHyp_c(S)$ with $\CG_{\tphi_1} = \CG_{\tphi_2}$ and $\varphi_1 \not \sim \varphi_2$, we know that the $\varphi_2$ is obtained by deforming $\mathcal O_1^0$ and pulling back via $q_1^0$.  It follows that the space of equivalence classes of all such metrics in $\Hyp_c(S)$,
\[ \{ \varphi \in \tHyp_c(S) \mid \CG_{\tphi} = \CG_{\tphi_1}\}/\!\sim \]
are parameterized by the Teichm\"uller space of the orbifold $\mathcal O_1^0$.  In particular, the dimension of this space can be arbitrarily large (depending on $\mathcal O_1^0$).
\end{remark}

\subsection{Too many cone points implies rigidity} \label{S:too many cone points}

As another consequence of the \currentsupport\!, we show that when there are too many cone points, the metric is forced to be rigid.

\medskip

\noindent
{\bf Corollary~\ref{C:too many cone points}.} {\em If $\varphi \in \Hyp_c(S)$ has at least $32(g-1)$ cone points (where $g$ is the genus of $S$) then $\varphi$ is rigid.}

\medskip

\begin{proof} If $\varphi$ is flexible then by the \currentsupport there is a locally isometric branched covering $(S,\varphi) \to \mathcal O$ over a hyperbolic orbifold $\mathcal O$ sending each of the cone points to an even order orbifold point.  We will thus show that if there is such a branched covering, then the number of cone points is less than $32(g-1)$.

The proof requires a repeated application of the Gauss-Bonnet Theorem, and is similar in spirit to the proof of the ``$84(g-1)$ Theorem"; see e.g.~\cite[Theorem~7.4]{farb:MCG}.  The Gauss-Bonnet Theorem implies in our case that the total area of a hyperbolic cone surface or orbifold of genus $g_0$ with $n$ cone/orbifold points having cone angles $\theta_1,\ldots,\theta_n$ is given by
\begin{equation} \label{E:GaussBonnet} \Area = 4\pi(g_0-1) + 2\pi n - \sum_{i=1}^n \theta_i.
\end{equation}

Fix any $\varphi \in \Hyp_c(S)$ and observe that by \eqref{E:GaussBonnet}, we have $\Area(\varphi) < 4\pi(g-1)$. Suppose that $p \colon (S,\varphi) \to \mathcal O$ is a branched cover as above and let $d$ be the degree.  Then we have
\begin{equation} \label{E:Area 1} d  = \tfrac{\Area(\varphi)}{\Area(\mathcal O)} < \tfrac{4\pi(g-1)}{\Area(\mathcal O)}.
\end{equation}
Next, let $k$ be the number of cone points of $\varphi$ and let $r$ be the number of even-order orbifold points of $\mathcal O$. We write $b_1,\ldots, b_r \in 2 \mathbb Z_+$ to denote the orders of the orbifold points, so that the $i^{th}$ orbifold point has cone angle $\tfrac{2\pi}{b_i}$.  Let $k_i$ be the number of cone points that map to the $i^{th}$ orbifold point so that $k = \sum_i k_i$.  Now observe that since the cone angle of any cone point of $\varphi$ is {\em strictly greater} than $2\pi$, it follows that the local degree of $p$ near any cone point is at least $3$.  This gives us
\[ d \geq 3 \max_i k_i \geq 3 \tfrac{k}r.\]
Combining this with \eqref{E:Area 1} we get 
\begin{equation} \label{E:Area 2}
k < \tfrac{4r\pi(g-1)}{3 \Area(\mathcal O)}.
\end{equation}

All that remains is to show that the right-hand side of \eqref{E:Area 2} is at most $32(g-1)$.  For this, we consider a case-by-case analysis depending on $r \geq 1$: for any fixed $r$, the right-hand side is maximized when $\Area(\mathcal O)$ is minimized.  For $r \geq 5$, we note that \eqref{E:GaussBonnet} implies that the area of $\mathcal O$ is at least $\pi(r-4)$, which occurs when $\mathcal O$ is a sphere with $r$ orbifold points, all of order $2$.  In this case we have
\[ k < \tfrac{4 \pi r(g-1)}{3\pi(r-4)} = \tfrac{4(g-1)}3\left(\tfrac{r}{r-4}\right) = \tfrac{4(g-1)}3\left( 1+ \tfrac{4}{r-4} \right) \leq \tfrac{20(g-1)}{3} < 32(g-1).\]
Again appealing to \eqref{E:GaussBonnet}, for the remaining values of $r$ one can check that the minimum area of an orbifold with at least $r$ even-order orbifold points is realized by a sphere with $n$ orbifold points of orders $(b_1,\ldots, b_n)$ described by the following table:
\[ \begin{array}{c|c|c|c}
r & \Area & (b_1,\ldots,b_n) & \displaystyle{\phantom{\int}  \tfrac{4 r \pi(g-1)}{3 \Area(\mathcal O)}  \phantom{\int}} \\
\hline 
1 & \tfrac{\pi}{21} & \displaystyle{\phantom{\int} (2,3,7) \phantom{\int}} & 28(g-1) \\
\hline 
2 & \tfrac{\pi}{12} &  \displaystyle{\phantom{\int} (2,3,8) \phantom{\int}} & 32(g-1) \\
\hline 
3 & \tfrac{\pi}{6} &  \displaystyle{\phantom{\int} (2,4,6) \phantom{\int}}  & 24(g-1) \\
\hline 
4 & \tfrac{\pi}2 &  \displaystyle{\phantom{\int} (2,2,2,4) \phantom{\int}}  & \tfrac{32}3(g-1)\\
\end{array} \]
The case $r=1$ is classically known to be the minimal area of any (orientable) hyperbolic orbifold (see e.g.~\cite[Theorem~7.10]{farb:MCG}).  The cases $r=2,3,4$ follow a similar argument to this classical result by enumerating the possibilities, appealing to the following two facts: (1) the genus must be zero (since positive genus and at least $2$ orbifold points gives area at least $2\pi$) and (2) we must have
\[ n-\sum_{i=1}^n \tfrac{1}{b_i} > 2\]
if $(b_1,\ldots,b_n)$ are the orders of orbifold points of a hyperbolic orbifold of genus zero.

Since each quantity in the last column is at most $32(g-1)$, this completes the proof.
\end{proof}


\section{Billiards}\label{S:billiards}

Consider a finite sided (simply connected, but not necessarily convex) polygon $P$ in the hyperbolic plane and trajectories of "billiard balls", modeled by continuous piecewise geodesic paths $\beta: \mathbb{R} \rightarrow P$. The billiard trajectories are geodesic segments in the interior of $P$ and reflect off the sides of the polygon with angle of incidence equal to angle of reflection. Note that we do not consider trajectories that encounter vertices of $P$. We call these dynamical systems polygonal billiard tables in the hyperbolic plane.

As in the introduction, we assume an $n$--gon $P$ comes equipped with a labelling of the sides by elements of the set $\mathcal{A}=\{1, 2, \ldots, n\}$, in counterclockwise cyclical order. Given a billiard trajectory $\beta:\mathbb{R}\to P$ its corresponding {\em bounce sequence} is the biinfinite sequence $\mathbf{b}(\beta)=(\ldots, b_{-1}, b_0, b_1, \ldots)\in\mathcal{A}^{\mathbb{Z}}$ corresponding to the ordered sequence of labels of the sides of $P$ that $\beta$ encounters, where $b_0$ is the first side encountered by $\beta([0, \infty))$. The {\em bounce spectrum} of $P$, denoted $\mathcal{B}(P)$, is the set of all bounce sequences realized by billiard trajectories on $P$.   We say that two labeled $n$--gons $P_1, P_2$ have the {\em same bounce spectrum} if $\mathcal{B}(P_1)=\mathcal{B}(P_2)$.


We recall the main theorem about billiards from the introduction.

\begin{billiardtheorem} Given hyperbolic polygons $P_1,P_2$, we have $\CB (P_1) = \CB (P_2)$ if and only if
\begin{enumerate}
\item $P_1$ is isometric to $P_2$ by a label preserving isometry, or
\item $P_1,P_2$ are reflectively tiled and there exists a label-preserving homeomorphism $H \colon P_1 \to P_2$ that maps tiles to tiles, preserving their interior angles.
\end{enumerate}
\end{billiardtheorem}


We note that the two cases are not mutually exclusive as the homeomorphism $H$ may be an isometry. 
Before launching into the proof of the \billiardrigidity\!\!, we make the notion of reflective tilings precise and introduce the tools we need in the sections below. 

\newcommand{\tile}{\bf t}

\subsection{Reflective tilings} \label{S:reflective tilings}
A {\em tiling} of $\mathbb H$ is a cell decomposition so that the closed $2$--cells, called the {\em tiles}, are compact convex hyperbolic polygons.  For each tile, we require the intersection with the $0$--skeleton to be the vertices (thus having interior angles less than $\pi$).  The $1$--cells are called the {\em edges} of the tiling; an edge contained in a tile will also be called a {\em side} of the tile.  
Given a tiling $T$ of $\mathbb H$ let $R(T)$ be the group generated by reflections in the (geodesic lines containing the) edges.

We say that a polygon in $\mathbb{H}$ is {\em good} if it is a finite sided compact polygon with each of its interior angles being an integral submultiple of $\pi$, that is, of the form $\frac{\pi}{k}$ for some $k\in\mathbb{Z}$.  The Poincar\'e polygon theorem implies that the group $R$ generated by reflections in the sides of a good polygon $\tile$ is discrete and the polygon is a fundamental domain  (see, e.g.~\cite[Theorem~7.1.3]{Ratcliffe}).  The translates of $\tile$ by the group $R$ determine a tiling $T$ with $R = R(T)$.  We call such a tiling of $\mathbb H$ a {\em reflective tiling}. 

\begin{lemma} \label{L:reflective H2} A tiling $T$ of $\mathbb H$ is reflective if and only if the fixed point set of any reflection in $R(T)$ is contained in the $1$--skeleton.
\end{lemma}
\begin{proof} If $T$ is a reflective tiling, then all tiles sharing a vertex differ by an element of $R(T)$ fixing that vertex.  Since the interior angles of tiles are integral submultiples of $\pi$, we see that the  biinfinite geodesic containing any edge is contained in the $1$--skeleton.  Therefore, the reflections in sides of any tile are contained in the $1$--skeleton.  If there were any other reflection in $R(T)$ whose fixed point set was not contained in the $1$--skeleton, then it would necessarily pass through the interior of a tile.  Since any tile is a fundamental domain, this is impossible, and hence the fixed point set of any reflection is in the $1$--skeleton, as required.

Now we assume that the fixed point set of any reflection in $R(T)$ is contained in the $1$--skeleton, and prove that $T$ is reflective.  

First observe that the $1$--skeleton is a union of biinfinite geodesics: since the reflection in any edge is in $R(T)$, it must have its entire fixed line in the $1$--skeleton by assumption, and thus it follows that for every edge of $T$, the unique geodesic containing it is in the $1$--skeleton.

Next, we claim that $R(T)$ preserves the tiling.  To prove this, it suffices to show that it preserves the collection of biiinfinite geodesics whose union is the $1$--skeleton (since $R(T)$ preserves the tiling if and only if it preserves the $1$--skeleton).  Thus let $g \in R(T)$ be any element and $\eta \subset \mathbb H$ a biinfinite geodesic in the $1$--skeleton.  The reflection $r_\eta$ in $\eta$ is a generator of $R(T)$, and hence $g r_\eta g^{-1} = r_{g(\eta)}$, the reflection in $g(\eta)$, is in $R(T)$.  By assumption $g(\eta)$ is contained in the $1$--skeleton, as required.

Since $R(T)$ preserves the titling, it is discrete and acts on the set of tiles.  In fact, this action on the tiles is transitive.  To see this, take any two tiles $\tile_1,\tile_2$ and consider a geodesic segment from a point in the interior of $\tile_1$ to a point in the interior of $\tile_2$ that misses the vertex set.  Then the composition of the reflections in the sides of tiles encountered  by the geodesic (in order) maps $\tile_1$ to $\tile_2$.  

From the proof of transitivity, we see that $R(T)$ is generated by reflections in the sides of any fixed tile $\tile$.
All that remains to prove is that the interior angle at each vertex of a tile is an integral submultiple of $\pi$.  This follows because the stabilizer of any vertex $\zeta$ is a dihedral group whose lines of reflection are precisely the biinfinite geodesic extensions of the edges adjacent to $\zeta$.  Therefore, $\tile$ is a good polygon, and the tiling is reflective.
\end{proof}

\begin{corollary} \label{C:gen reflective} Given a discrete group $R$ generated by reflections, including reflections in the sides of a compact polygon $P$, then the union of the fixed point sets of all reflections in $R$ is the $1$--skeleton of a reflective tiling.
\end{corollary}
\begin{proof} Since $R$ is discrete, the set of fixed point sets of the reflections is a locally finite collection of lines: otherwise there would be a sequence of such lines that meet a fixed compact set, and from this it is easy to construct a nonconstant sequence of elements of $R$ as products of pairs of reflections that converge to the identity.  For any point in the interior of $P$ which is disjoint from the union of the lines, the intersection of half planes containing the point and bounded by the lines in our set is contained in $P$, and is thus a compact, convex polygon $\tile$.  That is, the closure of one component of the complement of the lines in our set is a compact polygon $\tile$.  Given any other point $\zeta$ in the complement of the union of the lines, we can construct a geodesic from a point in $\tile$ to $\zeta$, and as in the proof above, from this construct a sequence of reflections whose composition takes $\tile$ to the closure of the complementary component containing $\zeta$.  It follows that the lines form the $1$--skeleton of a tiling, and by construction, the fixed point set of any reflection of $R$ is contained in it.  Thus, by Lemma~\ref{L:reflective H2}, the tiling is reflective.
\end{proof}

We now define a {\em reflective tiling of a polygon $P$} similarly to be a cell decomposition of $P$ such that: 
\begin{itemize}
\item the closed 2-cells (the {\em tiles}) are isometric copies of a good polygon, and any tile meets the $0$--skeleton precisely in its vertices,
\item if two tiles share a {\em side} (or {\em edge}, i.e.~$1$--cell) then they differ by a reflection in that side, and
\item if a tile $\tile$ shares a vertex with $P$, then the interior angle of $\tile$ at that vertex is an {\em even} integer submultiple of $\pi$, i.e. is of the form $\frac{\pi}{2k}$ for some $k\in\mathbb{N}$.
\end{itemize}

The third condition comes up naturally in the billiards theorem, though it is not clear what relevance it has outside our context.  We will need the following lemma in our proof of the Billiard Rigidity Theorem.
\begin{lemma}\label{L:reflective polygon}
If $P$ is a reflectively tiled polygon, then we can extend this tiling to a reflective tiling of $\mathbb{H}$. That is, there is a reflective tiling $T$ of $\mathbb{H}$  such that the reflective tiling of $P$ is a subcomplex of $T$.  
\end{lemma}

\begin{proof}
Because our polygons are assumed simply connected (hence connected), any two tiles differ by a sequence of reflections in tiles.  The entire group generated by reflections in any given tile defines a reflective tiling since the tile is a good polygon.
\end{proof}

%

\newcommand{\cD}{\mathcal D}

\subsection{Unfoldings} \label{S:unfoldings}
Next we define the concept of unfolding a polygonal billiard table. This connects polygonal billiard tables in the hyperbolic plane to the focus of the rest of this paper, hyperbolic cone metrics on closed orientable surfaces.  This a generalization (as in \cite{oldpaper}) of a classical construction (see e.g.~\cite{ZeKa}).

Given a polygon $P$ in $\mathbb{H}$, consider its {\em double}
\[ \cD P =  P \times \{ 0, 1 \} / \{(x,0) \sim (x,1) \,\, \forall x \in \partial P \} \]
which admits a hyperbolic cone metric for which the cone points are the vertices (or more precisely, the identified vertices of the two copies of $P$).  Observe that all cone angles are less than $2\pi$, however, so these metrics are {\em not} negatively curved.  Note that there is a canonical map $\cD P \rightarrow P$ sending $(x, t) \mapsto x$ for $t=0,1$, which is a local isometry away from the identified boundaries.  An \emph{unfolding} of $P$ is any locally isometric branched cover $(S,\varphi) \to \cD P$, branched over the cone points, so that the total angle around any point in the pre-image of a cone point of $P$ is greater than $2\pi$.  In particular $\varphi \in \tHyp_c(S)$ and all cone points map to cone points.  

Given a hyperbolic polygon $P$, observe that every (biinfinite) nonsingular geodesic on $\cD P$ projects to a (biinfinite) billiard trajectory.  Moreover, every billiard trajectory arises in this way; indeed, this sets up a $2$-to-$1$ correspondence between nonsingular geodesics on $\cD P$ and billiard trajectories on $P$.  
If $(S,\varphi) \to \cD P$ is an unfolding (with degree $k$, say) then the nonsingular $\varphi$--geodesics on $S$ project to nonsingular geodesics on $\cD P$, and all nonsingular geodesics arise in this way.  Combined with the above, this sets up a $2k$-to-$1$ correspondence between nonsingular $\varphi$--geodesics on $S$ and billiard trajectories on $P$.
 
For any two labeled $n$--gons $P_1$ and $P_2$ in $\mathbb{H}$ we may construct a common unfolding in the following sense.  First, fix a labeled {\em topological $n$--gon} $P$: this is a CW complex with $n$ vertices and $n$ edges attached to form a circle, with the edges labeled by $\mathcal A$ in order cyclically, and a single $2$--cell (a disk) whose attaching map is a homeomorphism to the $1$--skeleton.  The cell structure on $P$ lifts to one on its double, $\cD P$.
For each $i=1, 2$ fix a homeomorphism $P\to P_i$ respecting the labeling, which induces a homeomorphism of the doubles $\cD P \to \cD P_i$.  A {\em common unfolding} of $P_1$ and $P_2$ is a branched covering $S \to \cD P$ branched over the vertices together with a pair of metrics $\varphi_1,\varphi_2 \in \tHyp_c(S)$ so that the composition $(S,\varphi_i) \to \cD P \to \cD P_i$ is an unfolding for each $i=1,2$.  Observe that there is a natural cell structure on $S$ obtained by lifting the one on $\cD P$ and that with respect to $\varphi_i$, the closed $2$--cells are isometric copies of $P_i$: indeed, the map $(S,\varphi_i) \to \cD P \to \cD P_i \to P_i$ restricts to an isometry on each closed $2$--cell.  Moreover, the edges have labels in $\mathcal A$, induced from those on the edges of $\cD P$.

Now, let $S$ be a common unfolding of two $n$-gons $P_1$ and $P_2$ with corresponding metrics $\varphi_1$ and $\varphi_2$. Let $p: \tilde S\to S$ denote the universal cover of $S$, and as usual let $\tphi_1, \tphi_2$ be the pull back metrics by $p$.  Through the correspondence between billiard trajectories and nonsingular $\varphi_i$--geodesics on $S$ we get the following:

\begin{lemma}\label{L:billiards endpoints}
Suppose $P_1$ and $P_2$ are two $n$--gons (with labeled sides) and let $S$ be a common unfolding with metrics $\varphi_1$ and $\varphi_2$, respectively.  If $\CB(P_1)=\CB(P_2)$ then $\CG_{\tphi_1} = \CG_{\tphi_2}$ and $\varphi_2$ is $C_0$--uber-normalized with respect to $\varphi_1$ (where $C_0 = \cone(\varphi_1)$, as usual).
\end{lemma}
\begin{proof} First observe that the cell structure on $S$ with labeled edges lifts to a cell structure on $\tS$ with labeled edges.  With respect to $\tphi_i$, each closed $2$--cell is isometric to $P_i$, and as with $S$, the map $(\tS,\tphi_i) \to (S,\varphi_i) \to P_i$ restricts to an isometry on each $2$--cell.

Now suppose $\CB(P_1) = \CB(P_2)$. Let $\{x,y\}\in\CG_{\tphi_1}$ be the endpoints of a nonsingular $\tphi_1$--geodesic $\eta$. Reading off the labels crossed gives a sequence ${\bf b} \in \mathcal A^{\mathbb Z}$.  Since this nonsingular $\tphi_1$--geodesic projects to a nonsingular $\varphi_1$--geodesic in $S$, and then down to billiard trajectory in $P_1$, it follows that ${\bf b} \in \CB(P_1)$.
 Since $\CB(P_1)=\CB(P_2)$, $\mathbf{b}\in \CB(P_2)$ too, and hence there is a nonsingular $\tphi_2$--geodesic $\eta'$ crossing the exact same set of edges of the $1$--skeleton as $\eta$ which projects to the billiard trajectory in $P_2$ with bounce sequence ${\bf b}$.   Since $\eta$ and $\eta'$ cross the exact same set of edges in $\tS$, they remain a bounded distance apart, and hence the endpoints of $\eta'$ are $\{x,y\} \in \CG_{\tphi_2}$.
Since endpoints of nonsingular $\tphi_i$--geodesics are dense in $\CG_{\tphi_i}$, it follows that $\CG_{\tphi_1} \subset \CG_{\tphi_2}$.  The other inclusion follows by a symmetric argument and we have $\CG_{\tphi_1} = \CG_{\tphi_2}$.

To see that $\varphi_2$ is $C_0$--uber-normalized, first observe that $\tC_0 = \cone(\tphi_1) = \cone(\tphi_2)$ (since the metrics come from a common unfolding of $P_1$ and $P_2$).  Moreover, the $\pi_1S$--equivariant bijection $\cone(\tphi_1) \to \cone(\tphi_2)$ from Proposition~\ref{P:chain consequence} is the {\em identity}.  This follows directly from the fact that $g \colon \CG(\tphi_1) \to \CG(\tphi_2)$ preserves partitions (Lemma~\ref{l:partition2}) and the proof above showing that $\CG_{\tphi_1} = \CG_{\tphi_2}$.
Next we subdivide the cell structure on $S$ to a $(C_0,\varphi_1)$--triangulation $\mathcal T$.  We may adjust $\varphi_2$ by a homeomorphism isotopic to the identity so that each of the triangles are also $(C_0,\varphi_2)$--triangles.  Since the $1$--cells are already saddle connections with respect to both $\varphi_1$ and $\varphi_2$, the isotopy can be assumed to preserve the $1$--skeleton of the cell structure, and hence is relative to $C_0$.  After further isotopy relative to $C_0$ if necessary we can assume that the identity is the canonical map for $\mathcal T$, and thus $\mathcal T$ is a $(C_0,\varphi_1,\varphi_2)$--triangulation.  Since all adjustments to $\varphi_2$ were by homeomorphism isotopic to the identity relative to $C_0$ and the result is a metric that is $C_0$--uber-normalized (by Corollary~\ref{C:compatible}), the original $\varphi_2$ is $C_0$--uber-normalized.
\end{proof}


\subsection{Reflection groups}\label{S:reflection groups}

Let $P_1$, $P_2$ be two $n$--gons and $S \to \cD P \to \cD P_i$, $i=1,2$, a common unfolding with corresponding metrics $\varphi_1, \varphi_2$.  If $\CB(P_1) = \CB(P_2)$, then Lemma \ref{L:billiards endpoints} implies $\CG_{\tphi_1} = \CG_{\tphi_2}$ and $\varphi_2$ is $C_0$--uber-normalized.   
Let $\hat p\colon\hS\to S$ be the extension to the completion of the universal cover of $\dot S = S\setminus C_0$, as defined in \S\ref{S:puncturing surfaces}, and $\hphi_1, \hphi_2$ the pull backs of $\varphi_1, \varphi_2$ by $\hat p$. By Lemma~\ref{L:hat combinatorics} we also have $\CG_{\hphi_1} = \CG_{\hphi_2}$.  We assume that $S$ and $\tS$ are given the cell structures with labeled edges as described in the previous section, as is $\hS$ in a similar fashion.

Viewing $\dot P = P\setminus \{\text{vertices of }P\}$ as an orbifold with reflectors as sides, let $R=\pi_1^{orb}(\dot P)$. Then $\pi_1\cD \dot P<R$ is an index 2 subgroup, and the maps $\dot S \to \cD \dot P \to \dot P$ induce inclusions
$$\pi_1\dot S< \pi_1 \cD \dot P < R$$
and all three groups act on $\hS$ by isometries with respect to both $\hphi_1$ and $\hphi_2$.
For $i=1, 2$ let $D_i: \hS\to\mathbb{H}$ be the developing maps and $\rho_i: R\to \Isom(\mathbb{H})$ the corresponding holonomy homomorphisms. 
We fix some $2$--cell $B$ in $\hS$ and assume that $D_i$ maps this $2$--cell to $P_i$, for $i=1,2$.  Then, $\rho_i$ sends the reflection in the $j^{th}$ side of $B$ to reflection in the $j^{th}$ side of $P_i$.

For each $i=1,2$, let $R_{P_i} = \rho_i(R)<\Isom(\mathbb{H})$ and note that $R_{P_i}$ is the group generated by reflections in the sides of $P_i$. For any $\zeta \in \mathbb H$ let $\tau_\zeta$ denote a rotation about $\zeta$ of order 2 and let 
\[ R_{P_i}^0 = \langle R_{P_i}, \{\tau_v\}_{v \in V(P_i)} \rangle \]
where $V(P_i)$ denotes the vertex set of $P_i$. 
Note that $R_{P_i}^0$ is also generated by reflections because each $\tau_v$ is in $\Stab_{R_{P_i}^0}(v)$, which is a dihedral group and is thus generated by its reflections. 

\begin{lemma}\label{L:billiards topological conjugate}
Suppose $\CB(P_1)=\CB(P_2)$ and that $S \to \cD P \to \cD P_i$, for $i=1,2$, is a common unfolding with associated metrics $\varphi_1,\varphi_2$. Further suppose that we have chosen developing maps $D_1,D_2 \colon \hS \to \mathbb H$ as above with associated holonomy homomorphisms $\rho_1,\rho_2$, respectively.  Then there exists an orientation preserving homeomorphism $h: \partial\mathbb{H}\to\partial\mathbb{H}$ topologically conjugating $\rho_1$ to $\rho_2$.  In particular,  $R_{P_1}$ to $R_{P_2}$ are topologically conjugate by $h$, with $h$ conjugating the reflection in the $j^{th}$ sides of $P_1$ to the reflection in the $j^{th}$ side of $P_2$. 
\end{lemma}

\begin{proof}
Since $\CB(P_1) = \CB(P_2)$ we have that $\CG_{\tphi_1} = \CG_{\tphi_2}$ with $\varphi_2$ being $C_0$--uber-normalized by Lemma \ref{L:billiards endpoints}, and $\CG_{\hphi_1} = \CG_{\hphi_2}$ by Lemma~\ref{L:hat combinatorics}. Since $R$ acts on $\hS$ by isometries with respect to $\hphi_1$ and $\hphi_2$, the existence of the conjugating homeomorphism $h$ is guaranteed by Proposition \ref{P:holonomy conjugation}. From the same proposition we have that 
$$h(\rho_1(\gamma) \cdot x) = \rho_2(\gamma) \cdot h(x)$$
for all $\gamma\in R$ and $x\in \partial\mathbb{H}$. Hence if $\gamma\in R$ corresponds to reflection in the $j^{th}$ side of $P$ 
then $\rho_i(\gamma)$ is the reflection across the $j^{th}$ side of $P_i$. Therefore, the equation above expresses precisely that $h$ conjugates reflection in the $j^{th}$ sides of $P_1$ and $P_2$.
\end{proof}

The following is useful in the proof of the \billiardrigidity\!.

\begin{lemma}\label{L:billiards indiscrete}
If $R^0_{P_1}$ is either indiscrete or contains reflections in all three sides of a triangle, then $P_1$ is billiard rigid.  That is, if $P_2$ is another polygon with $\CB(P_1) = \CB(P_2)$, then there exists a label preserving isometry from $P_1$ to $P_2$.
\end{lemma}

\begin{proof} 
Suppose $\CB(P_1) = \CB(P_2)$. By Lemma \ref{L:billiards topological conjugate} $R_{P_1}$ and $R_{P_2}$ are topologically conjugate by a homeomorphism $h \colon \partial \mathbb H \to \partial \mathbb H$.  It follows from Proposition~\ref{P:extending involutions} (after choosing a common unfolding and developing maps $D_1,D_2$ as above) that $R^0_{P_1}$ and $R^0_{P_2}$ are also topologically conjugate by $h$ since the $j^{th}$ vertex of $P_1$ and $P_2$ are given by $D_1(\zeta^j)$ and $D_2(\zeta^j)$ for some $\zeta^j \in \hC_0$.  If $R^0_{P_1}$ is indiscrete then it follows by Proposition \ref{p:subgroups} that $h$ is in $\PSL_2(\mathbb{R})$.  

Next suppose there is a hyperbolic triangle such that $R^0_{P_1}$ contains reflections in all three of its sides; call these reflections $r, s, t$ and let $R^{\Delta}$ be the subgroup of $R^0_{P_1}$ they generate. If $R^{\Delta}$ is indiscrete then so is $R^0_{P_1}$ and $h$ is in $\PSL_2(\mathbb{R})$ by the above, so assume $R^{\Delta}$ is discrete. Let $\Gamma^{\Delta} = \langle sr, tr\rangle$ be the subgroup of $R^{\Delta}$ generated by the two elliptic elements $sr$ and $tr$ (which must be of finite order since $\Gamma^{\Delta}$ is discrete). Note that $\Gamma^{\Delta}$ is orientation preserving since its generators are. Next, observe that $R^{\Delta}=\langle \Gamma^{\Delta}, r\rangle$  and that $r(sr)r = (sr)^{-1}\in\Gamma^{\Delta}$ and $r(tr)r=(tr)^{-1}\in\Gamma^{\Delta}$ and hence $r\Gamma^{\Delta}=\Gamma^{\Delta} r$. It follows that $\Gamma^{\Delta}$ is an index 2 subgroup of $R^{\Delta}$ and in particular it is cocompact since $R^{\Delta}$ is. We have shown that $\Gamma^{\Delta}$ is a discrete cocompact subgroup of $\PSL_2(\mathbb{R})$ generated by two finite order elements; the only such groups are Fuchsian triangle groups (see \cite{Purzitsky}). It follows by Lemma \ref{l:trianglegroup} that $h$ must lie in $\PSL_2(\mathbb{R})$.

Since $h$ is in $\PSL_2(\mathbb{R})$ we can extend it to an isometry of $\mathbb{H}$ which we also denote by $h$.  Since $h$ conjugates the reflection in the $j^{th}$ side of $P_1$ to the reflection in the $j^{th}$ side of $P_2$, it sends the geodesic extension of the $j^{th}$ side of $P_1$ to the geodesic extension of the $j^{th}$ side of $P_2$. This implies that $h$ sends sides of $P_1$ to sides of $P_2$, so defines a label preserving isometry $P_1 \to P_2$.
\end{proof}

Before proving the full billiard rigidity theorem in the next section, we already now see as a direct consequence of the above lemma that billiard rigidity is generic:

\begin{corollary} \label{C:irrational angle rigid} If any interior angle of a polygon $P$ is an irrational multiple of $\pi$, then $P$ is billiard rigid.
\end{corollary}
\begin{proof} The assumption on $P$ implies $R_P$ is indiscrete, and hence Lemma~\ref{L:billiards indiscrete} implies $P$ is billiard rigid.
\end{proof}

\subsection{The \billiardrigidity\!\! and examples} \label{S:billiard rigidity proof}

We are now ready to prove the theorem.

\begin{proof}[Proof of the \billiardrigidity\!] 
Suppose $\CB (P_1) = \CB (P_2)$. If $R^0_{P_1}$ is indiscrete we are done by Lemma \ref{L:billiards indiscrete} so we will assume that it is discrete.  Then, for each $i=1,2$ the union of the lines of reflection for $R_{P_i}^0$ defines a reflective tiling $T_i$ by Corollary~\ref{C:gen reflective}, and $P_i$ is a union of tiles.

The homeomorphism $h: \partial \mathbb{H} \rightarrow \partial \mathbb{H}$ from Lemma \ref{L:billiards topological conjugate} topologically conjugates $R_{P_1}^0$ to $R_{P_2}^0$, and so defines a bijection, $h_*$, between the set of oriented lines of reflection of $R_{P_1}^0$ to those of $R_{P_2}^0$ via the action on the endpoints of such lines.
In particular, the lemma guarantees that $h_*$ sends the line containing the $j^{th}$ side of $P_1$ to the line containing the $j^{th}$ side of $P_2$.
Note that $h_*$ also preserves concurrence of any pair of lines of reflection because the point of intersection is the unique fixed point of a finite subgroup and $h$ conjugates this subgroup of $R_1^0$ to the corresponding subgroup of $R_2^0$. Hence, $h$ determines an equivariant bijection $H'$ from the vertices of $T_1$ to the vertices of $T_2$.

We would like to extend $H'$ to a map $H \colon (\mathbb H, T_1) \to (\mathbb H,T_2)$ that sends $P_1$ to $P_2$ by first extending over the $1$--skeleton, $T_1^{(1)}$, and then over the tiles.  For each edge $[v_1,v_2]$ of the $1$--skeleton of $T_1$, we define $H|_{[v_1,v_2]}$ to send $[v_1,v_2]$ to the geodesic between $H'(v_1)$ and $H'(v_2)$ by the canonical map.  We have thus defined an equivariant map $H|_{T_1^{(1)}}$, and for any line $\ell$ of reflection for $R_{P_1}^0$, $H|_{T_1^{(1)}}$ maps all vertices in $\ell$ to vertices in $h_*(\ell)$, and hence $H|_{T_1^{(1)}}(\ell) = h_*(\ell)$.  To show that $H|_{T_1^{(1)}}$ is a homeomorphism, it suffices to show that for any line of reflection, $\ell$, the ordering of the vertices along $\ell$ agrees with the $H'$--image ordering of vertices along $h_*(\ell)$.


In order to get a contradiction, suppose $\ell$ is an oriented line of reflection of $R_{P_1}^0$ and let $v_1$ and $v_2$ be vertices along $\ell$ such that $H|_{T_1^{(1)}}(v_1)$ and $H|_{T_1^{(1)}}(v_2)$ occur in the opposite order along $h_*(\ell)$.  Take two lines of reflection $\ell_1$ and $\ell_2$ of $R_{P_1}^0$ that are transverse to $\ell$ at $v_1$ and $v_2$ respectively. If $\ell_1$ and $\ell_2$ do not intersect, then $H|_{T_1^{(1)}}$ must not reorder $v_1$ and $v_2$, since $h$ preserves the cyclic order of the end points of $\ell$, $\ell_1$, $\ell_2$. If they do intersect, then there is a subgroup of $R_{P_1}^0$ generated by reflections in sides of a triangle, and $P_1$ and $P_2$ are isometric by Lemma~\ref{L:billiards indiscrete}, which is impossible if the order of $v_1$ and $v_2$ are reversed along $\ell$ (for, in this case, $H|_{T_1}^{(1)}$ is the restriction of an isometry conjugating $R_{P_1}^0$ to $R_{P_2}^0$).
Therefore, the order of $v_1$ and $v_2$ along $\ell$ must be preserved by $H|_{T_1^{(1)}}$, and hence this map is an equivariant homeomorphism from $T_1^{(1)}$ to $T_2^{(1)}$.
 Note that the cyclic order of edges incident to a given vertex given by the embedding of the $1$--skeleton of $T_1$ into $\mathbb{H}$ is preserved by $H|_{T_1^{(1)}}$, since $h$ preserves the cyclic order of the endpoints of the geodesics containing these edges. Therefore, we may extend $H|_{T_1^{(1)}}$ further to the faces of the tiling $T_1$, giving a equivariant homeomorphism $H \colon (\mathbb H, T_1) \to (\mathbb H,T_2)$.  Since $h_*$ sends the line of reflection in the $j^{th}$ side of $P_1$ to the line of reflection in the $j^{th}$ side of $P_2$, it follows that $H$ sends $P_1$ to $P_2$ and preserves labels. 

For every $v$ a vertex of $P_1$, we have $\tau_v \in \Stab_{R_{P_1}^0}(v)$, so this stabilizer is the dihedral group of an even sided polygon, and hence the angle of a tile at $v$ must be of the form $\pi/2k$ for some $k\in\mathbb{N}$.  Since the stabilizer of $v$ in $R_{P_1}^0$ and $H(v)$ in $R_{P_2}^0$ are isomorphic dihedral groups, $H$ preserves the angles of the tiles. This proves one direction of the theorem.

For the converse, assuming (1), we clearly get $\CB(P_1) = \CB(P_2)$, so assume (2).  The reflective tiling $T_i$ of $P_i$ for $i=1,2$, determine reflective tilings of $\mathbb H$ of the same names.  Changing $H$ by an isotopy if necessary, we may assume that it is the restriction of a homeomorphism $\mathbb H \to \mathbb H$ that conjugates $R(T_1)$ to $R(T_2)$.  The homeomorphism $H$ induces a correspondence between the edges of $P_1$ and $P_2$ that preserves labelings, and we construct a common unfolding $S \to \cD P$ with associated metric $\varphi_1$ and $\varphi_2$ with respect to label preserving homeomorphisms $P \to P_i$.  The restricted map $H|_{P_1}$ lifts to a map $H_{\cD P} \colon \cD P_1 \rightarrow \cD P_2$, and we assume that our homeomorphism $\cD P \to \cD P_2$ factors as the composition of $H_{\cD P}$ with the homeomorphism $\cD P \to \cD P_1$.

The doubles $\cD P_1$ and $\cD P_2$ must branch-cover orbifolds $\mathcal{O}_1$ and $\mathcal{O}_2$, respectively, given by taking the quotient of $\mathbb{H}$ by the orientation preserving subgroup $\Gamma_i = R(T_i) \cap \PSL_2(\mathbb R)$. This can be seen as follows. Let $P_i$ and $\bar{P_i}$ be the two copies of the polygon used to construct $\cD P_i$. Define the map $f_i: \cD P_i \rightarrow \mathbb{H}/\Gamma_i$ on each of these copies. Take $f_i|_{P_i} \colon P_i \rightarrow \mathbb{H}/\Gamma_i$ to be the inclusion of $P_i$ into the hyperbolic plane, followed by the quotient by the action of $\Gamma_i$. Take $f_i|_{\bar{P_i}} \colon \bar{P_i} \rightarrow \mathbb{H}/\Gamma_i$ to be the embedding of the polygon into the hyperbolic plane (same as before) \emph{followed by reflecting it across one of its edges}, followed by the quotient by the action of $\Gamma_i$. Note that it does not matter what edge we reflect across in this construction, since any two choices differ by the action of an element of $\Gamma_i$. This is a branched covering because it is possible to find charts in $\mathbb{H}$ covering $\cD P_i$ on which $f_i$ is a local branched covering.

Since the homeomorphism $H$ conjugates $R(T_1)$ to $R(T_2)$, it induces an orbifold homeomorphism $H_{\mathcal{O}}$ fitting into the following commutative diagram.\\
\[
\begin{tikzcd}
(S, \varphi_1) \arrow[d] \arrow[r, "id_S"] & (S, \varphi_2) \arrow[d]\\
\cD P_1 \arrow[d, "f_1"] \arrow[r, "H_{\cD P}"] &\cD P_2 \arrow[d, "f_2"]\\
\mathcal{O}_1 \arrow[r, "H_{\mathcal{O}}"] & \mathcal{O}_2
\end{tikzcd}
\]

Note that the cone points of $S$ are sent to the cone points/vertices of $\cD P_i$, which are then sent to even order orbifold points in $\mathcal{O}_i$.  This is because at each polygon vertex, the incident tiles of $T_i$ have angle $\frac{\pi}{2k}$, so $\Gamma_i$ must contain a rotation through angle $\frac{\pi}{k}$, which has even order.

Theorem \ref{T:deforming} implies $\CG_{\tphi_1} = \CG_{\tphi_2}$, while by Lemma~\ref{L:examples C--uber-normalized}, the preimage $C$ of the even order orbifold points of $\mathcal O_1$ is a concurrency set for $(\varphi_1,\varphi_2)$ with $\varphi_2$ being $C$--uber-normalized.  The cell structure on $S$ has $2$--cells that map isometrically to $P_i$ with respect to $\varphi_i$ via the common unfolding, and sides labeled according to the labels of $P$ (and hence $P_i$).   As usual, we lift this to a cell structure on $\tS$ with induced labels on the edges.
Since every edge is simultaneously a $\tphi_1$--saddle connection and $\tphi_2$--saddle connection, Lemma~\ref{l:transverse} implies that for every nonsingular $\tphi_1$--geodesic $\eta$, $g(\eta)$ crosses the same set of saddle connections.  

Now given any ${\bf b} \in \CB(P_1)$, let $\eta$ be a nonsingular $\tphi_1$--geodesic that projects to the billiard trajectory in $P_1$ with bounce sequence ${\bf b}$.  Recording the labels on the edges crossed by $\eta$ we get the sequence ${\bf b}$.  Since $g(\eta)$ is a $\tphi_2$--geodesic that crosses the same set of edges as $\eta$, it projects to a billiard trajectory in $P_2$ with the same bounce sequence ${\bf b}$, and thus $\CB(P_1) \subset \CB(P_2)$. A symmetric argument proves the other containment, and hence $\CB(P_1) = \CB(P_2)$.
\end{proof}

The next proposition is an application to the \billiardrigidity\!\!, giving concrete examples of both rigid and flexible polygons.

\begin{proposition} \label{P:lots of examples} 
Let $n\geq4$.  Suppose $P,P'$ are $n$--gons (with labeled sides) and write $\alpha_i$ and $\alpha_i'$ to denote the interior angle between the sides of $P$ and $P'$, respectively, labeled $i$ and $i+1$ (indices taken modulo $n$).   Then we have the following:
\begin{enumerate}
\item[(1)] If for all $i$, $\alpha_i$ is an even submultiple of $\pi$, then $\CB(P) = \CB(P')$ if and only if $\alpha_i = \alpha_i'$, for all $i$.
\item[(2)] If for all $i$, $\alpha_i$ is not an even submultiple of $\pi$, then $\CB(P) = \CB(P')$ if and only if $P$ and $P'$ are isometric by a label preserving isometry.
\end{enumerate}
\end{proposition}

\begin{proof}
\noindent First assume that we are in case (1), and hence $\alpha_i$ is an even submultiple of $\pi$ for all $i$.  If $\CB(P) = \CB(P')$ it follows immediately from the \billiardrigidity\!\! that $\alpha_i = \alpha_i'$ for all $i$.
For the converse, assume $\alpha_i = \alpha_i'$ for all $i$.  In this case $P$ (respectively, $P'$) is reflectively tiled by the single tile $P$ (respectively, $P'$).  Then any label preserving homeomorphism $H \colon P \to P'$ trivially maps tiles to tiles and preserves interior angles, since $\alpha_i = \alpha_i'$, and so by the \billiardrigidity we have $\CB(P) = \CB(P')$.

Now suppose none of the angles $\alpha_j$ are even submultiples of $\pi$. If $P$ and $P'$ are isometric by a label preserving isometry, then we must have $\CB(P)=\CB(P')$. For the converse, assume $\CB(P)=\CB(P')$. First note that if there exists $i$ with $\alpha_i\notin\mathbb{Q}\pi$ then $P$ and $P'$ are isometric by a label preserving isometry by Corollary \ref{C:irrational angle rigid}. Hence we will assume that each $\alpha_i$ is of the form $\frac{p_i\pi}{q_i}$ where $p_i, q_i\in\mathbb{N}$ with either $p_i>1$ or $p_i=1$ and $q_i$ odd. We will show that $R_{P}^0$ contains reflections in all three sides of a triangle, and hence by Lemma \ref{L:billiards indiscrete} we must have that $P$ and $P'$ are again isometric by a label preserving isometry. To obtain such a triangle, we first show the following claim.  Let $v_i$ be the vertex of $P$ corresponding to angle $\alpha_i$ and let $l_i$ be the geodesic ray emanating from $v_i$, traveling into $P$, and bisecting the angle $\alpha_i$. 


\begin{claim} 
There exists an $i$ such that $l_i$ and $l_{i+1}$ intersect (where indices are taken modulo $n$). 
\end{claim}

\begin{proof}[Proof of Claim]
Pick any vertex $v_j$ and cut $P$ along $l_j$ until the polygon has been divided into two. Note that at least one of the resulting polygons is an $n_1$--gon with $3\leq n_1<n$, call it $P_1$.  Let $v_k$ be the vertex adjacent to $v_j$ in $P_1$ and not incident to $l_j$. Note that $v_k$ is either $v_{j+1}$ or $v_{j-1}$. If $l_k$ intersects $l_j$ we are done. If not, cut $P_1$ along $l_{k}$, resulting in two new polygons, and let $P_2$ be the one not having a subsegment of $l_j$ as a side. Then $P_2$ is an $n_2$--gon with $3\leq n_2<n_1$. We repeat the process by considering the vertex $v_l$ adjacent to $v_k$ in $P_2$ and not incident to $l_{k}$. Note that either $(k,l)=(j+1, j+2)$ or $(k, l)=(j-1, j-2)$. This process must terminate since each step reduces the number of sides of the resulting polygon and if we arrive at a triangle, having a subsegment of a bisector $l$ as one of its sides, the vertex $v_i$ not incident to $l$ must have $l_i$ intersecting $l$.
\end{proof}

Now, using the claim we choose $i$ such that $l_i$ and $l_{i+1}$ intersect. Let $s$ denote the side of $P$ labeled ${i+1}$, that is, the side between $v_i$ and $v_{i+1}$. For $j=1,2$ write $\alpha_j=\frac{p_j\pi}{q_j}$ as above. If $p_j=1$ set $\hat{l}_j$ to be the bisecting ray $l_j$ defined above and if $p_j>1$ let $\hat{l}_j$ be the ray emanating from $v_j$ and making angle $\frac{\pi}{q_j}$ with $s$ inside $P$. Note that in all cases, $\hat{l}_{i}$ and $\hat{l}_{i+1}$ must intersect and so $s$ together with subsegments of $\hat{l}_{i}$ and $\hat{l}_{i+1}$ bound a triangle $\Delta$. We claim that $R_P^0$ contains reflections in all three of its sides. 

To see this, first note that $R_P^0$ contains the reflection $r$ in the side $s$ by definition and it contains a rotation about $v_j$ of angle $2\alpha_j$ for $j=1,2$. If $\alpha_j=\frac{\pi}{q_j}$ for some odd $q_j$ then, since $R_{P}^0$ contains rotations about $v_j$ both of angle $\frac{2\pi}{q_j}$ and of angle $\pi$, it also contains the rotation $\sigma_j$ about $v_j$ of angle $\frac{\pi}{q_j}$. Reflection in $\hat{l}_j=l_j$ is then given by  $\sigma_{i}\circ r$ for $j=i$ or by $\sigma^{-1}_{i+1}\circ r$ for $j=i+1$. If $\alpha_j=\frac{p_j\pi}{q_j}$, for $p_j>1$, then $R_P^0$ contains a rotation about $v_j$ of angle $\frac{2p_j\pi}{q_j}$ and hence also the rotation $\sigma_j$ about $v_j$ of angle $\frac{2\pi}{q_i}$ and reflection in $\hat{l}_j$ is then given by  $\sigma_{i}\circ r$ or $\sigma^{-1}_{i+1}\circ r$.

Hence $R_p^0$ contains reflections in all three sides of the triangle $\Delta$ as claimed and by Lemma \ref{L:billiards indiscrete} we are done.

\end{proof}

\newcommand{\polygons}{\frak P}

In case (1) of the proposition we can also compute the dimension of the space of polygons having the same bounce spectrum.  To state this precisely, let
\[ \polygons(P) = \{P' \subset \mathbb H \mid \CB(P') = \CB(P)\}/\!\sim , \]
where $P' \sim P''$ if there is a label preserving isometry from $P'$ to $P''$.
\begin{corollary} \label{C:deforming polygons} If $P$ is an $n$--gon and all interior angles are even submultiples of $\pi$, then $\polygons(P) \cong \mathbb R^{n-3}$.
\end{corollary}

\begin{proof} Let $\alpha_j$ be the interior angle between side $j$ and $j+1$ of $P$ for all $j$ (indices taking modulo $n$).  According to Proposition~\ref{P:lots of examples}, $\polygons(P)$ is precisely the set of polygons $P'$ (up to label preserving isometry) for which the interior angle between the side $j$ and $j'$ is $\alpha_j$ for all $j$.  This space is well-known to be homeomorphic to $\mathbb R^{n-3}_+$, but we briefly explain this for completeness.  One approach to this is to construct a polygonal analogue of a pants decomposition and Fenchel-Nielsen coordiantes (see e.g.~\cite[Section~4.6]{Thurston:book}) which we sketch now.

Take a maximal collection of pairwise disjoint arcs in $P$ connecting distinct, non-adjacent sides, so that no two arcs connect the same pair of sides.  The number of such arcs is precisely $n-3$.  Replacing each arc by a minimal length geodesic connecting the same sides, these geodesic segments produce a decomposition of $P$ into right-angled hexagons, pentagons with at most one non-right angle (which occurs at a vertex of $P$) and quadrilaterals with at most two non-right angles (each of which occur at a vertex of $P$).  We note that if angles $\alpha_j$ and $\alpha_{j+1}$ are both $\tfrac{\pi}2$ and if one of the arcs has endpoints on side $j$ and $j+2$ (indices modulo $n$), then the geodesic replacement arc {\bf is} the side labeled $j+1$.

The lengths of the geodesic arcs $(\ell_1,\ldots,\ell_{n-3}) \in \mathbb R_+^{n-3}$ serve as parameters since the space of deformations of each of the complementary polygons, preserving all angles, are parameterized by these lengths (see e.g.~\cite[Section~4.6]{Thurston:book} for the case of a right-angled hexagon--the other cases are simpler; c.f.~\cite[\S7.17-7.19]{Beardon}).  Figure~\ref{F:F-N polygon} illustrates an example. \end{proof}


\begin{figure}[h]
\begin{center}
\captionsetup{width=.85\linewidth}
\begin{tikzpicture}[scale=.25]
\node at (.25,0) {\includegraphics[width=.4\linewidth]{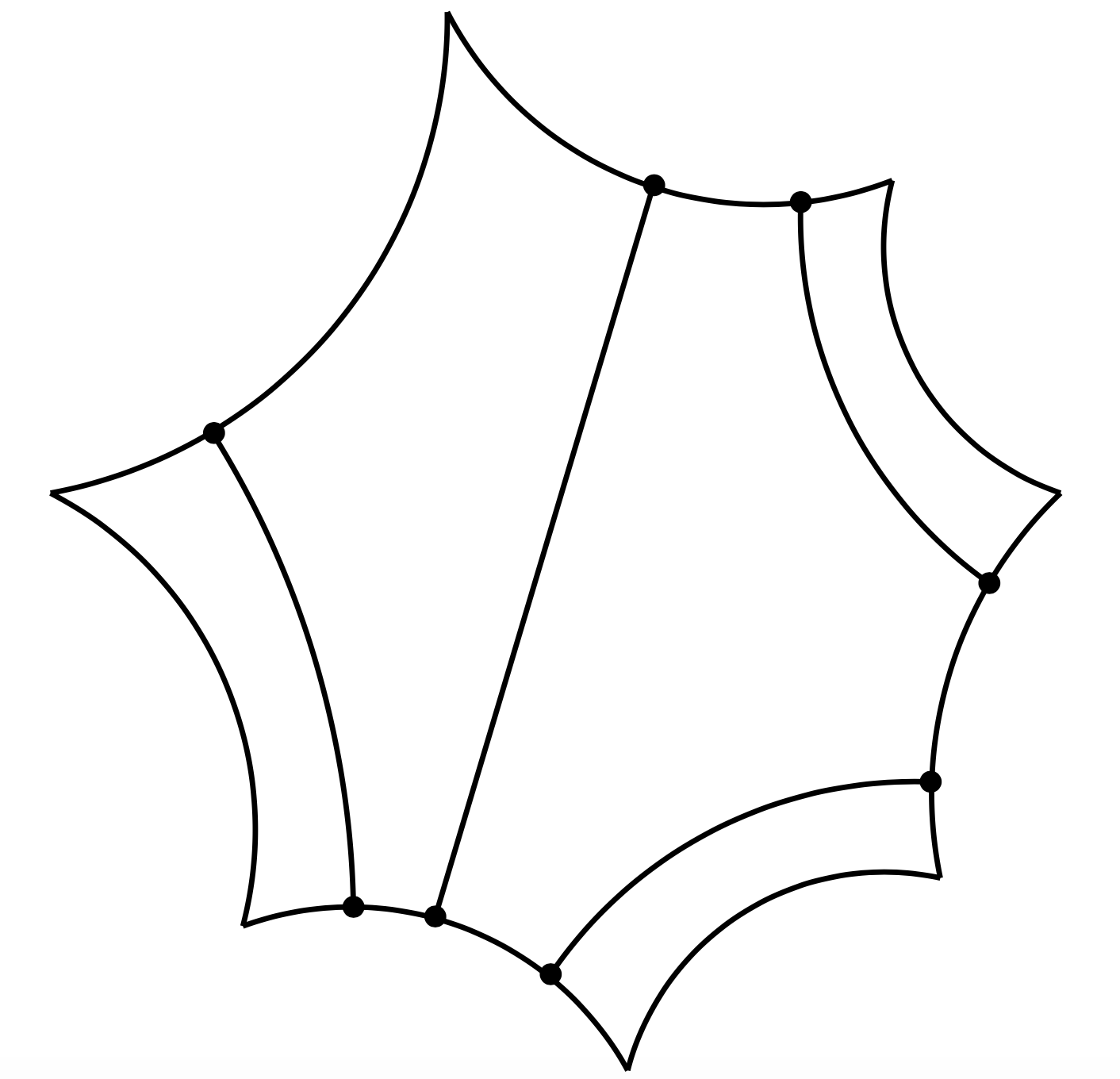}};
\node at (-4,-1.8) {\small $\ell_1$};
\node at (-.4,2) {\small $\ell_2$};
\node at (5.5,3) {\small $\ell_3$};
\node at (4,-5) {\small $\ell_4$};
\end{tikzpicture}
\caption{\small Arcs subdividing a $7$--gon with all acute interior angles into one pentagon, one hexagon, and three quadrilaterals. (The dots indicate orthogonal intersections.) The deformations of the polygon, preserving the interior angles are parameterized by the lengths $\ell_1,\ell_2,\ell_3,\ell_4$.} \label{F:F-N polygon}
\end{center}
\end{figure}
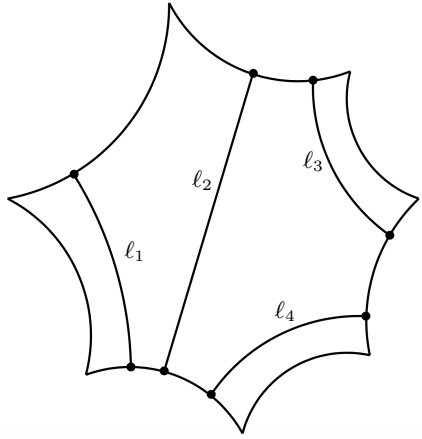

By the \billiardrigidity \hspace{-.095cm}, a polygon $P_1$ can only be billiard flexible if it admits a reflective tiling for which the tile $\tile$ is a good polygon.  Moreover, we can deform $P_1$ to a non-isometric polygon $P_2$ with $\CB(P_1) = \CB(P_2)$ precisely by deforming the tile $\tile$.  The deformation space of $\tile$, preserving the interior angles, has dimension $(n-3)$ (and is a point if $n=3$) by the same argument as the one given in the proof of Corollary~\ref{C:deforming polygons}. This gives us the following.\\

\noindent
{\bf Corollary~\ref{C:billiard flexible}} {\em A hyperbolic polygon $P$ is billiard flexible if and only it is reflectively tiled with a non-triangular tile.} \qed

\subsection{Generalized diagonals}

As discussed in the introduction, a {\em generalized diagonal} in a polygon $P$ is finite billiard trajectory $\gamma \colon [a,b] \to P$ that starts and ends at a vertex of the polygon.  Such a segment has a {\em finite} bounce sequence ${\bf b}(\gamma) = (b_1,\ldots,b_n)$ and we denote the set of all such finite bounce sequences $\CB_\Delta(P)$.\\

\noindent
{\bf Theorem~\ref{T:gen diagonal hyp}}
{\em Given polygons $P_1,P_2 \subset \mathbb H$, we have $\CB_\Delta(P_1) = \CB_\Delta(P_2)$ if and only if one of the two conclusions of the \billiardrigidity holds.}\\

\begin{proof} First, suppose $\CB_\Delta(P_1) = \CB_\Delta(P_2)$ and let $S$ be a common unfolding with metrics $\varphi_1$ and $\varphi_2$, respectively.  The same idea as in the proof of Lemma~\ref{L:billiards endpoints} implies that every $\tphi_1$--saddle connection is uniformly close to some $\tphi_2$--saddle connection (possibly having different initial/terminal endpoints) by considering the associated bounce sequences.  Now observe that any non-singular $\tphi_i$--geodesic is a limit of $\tphi_i$--saddle connections, for $i=1,2$.  Given a $\tphi_1$--geodesic $\eta$ which is a limit of $\tphi_1$--saddle connections, $\eta$ straightens to a basic $\tphi_2$--geodesic, and hence $\CG_{\tphi_1} \subset \CG_{\tphi_2}$.  Interchanging the roles of $\tphi_1$ and $\tphi_2$, we get equality $\CG_{\tphi_1} = \CG_{\tphi_2}$.  At this point, the rest of the proof of the \billiardrigidity can be carried out verbatim since $\CG_{\tphi_1} = \CG_{\tphi_2}$ is the only fact that is used after choosing a common unfolding.

For the reverse implication, we assume either (1) or (2) of the \billiardrigidity holds.  As in the proof of the \billiardrigidity we choose a common unfolding $S$ with associated metrics $\varphi_1$ and $\varphi_2$ and deduce that $\CG_{\varphi_1} = \CG_{\varphi_2}$.  As in that proof, appealing to Lemma~\ref{L:examples C--uber-normalized} we see that the preimage of the even order orbifold points is a concurrency set $C \subset S$ for $(\varphi_1,\varphi_2)$, with $\varphi_2$  $C$--uber-normalized with respect to $\varphi_1$.  By Lemma~\ref{L:relative straightening} any $\tphi_1$--saddle connection straightens to a $\tphi_2$--saddle connection in the complement of $C$.  As above we have a cell structure on $\tS$ whose $2$--cells are copies of the topological polygon $P$ with labeled $1$--cells, and it follows that a $\tphi_1$--saddle connection crosses exactly the same sequence of labeled $1$--cells as the $\tphi_2$--straightening (since the $1$--cells are geodesics with respect to both metrics).  Therefore, every sequence in $\CB_\Delta(P_1)$ is a sequence in $\CB_\Delta(P_2)$, and vice-versa, proving $\CB_\Delta(P_1) = \CB_\Delta(P_2)$, as required.
\end{proof}

\subsection{Connection to prior results and questions} \label{S:UG-connections} 
In \cite{UlGi}, Ullmo and Giannoni consider compact hyperbolic polygons $P \subset \mathbb H$ for which the interior angles are all acute.  The authors provide a nearly complete description of the bounce spectra for such a polygon.  The ideal objective is to describe finite words in the label set $\mathcal A$ that are forbidden in a bounce sequence of a billiard trajectory in $P$, and then characterize bounce sequences precisely as those that do not contain such subwords.  The authors do not quite obtain such a set, but do provide approximations to this.  Their analysis divides up into two cases: billiards in good polygons (which they call {\em tiling billiards}) and the rest. 


In the case that the polygon $P$ is not good, for every $N >0$, Ullmo-Giannoni produce two sets of words of length at most $N$; forbidden words and {\em authorized} words.   The authorized words have the property that any sequence for which all subwords of length at most $N$ are authorized are necessarily bounce sequences of a billiard trajectory in $P$ (thus, forbidden words provide necessary conditions and authorized words provide sufficient conditions).  They also show that as $N$ tends to infinity, the ratio of the number of authorized to non-forbidden words tends to $1$.  It would be interesting to see if, at least for acute polygons, one could push this analysis to give another proof of rigidity for such polygons.

To describe the situation for a good polygon $P$, let $\tfrac{\pi}{k_j}$ denote the interior angle at the vertex of $P$ between sides $j$ and $j+1$ (modulo $n$), where $k_j > 2$ is an integer for all $j$.  In this case, Ullmo-Giannoni show that length 2 subwords which simply repeat a label are forbidden as are sequences of the form $j,j+1,j,j+1,\ldots$ or $j+1,j,j+1,j,\ldots$ of length greater than $k_j$ and prove that this {\em almost} characterizes bounce sequences of billiard trajectories.  More precisely, they show that given a sequence that contains no forbidden words as above, after {\em possibly} replacing some of the length $k_j$ subwords of the form $j,j+1,j,j+1,\ldots$ with an equal length word of the form $j+1,j,j+1,j,\ldots$ (or vice versa), for each $j$, then the resulting sequence is the bounce sequence of a billiard trajectory in $P$.  See also Nagar--Singh \cite{NaSi}.

In the case that $P_1$ and $P_2$ are nonisometric, good, acute $n$--gons for which the corresponding interior angles (with respect to fixed labelings) are equal, their results provide an interesting connection to our work.
For concreteness, say the angle between sides $j$ and $j+1$ is $\tfrac{\pi}{k_j}$, as above.  If $P_1$ is billiard rigid (for example, $P_1$ could be as in case (2) of \ref{P:lots of examples}) then $\CB(P_1) \neq \CB(P_2)$.  On the other hand, the results of \cite{UlGi} described in the previous paragraph imply that given any bounce sequence ${\bf b} \in \CB(P_2)$, after possibly replacing some of the subwords of the form $j,j+1,j,j+1,\ldots$ of length $k_j$ by the length $k_j$ word of the form $j+1,j,j+1,j,\ldots$ (or vice versa) for all $j$, one obtains a bounce sequence ${\bf b'} \in \CB(P_1)$.  In particular, the difference between the difference between $\CB(P_1)$ and $\CB(P_2)$ is quite subtle.\\

\noindent
{\bf Questions.} We end with some questions and possible future directions. To start, it would be interesting to investigate further to what extent Ullmo-Giannoni's methods and ours are complementary. In particular:

\begin{itemize}
\item[(Q1)] Can one prove the \billiardrigidity using the techniques of \cite{UlGi} (perhaps restricted to polygons with only acute interior angles)?
\end{itemize} 

In an earlier paper Ullmo-Giannoni \cite{UlGi2} characterized the bounce spectra for ideal hyperbolic polygons; that is, polygons all of whose vertices are on $\partial\mathbb{H}$ (see also Nagar--Singh \cite{NaSi}).  In this setting, the authors find that bounce sequences of an ideal $n$--gon are precisely the biinfinite sequences never repeating a label and not containing an infinite repeating string of the form
\[  \overline{i,i+1} = i, i+1,i,i+1,i,i+1,\ldots\]
either forward or backward, with indices taken modulo $n$.  In particular, any two ideal $n$--gons, $P_1$ and $P_2$ have the same bounce spectra, $\CB(P_1) = \CB(P_2)$.  This fits into the \billiardrigidity if one extends the notion of reflective tiling to allow vertices at infinity (where there are no longer any conditions on interior angles at such points).  Indeed, then $P_1$ and $P_2$ are reflectively tiled with a single tile (themselves).  Observe that reflections in the sides of an ideal polygon $P$ generate a reflective tiling of $\mathbb H$ by copies of $P$.

Since ideal polygons are only a special case of noncompact polygons, one could further ask what happens for noncompact polygons in general (some special cases are considered in Nagar--Singh \cite{NaSi}).

\begin{itemize}
\item[(Q2)] Is the \billiardrigidity true if one allows for finite area noncompact hyperbolic polygons? How about infinite area polygons?  \end{itemize} 

Note that noncompact polygons could also have infinitely many sides (even in the finite area, convex case), and so some care should be taken in clearly defining what is meant by a cyclic labeling of the sides.  This question also makes sense in the case of Euclidean polygons of infinite area.

\begin{itemize}
\item[(Q3)] Is there some verson of the \billiardrigidity theorem for infinite area Euclidean polygons?
\end{itemize}

Another natural question is whether a version of \billiardrigidity holds in higher dimensions.  Coding billiards in three-dimensional hyperbolic polyhedra was considered in \cite{Singh}.

\begin{itemize}
\item[(Q4)] Is there a version of the \billiardrigidity for billiards in Euclidean or hyperbolic $n$--space, $n \geq 3$?
\end{itemize}

Returning to dimension 2, we note that in order to answer the special case of (Q2) using our methods, one would first need a more general version of the \currentsupport\!\!.  Of course, we could also formulate this question for infinite area hyperbolic or Euclidean cone surfaces (c.f.~\cite{Hooper-inf}), but we stick to the following concrete instance.

\begin{itemize}
\item[(Q5)] Is the \currentsupport\!\!\! true for noncompact, finite area, complete, negatively curved, hyperbolic cone surfaces? 
\end{itemize} 

The universal cover has a circle at infinity in this case, though now the action has parabolic fixed points on the boundary associated to the ``cusps" of $S$.  We note that even under the finite area assumption, one could still have infinitely many cone points, so the cusps here are not necessarily standard, and there are various technical issues one must consider.  For such a metric, we would take $\CG_{\tphi}$ to be the closure of the space of endpoints of nonsingular geodesics, which should coincide with the support of a Liouville current, though again some care must be taken in defining such an object in this setting.

There are other ways in which one could try to generalize the \currentsupport\!. As we mentioned in the introduction, not only are hyperbolic and Euclidean cone metrics determined by their Liouville currents, but so are any (variable) negative or non-positively curved cone metrics, by a result of Constantine \cite{Constantine}. Hence it is natural to ask to what extent they are determined by only their support:

 \begin{itemize}
\item[(Q6)] Is there a version of the \currentsupport for negative or nonpositively curved cone metrics?
\end{itemize} 

We note that the methods used in this paper to prove the \currentsupport theorem used the fact that $\hS$ is locally isometric to $\mathbb{H}$ in a crucial way (and in \cite{oldpaper} the Euclidean setting was essential).  In particular, it seems highly unlikely that (Q6) would follow by direct generalizations of the techniques used in this paper.  The statement itself is likely to be quite different: for example, the second conclusion of the \currentsupport would certainly have to be relaxed to allow lifting deformations via branched covers of negatively curved orbifolds, but an even more general alternative seems likely.  There is also nothing stopping one from asking similar questions in the case of spherical surfaces/polygons, but it seems likely the techniques will stray even further from those considered here.

  \bibliographystyle{alpha}
  \bibliography{main}

\begin{thebibliography}{DELS21}

\bibitem[Bea83]{Beardon}
Alan~F. Beardon.
\newblock {\em The geometry of discrete groups}, volume~91 of {\em Graduate
  Texts in Mathematics}.
\newblock Springer-Verlag, New York, 1983.

\bibitem[BH99]{BridHaef}
M.R. Bridson and A.~Haefliger.
\newblock {\em Metric spaces of non-positive curvature}, volume 319 of {\em
  Grundlehren der Mathematischen Wissenschaften [Fundamental Principles of
  Mathematical Sciences]}.
\newblock Springer-Verlag, Berlin, 1999.

\bibitem[BL18]{BL}
Anja Bankovic and Christopher~J. Leininger.
\newblock Marked-length-spectral rigidity for flat metrics.
\newblock {\em Trans. Amer. Math. Soc.}, 370(3):1867--1884, 2018.

\bibitem[BTV16]{BTV-Delaunay}
Mikhail Bogdanov, Monique Teillaud, and Gert Vegter.
\newblock Delaunay triangulations on orientable surfaces of low genus.
\newblock In {\em 32nd {I}nternational {S}ymposium on {C}omputational
  {G}eometry}, volume~51 of {\em LIPIcs. Leibniz Int. Proc. Inform.}, pages
  Art. 20, 17. Schloss Dagstuhl. Leibniz-Zent. Inform., Wadern, 2016.

\bibitem[CFF92]{CFF}
C.~Croke, A.~Fathi, and J.~Feldman.
\newblock The marked length-spectrum of a surface of nonpositive curvature.
\newblock {\em Topology}, 31(4):847--855, 1992.

\bibitem[Con18]{Constantine}
David Constantine.
\newblock Marked length spectrum rigidity in non-positive curvature with
  singularities.
\newblock {\em Indiana Univ. Math. J.}, 67(6):2337--2361, 2018.

\bibitem[Cro90]{croke}
Christopher~B. Croke.
\newblock Rigidity for surfaces of nonpositive curvature.
\newblock {\em Comment. Math. Helv.}, 65(1):150--169, 1990.

\bibitem[DE86]{DouadyEarle}
Adrien Douady and Clifford~J. Earle.
\newblock Conformally natural extension of homeomorphisms of the circle.
\newblock {\em Acta Math.}, 157(1-2):23--48, 1986.

\bibitem[DELS21]{oldpaper}
Moon Duchin, Viveka Erlandsson, Christopher~J. Leininger, and Chandrika
  Sadanand.
\newblock You can hear the shape of a billiard table: symbolic dynamics and
  rigidity for flat surfaces.
\newblock {\em Comment. Math. Helv.}, 96(3):421--463, 2021.

\bibitem[DLR10]{DLR}
M.~Duchin, C.~J. Leininger, and K.~Rafi.
\newblock Length spectra and degeneration of flat metrics.
\newblock {\em Invent. Math.}, 182(2):231--277, 2010.

\bibitem[FM10]{farb:MCG}
B.~Farb and D.~Margalit.
\newblock {\em A primer on mapping class groups}.
\newblock Princeton Univ. Press, Princeton, N.J., 2010.

\bibitem[Fra12]{Frazier:LS}
Jeffrey Frazier.
\newblock {\em Length spectral rigidity of non-positively curved surfaces}.
\newblock ProQuest LLC, Ann Arbor, MI, 2012.
\newblock Thesis (Ph.D.)--University of Maryland, College Park.

\bibitem[GU90]{UlGi2}
Marie-Joya Giannoni and Denis Ullmo.
\newblock Coding chaotic billiards. {I}. {N}oncompact billiards on a negative
  curvature manifold.
\newblock {\em Phys. D}, 41(3):371--390, 1990.

\bibitem[Hat02]{Hatcher}
Allen Hatcher.
\newblock {\em Algebraic topology}.
\newblock Cambridge University Press, Cambridge, 2002.

\bibitem[Hoo14]{Hooper-inf}
W.~Patrick Hooper.
\newblock An infinite surface with the lattice property {I}: {V}eech groups and
  coding geodesics.
\newblock {\em Trans. Amer. Math. Soc.}, 366(5):2625--2649, 2014.

\bibitem[HP97]{HerPaul}
Sa'ar Hersonsky and Fr{\'e}d{\'e}ric Paulin.
\newblock On the rigidity of discrete isometry groups of negatively curved
  spaces.
\newblock {\em Comment. Math. Helv.}, 72(3):349--388, 1997.

\bibitem[NS21]{NaSi}
Anima Nagar and Pradeep Singh.
\newblock Finiteness in polygonal billiards on hyperbolic plane.
\newblock {\em Topol. Methods Nonlinear Anal.}, 58(2):481--520, 2021.

\bibitem[Ota90]{Otal}
Jean-Pierre Otal.
\newblock Le spectre marqu\'e des longueurs des surfaces \`a courbure
  n\'egative.
\newblock {\em Ann. of Math. (2)}, 131(1):151--162, 1990.

\bibitem[Pur76]{Purzitsky}
Norman Purzitsky.
\newblock All two-generator {F}uchsian groups.
\newblock {\em Math. Z.}, 147(1):87--92, 1976.

\bibitem[Rat94]{Ratcliffe}
John~G. Ratcliffe.
\newblock {\em Foundations of hyperbolic manifolds}, volume 149 of {\em
  Graduate Texts in Mathematics}.
\newblock Springer-Verlag, New York, 1994.

\bibitem[Sin20]{Singh}
Pradeep Singh.
\newblock Coding of billiards in hyperbolic 3-space.
\newblock preprint, {\tt arXiv:2009.14427}, 2020.

\bibitem[Thu97]{Thurston:book}
William~P. Thurston.
\newblock {\em Three-dimensional geometry and topology. {V}ol. 1}, volume~35 of
  {\em Princeton Mathematical Series}.
\newblock Princeton University Press, Princeton, NJ, 1997.
\newblock Edited by Silvio Levy.

\bibitem[UG95]{UlGi}
Denis Ullmo and Marie-Joya Giannoni.
\newblock Coding chaotic billiards. {II}. {C}ompact billiards defined on the
  pseudosphere.
\newblock {\em Phys. D}, 84(3-4):329--356, 1995.

\bibitem[ZK75]{ZeKa}
A.~N. Zemljakov and A.~B. Katok.
\newblock Topological transitivity of billiards in polygons.
\newblock {\em Mat. Zametki}, 18(2):291--300, 1975.

\end{thebibliography}

\Addresses

\end{document}